\documentclass[10pt]{siamltex} 
\usepackage{amsmath,amssymb,amsfonts,graphicx,epsfig,color,url}
\usepackage{latexsym} 

\usepackage{color}
\definecolor{colorJRblue}{rgb}{0.,0.,1.}%67}
\definecolor{colorJRred}{rgb}{1.,0.,0.}%67}

\newtheorem{hypothesis}{Hypothesis}
\newtheorem{remark}{Remark}
\newenvironment{notation}{\vspace{2mm}\textsc{Notation.\hspace{3mm}}\it}{\vspace{2mm}}

\renewcommand{\Re}{\mathrm{Re}}
\renewcommand{\Im}{\mathrm{Im}}

\newcommand{\Cos}{\mathrm{Cos}}

\newcommand{\rmred}{\mathrm{red}}
\newcommand{\rmf}{\mathrm{f}}

\newcommand{\rmt}{\mathrm{t}}
\newcommand{\actI}{I}
\newcommand{\virI}{J}

\newcommand{\rmo}{\mathrm{o}}
\newcommand{\rmb}{\mathrm{b}}
\newcommand{\rmL}{\mathrm{L}}
\newcommand{\rmE}{\mathrm{E}}
\newcommand{\rmP}{\mathrm{P}}
\newcommand{\om}{\omega}
\newcommand{\eps}{\varepsilon}
\newcommand{\couplin}{\mathfrak{c}}
\newcommand{\coupnl}{\mathfrak{d}}
\newcommand{\rmi}{\mathrm{i}}
\newcommand{\calS}{\mathcal{S}}
\newcommand{\calU}{\mathcal{U}}
\newcommand{\eqa}{\zeta}%{{\mathfrak{a}}}
\newcommand{\eqb}{\xi}%{{\mathfrak{b}}}
\newcommand{\calY}{\mathcal{Y}}

\newcommand{\calR}{\mathcal{R}}

\newcommand{\Proj}{\mathrm{Proj}}

\newcommand{\calB}{\mathcal{B}}
\newcommand{\calC}{\mathcal{C}}
\newcommand{\calG}{\mathcal{G}}

\newcommand{\calL}{\mathcal{L}}
\newcommand{\calM}{\mathcal{M}}

\newcommand{\calO}{\mathcal{O}}
\newcommand{\calT}{\mathcal{T}}
\newcommand{\calV}{\mathcal{V}}
\newcommand{\calW}{\mathcal{W}}

\newcommand{\dblx}{{x \choose \hat x}}

\newcommand{\dy}{\mathrm{d}y}

\newcommand{\Id}{\mathrm{Id}}

\newcommand{\N}{\mathbb{N}}

\newcommand{\rmc}{\mathrm{c}}
\newcommand{\rmd}{\mathrm{d}}
\newcommand{\rme}{\mathrm{e}}
\newcommand{\rms}{\mathrm{s}}

\newcommand{\rmu}{\mathrm{u}}

\newcommand{\R}{\mathbb{R}}
\newcommand{\Rg}{\mathrm{Rg}}

\newcommand{\spec}{\mathrm{spec}}

\def\enum{\setcounter{Lcount}{1}
\begin{list}{\arabic{Lcount}.}{\usecounter{Lcount}\leftmargin=1.2em}}

\numberwithin{remark}{section}
\newcounter{Lcount}

\begin{document}

\title{Lyapunov-Schmidt Reduction for Unfolding Heteroclinic Networks of Equilibria and\\ Periodic Orbits with Tangencies}

\author{Jens D.M. Rademacher\thanks{Centrum Wiskunde en Informatica, Science Park 123, 1098 XG Amsterdam, the Netherlands, rademach@cwi.nl [\today]
}}

\maketitle

\begin{abstract}
This article concerns arbitrary finite heteroclinic networks in any
phase space dimension whose vertices can be a random mixture of
equilibria and periodic orbits. In addition, tangencies in the
intersection of un/stable manifolds are allowed. The main result is a
reduction to algebraic equations of the problem to find all solutions
that are close to the heteroclinic network for all time, and their
parameter values. A leading order expansion is given in terms of the
time spent near vertices and, if applicable, the location on the
non-trivial tangent directions. The only difference between a periodic
orbit and an equilibrium is that the time parameter is discrete for a
periodic orbit. The essential assumptions are hyperbolicity of the
vertices and transversality of parameters. Using the result,
conjugacy to shift dynamics for a generic homoclinic orbit to a
periodic orbit is proven. Finally, equilibrium-to-periodic orbit
heteroclinic cycles of various types are considered.
\end{abstract}

%%%%%%%%%%%%%%%%%%%%%%%%%%%%%%%%%%%%%%%%%%%%%%%%%%
\section{Introduction}\label{s:intro}
%%%%%%%%%%%%%%%%%%%%%%%%%%%%%%%%%%%%%%%%%%%%%%%%%%

Heteroclinic networks in ordinary differential equations organise the
nearby dynamics in phase space for closeby parameters. They thus act
as organising centres and explain qualitative properties of solutions,
and predict variations upon parameter changes. This makes heteroclinic
networks a valuable object in studies of models for applications. When
all vertices in the network are equilibria much about such
bifurcations is known. Recently, heteroclinic networks whose vertices
can also be periodic orbits have found increasing attention. 

This article concerns the unfolding of finite heteroclinic networks consisting of hyperbolic equilibria and periodic orbits in an ordinary differential equation
\begin{equation}\label{e:ode}
\frac{\rmd}{\rmd x}u(x) = f(u(x);\mu),
\end{equation}
with $x\in \R$, $u(x)\in \R^n$ and parameter $\mu\in \R^d$ for arbitrary $n$ and sufficiently large $d$.

Any solution that remains close to the heteroclinic network of \eqref{e:ode} assumed at $\mu=0$ for all time can be cast in terms of its itinerary in the heteroclinic network. See Figure~\ref{f:2-hom} for a simple example. For any  itinerary, bifurcation equations will be derived for the locus of parameters of all corresponding solutions to \eqref{e:ode}. 
The idea to formulate the unfolding in this way is borrowed from previous studies of heteroclinic chains of \emph{equilibria}  \cite{Lin90,bjorndiss}. 

Bifurcation studies from heteroclinic chains with equilibria mainly concerned homoclinic orbits and generated a huge amount of literature, see, e.g., \cite{chow,denghom,homkraus,knobtang,Lin90,bjorndiss,Shilnikov,vanfie} just to name a few, and heteroclinic loops between two equilibria, see, e.g., \cite{bykov,chow1,deng,glenSpa,kokubu,Lin90,Shilnikov,zhu1,zhu}. Heteroclinic cycles with periodic orbits have found increasing attention recently, see, e.g., \cite{becks,krock,KrausOlde,KrausRiess,myhombif,Riess,sieber,simpson}. 

The main new contributions of the present work are rigorous results allowing for periodic orbits in general heteroclinic networks, for tangent heteroclinic connections, and to formulate the bifurcation equations in a general form that can be used as the basic building block for a specific study.
It is hoped that this makes the results useful for readers with applications to specific cases in mind.

\begin{figure}
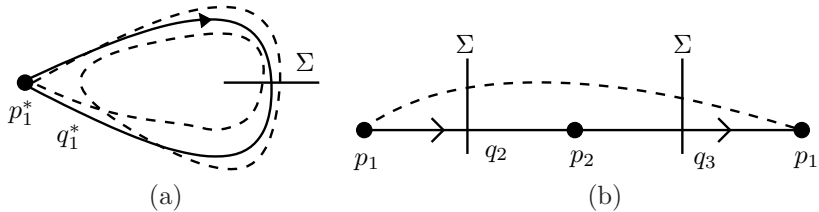

\begin{center}
\begin{tabular}{cc}
\input{2-hom.pstex_t}
& \input{2-hom-b.pstex_t}\\
(a) & (b)
\end{tabular}
\caption{Sketch of an itinerary for a 2-homoclinic orbit. (a) The homoclinic orbit $q_1^*$ (solid) and its asymptotic state $p_1^*$ at $\mu=0$, and a 2-homoclinic orbit (dashed). The Poincar\'e-section $\Sigma$ is just for orientation here. (b) Schematic plot of the itinerary (solid) for a 2-homoclinic orbit (dashed) where $q_2=q_3=q_1^*$, $p_1=p_2=p_1^*$.  }
\label{f:2-hom}
\end{center}
\end{figure}

A large field of applications are travelling waves in evolutionary partial differential equations in one space dimension whose profiles solve an equation of the form \eqref{e:ode}, but also for instance Laser models are often reduced to this form. There is a very large amount of such analytic and numerical studies in the literature, e.g., \cite{conley,KrausOlde,bar,langer, knobport,sanTAMS,ssglue,sieber,simpson} to hint at some.

In the applied literature such bifurcation equations are frequently derived formally by a geometric decomposition in terms of local and global maps, e.g., \cite{krock,glenSpa,knobport}. The justification in particular of the local map is an issue and, if linear, requires non-resonance conditions on eigenvalues or else dimension dependent normal form computations. Other issues are the form of parameter dependence and the persistence of solutions upon inclusion of the higher order terms of the original vector field. These problems do not arise in the approach taken here, and the results can provide a rigourous foundation of formal reductions.

The reduction to bifurcation equations in this paper is motivated by the so-called `Lin-method' described in \cite{Lin90}, which is a Lyapunov-Schmidt reduction for boundary value problems of the itinerary. This method has been used and modified in a number of ways and contexts for equilibria, e.g., \cite{knobdiscr,knobtang,vanfie}. Tangent intersection of stable and unstable manifolds have been considered mostly for homoclinic orbits, i.e., homoclinic tangencies, a paradigm of chaotic dynamics; Lin's method in this context has been used in \cite{knobtang}. Periodic orbits introduce technical complications and for Lin's method these have been overcome in \cite{myhombif} and, using Poincar\'e-maps, in \cite{Riess}. Transversality studies with respect to parameters in related cases were done in \cite{HaleLin}. An ergodic theory point of view is taken in \cite{Bam94,Diaz,Labarca, LaSa97,MorPac98, PaRo93}, and further papers by these authors, looking for instance at properties of non-wandering sets. More recently \cite{becks} treated periodic orbits in a very promising way using Fenichel-coordinates. 

Here \cite{myhombif} is used as a starting point, and equilibria or periodic orbits as vertices are treated in an essentially unified manner. Symmetries or conserved quantities are not used, but a generic setting is assumed. In contrast to \cite{myhombif}, winding numbers of heteroclinic sets are not considered, and the underlying heteroclinic network is held fixed. Together with \cite{myhombif} this exposition is self-contained, but somewhat technical, and parts of \cite{myhombif} have to be repeated and improved in order to track higher order terms in this extended setting. The precise statement of the main result Theorem~\ref{t:2} can only be given rather late after a number of preparatory steps, notation and definitions. In particular, this includes \S\ref{s:coord} where we obtain suitable coordinates near the vertices. We next describe the main result, and refer to \S\ref{s:bifana} for sample applications.

%%%%%%%%%%%%%%%%%%%%%%%%
\subsection{Description of the main result}\label{s:maindescr}

For a chosen itinerary the method is a Lyapunov-Schmidt reduction which yields algebraic equations that relate system parameters $\mu$ to certain geometric characteristics $L_j$, $v_j$ at each heteroclinic connection $q_j$ that the solution follows, perhaps repeatedly, and which connects vertex $p_{j-1}$ to $p_j$.  The time spent between Poincar\'e sections $\Sigma_{j-1}$ at $q_{j-1}(0)$, and $\Sigma_j$ at $q_j(0)$ is $2L_j$ for $L_j\in [L^*,\infty)$ if $p_j$ is an equilibrium. If $p_j$ is a periodic orbit, this time is in general only approximately $2L_j$ since we normalise $L_j\in \{\ell T_j/2 : \ell\in \N,\, \ell T_j/2\geq L^*\}$ for the minimal period $T_j$ of vertex $j$.  For tangent heteroclinic connections or more than one-dimensional heteroclinic sets, un/stable manifolds of $p_{j-1}$ and $p_j$ have a more than one-dimensional common tangent space at $q_j(0)$, and $v_j$ are the coordinates on that space, except the flow direction. The location of $L_j, v_j$ in the itinerary is illustrated in Figure~\ref{f:ext-chain}. 

The system parameters $\mu\in \R^d$ can generically be assumed to unfold each heteroclinic connection by a separate set of parameters $\mu_j^*\in\R^{d_j}$ which, however, must coincide at repeated connections. Here $d_j$ is the codimension of the $j$th connection and $d=\sum_j d_j$  \emph{without} repeating the same connection in the sum. The geometric characteristics couple these parameter sets, but to leading order only to the nearest neighbours $j\pm1$. If $d_j=0$ for all $j$, then all heteroclinic connections are transverse, and the result proves the existence of solutions for any itinerary, and an expansion for the coordinates in the Poincar\'e-sections. 
Otherwise, an expansion of $\mu_j^*$ in terms of $v_r$, $L_r$ is provided as described next. 

Let $\kappa_j^{\rmu/\rms}$ and $\sigma_j^{\rmu/\rms}$ be the real and imaginary parts of the leading un/stable eigenvalue or Floquet exponent at vertex $j$. For $\gamma\in\R^r$ let $\Cos(\gamma) = (\cos(\gamma_1),\ldots,\cos(\gamma_r))$. For all $j$ with $d_j\geq 1$ and for sufficiently large $\min_r\{L_r\}$ and small $\sup_r\{|v_r|\}$, there exist $\beta_j, \gamma_j \in\R^{d_j}$, and linear maps $\beta_j', \gamma_j', \beta_j'', \gamma_j'', \eqa_j, \eqb_j, \eqa_j', \eqb_j', \eqa_j'', \eqb_j''$, as well as quadratic maps $\calT_j$ (zero if $t_j=0$), so that a solution follows the chosen itinerary, if and only if
\begin{equation}\label{e:intro}
\begin{array}{rl}
\mu_j^* &=  \calT_j(v_j) +
\rme^{-2\kappa_j^\rmu L_j}\Cos(2\sigma_j^\rmu L_j + \beta_j^*(v_{j+1},L_{j+1})) \eqa_j^*(v_{j+1}, L_{j+1}) \\ %\label{e:expand}\\
&+ \rme^{-2\kappa_{j-1}^\rms L_{j-1}} \Cos(2\sigma_{j-1}^\rms L_{j-1} + \gamma_j^*(v_{j-1}, L_{j-2})\eqb_j^*(v_{j-1}, L_{j-2}) + \calR_j.
\end{array}
\end{equation}
Here $\mu_j^* = \mu_{j'}^*$ whenever the $j$ and $j'$ in the itinerary correspond to the same actual heteroclinic connection. 
The coupling to the nearest neighbours is given by
\begin{equation}\label{e:intro2}
\begin{array}{rl}
\beta_j^*(v,L) &= \beta_j + \beta_j'v +\beta_j'' B_{j+1}^\rmu(L), \\
\gamma_j^*(v,L) &= \gamma_j + \gamma_j'v +\gamma_j'' B_{j-2}^\rms(L), \\
\eqa_j^*(v,L) &= \eqa_j + \eqa_j'v +\eqa_j'' B_{j+1}^\rmu(L), \\
\eqb_j^*(v, L) &= \eqb_j +\eqb_j' v + \eqb_j'' B_{j-2}^\rms(L), \\
B_r^{\rms/\rmu}(L) &= \rme^{-2\kappa_r^{\rms/\rmu}L}\Cos(2\sigma_r^{\rms/\rmu} L + \beta^{\rms/\rmu}_r)\eqa^{\rms/\rmu}_r,
\end{array}
\end{equation}
where $\beta_r^{\rms/\rmu}\in\R^{d_r}$ and $\eqa_r^{\rms/\rmu}$ are linear maps. 

Negative Floquet multipliers, i.e., negative eigenvalues of the period map, have Floquet exponent with imaginary part $\pi/T_j$. Since $L = \ell T_j/2$, $\ell\in\N$, the argument in the cosine terms is $\pi\ell$ which generates an oscillating sign as $\ell$ is incremented. 

Note that if the itinerary has repetitions, then $v_j, L_j$ have to satisfy solvability conditions from repeating the corresponding equations. Each repetition yields new parameters $v_j$$, L_j$, but all other quantities in \eqref{e:intro} are the same if the underlying heteroclinic connections is the same.

\begin{figure}%[htbp]
\begin{center}
\input{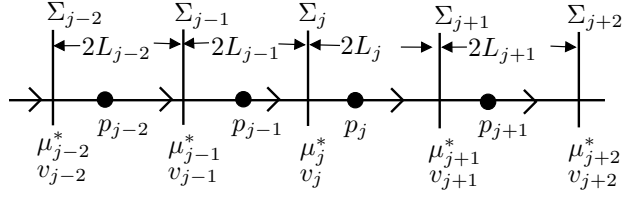}
\caption{Schematic illustration of an itinerary near index $j$, and the location of $\mu_j^*, v_j, L_j$. The arrows indicate the flow direction, connections can be copies of the same actual heteroclinic connection.}
\label{f:ext-chain}
\end{center}
\end{figure}

The significance of \eqref{e:intro} lies in the order of the remainder term $\calR_j$, which, for certain $\delta_j^{\rmu/\rms}>0$, and arbitrary $\delta>0$, is given by
\begin{equation*}%\label{e:intro3}
\begin{array}{rl}
&|v_j|^3 + \rme^{-2\kappa_j^\rmu L_j}\left(\rme^{-\delta_j^\rmu L_j} + |v_{j+1}|(\rme^{(\delta-\kappa^\rmu_{j+1})L_{j+1}} + |v_{j+1}|) + \rme^{(\delta-3\kappa_{j+1}^\rmu) L_{j+1}}\right)\\
&+ \rme^{-2\kappa_{j-1}^\rms L_{j-1}}\left( \rme^{-\delta_{j-1}^\rms L_{j-1}}+ |v_{j-1}|(\rme^{(\delta-\kappa^\rms_{j-2})L_{j-2}} + |v_{j-1}|) + \rme^{(\delta-3\kappa_{j-2}^\rms) L_{j-2}}\right).
\end{array}
\end{equation*}
In particular, $\calR_j$ is higher order with respect to at least one of the cosine terms if $L_j\sim L_r$, $r=j-1,j-2,j+1$. 

The application of the above result to a Shil'nikov-homoclinic orbit to an equilibrium yields the same bifurcation equations as in \cite{Lin90}, and as in that paper, many of the seminal results by Shil'nikov \cite{Shilnikov} follow from leading order analyses. Note that the resonant case $\kappa_j^\rmu = \kappa_{j-1}^\rms$ can be treated by the above result as well. See \cite{chow} for resonance at homoclinic bifurcations.

\begin{remark}\label{r:lead}
Under the `flip' condition $\mathrm{Rank}(\eqa_j)=0$ or $\mathrm{Rank}(\eqb_j)_j=0$ for the leading eigenvalue,  the next leading terms need to be taken into account for an unfolding. Viewing $\eqa_j$, $\eqb_j$ as parameters, the bifurcation of solutions can be understood from Theorem~\ref{t:2} if $2\kappa_j^\rmu < 2\kappa_{j-1}^\rms + \delta_{j-1}^\rms$ for $\eqa_j=0$ or $2\kappa_{j-1}^\rms < 2\kappa_{j}^\rmu + \delta_{j}^\rmu$ for $\eqb_j=0$. To overcome the barriers involving $\rho_{j-1}^\rms$ and $\rho_j^\rmu$ to the next order eigenvalues requires a more refined setup beyond the scope of this article. See \cite{bjorndiss} and also \cite{homkraus,flip,nii} for such considerations in case of equilibria applied to homoclinic bifurcations.
\end{remark}

\medskip
The main result unifies the treatment of equilibria and periodic orbits as vertices of a network: for the reduced equations the only difference between equilibria and periodic orbits is that $L_j$ is a semi-infinite interval for an equilibrium, but the above defined discrete infinite sequence for a periodic orbit. The discrete sequence essentially counts the number of rotations that the solution makes about the periodic orbit. Note that replacing a periodic orbit by an equilibrium has consequences for the codimensions of heteroclinic connections.

The reduced equations for a specific case can be determined in three steps. First, choose the itinerary of the solution type of interest. Second, determine the codimensions, including tangencies, of all visited heteroclinic connections. Third, copy the equations from Theorem~\ref{t:2} for each element in the itinerary with positive codimension, and remove geometric characteristics that do not occur according to the type of itinerary and tangencies.
In case of repetition in the itinerary, the locus of parameter values for the solutions should be found by analysing the arising algebraic solvability conditions (which can be highly non-trivial). Similarly, in case of tangencies, the locus of turning points or folds can be determined. 

To illustrate this and the applicability of the abstract results, some sample applications for specific heteroclinic networks are presented in \S\ref{s:bifana}. In particular, for a generic homoclinic orbit to a periodic orbit conjugacy to (suspended) shift dynamics is proven. 
In addition, equilibrium-to-periodic orbit heteroclinic cycles of various types are considered, and 2-homoclinic orbits are studied for the first time in this context.

\medskip
Note that the main result separately concerns each solution type as encoded in each itinerary. This is suitable, for instance, when looking for the aforementioned travelling waves. In some cases a whole class of solutions or even the entire invariant set can be characterised directly. However, our results do not provide stability information of the bifurcating solutions or hyperbolicity of invariant sets, or other ergodic properties. See \cite{sanTAMS} for stability results in homoclinic bifurcations using Lin's method (where additionally PDE spectra are considered). As mentioned, the above result does not provide an expansion in general in the case of vanishing leading order terms  (`flip conditions').

\medskip
This paper is organised as follows. Section \ref{s:set} contains details about the setting and some preparatory results. In \S\ref{s:coord} a suitable coordinate system is established for trajectories that pass near an equilibrium or periodic orbit or lie in un/stable manifolds. The main result is formulated and proven in \S\ref{s:bifeq}. Finally, \S\ref{s:bifana} contains sample bifurcation analyses and illustrate how to use the main result. 

\medskip \textbf{Acknowledgement.}
This research has been supported in part by NWO cluster NDNS+. The author thanks Bj\"orn Sandstede, Ale Jan Homburg and Alan Champneys for helpful discussions.

%%%%%%%%%%%%%%%%%%%%%%%%%%%%%%%%%%%%%%%%%%%%%%%%%%
\section{Setting and Preparation}\label{s:set}
%%%%%%%%%%%%%%%%%%%%%%%%%%%%%%%%%%%%%%%%%%%%%%%%%%

The basic assumption is that at $\mu=0$ \eqref{e:ode} possesses a finite heteroclinic network $\calC^*=(\calC_1^*,\calC_2^*)$ with vertices $p_i^*\in\calC_1^*$ being equilibria or periodic orbits, $i\in \actI_1$, and edges $q_i^*\in\calC_2^*$, $i\in \actI_2$, being heteroclinic connections. Rather than unfolding $\calC$ as a whole, we consider the following paths within $\calC$ separately. Here we set $\virI^\rmo:= \virI\setminus\min\virI$ (where $\min\virI=\emptyset$ if there is no minimum).

\begin{definition}
A possibly infinite set $\calC= \left( (p_j)_{j\in\virI},(q_j)_{j\in\virI^\rmo} \right)$ with  $p_j\in\calC^*_1$, $j\in\virI$ and $q_j\in\calC^*_2$, $j\in\virI^\rmo$, 
is called \emph{itinerary} if for all $j\in \virI^\rmo$ the edge $q_j$ is a heteroclinic connection from $p_{j-1}$ to $p_j$  (or homoclinic if $p_{j-1}=p_j$) and either $\virI=\{-\infty, \ldots 0,1\}$, $\virI=\{1,2,\ldots\infty\}$, $\virI=\mathbb{Z}$ or $\virI=\{1,\ldots,m\}$ for some $m>1$.

For ease of notation we say that a sequence $y_j$ of `objects' (numbers, vectors or maps given in the context) with $j\in\virI^\rmo$ has \emph{reducible indexing} (with respect to $\calC$) if $y_j=y_{j'}$ whenever $q_j=q_{j'}$ for $j,j'\in\virI^\rmo$.
\end{definition}

Note that an itinerary can cycle arbitrarily and perhaps infinitely long within the heteroclinic network viewed as a directed graph, and the labelling can differ from that in $\calC^*$. Any itinerary has a (possibly non-unique) reduced index set $\virI_\rmred\subset \virI$ so that $q_j\neq q_{j'}$ as well as $p_j\neq p_{j'}$ for $j\neq j'$, with $j,j'\in \virI_\rmred^\rmo:=\virI_\rmred\cap \virI^\rmo$, and $j,j'\in \virI_\rmred$, respectively. Associated to this is $\calC_\rmred =  \left( (p_j)_{j\in\virI_\rmred},(q_j)_{j\in\virI_\rmred^\rmo} \right) \subset \calC$,
which may not itself be an itinerary (though it contains one).

Let $\virI_\rmE\subset \virI$ be the index set of all equilibria $p_j$ in $\calC$ and $\virI_\rmP = \virI\setminus\virI_\rmE$ that of all periodic orbits. We set $\virI_\rmE^\rmo:=\virI_\rmE\cap \virI^\rmo$ and $\virI_\rmP^\rmo:=\virI_\rmP\cap \virI^\rmo$. Finally, $T_j>0$ denotes the minimal period of $p_j$ for $j\in \virI_\rmP$, and we set $T_j=0$ for $j\in \virI_\rmE$.

\medskip
In the following, unless noted otherwise, we consider an arbitrary fixed itinerary $\calC$. However, until \S\ref{s:bifeq} only neighbours $q_j$ and $q_{j-1}$ are relevant. 

\medskip
For $j\in \virI$ let $\tilde\Psi_j(x,0) = A_j(x)\rme^{F_jx}$ be the Floquet representation of the evolution of the linearization $\dot v=\partial_u f(p_j(x);0)v$ of \eqref{e:ode} in $p_j$. Here $\dot v=\rmd v/\rmd x$, and the matrices $A_j(x)$ satisfy $A_j(0)=\Id$, $A_j(x+T_j) = A_j(x)$ for $j\in \virI_\rmP$ and $x\in \R$, and $A(x)\equiv\Id$ for $j\in \virI_\rmE$ (in which case $F_j = \partial_u f(p_j;0)$).

The basic assumption about \eqref{e:ode} and the heteroclinic network is

%++++++++++++++++++
\begin{hypothesis}\label{h:basic}
The vector field $f$ in \eqref{e:ode} is of class $C^{k+2}$ for $k\geq 1$ in $u$ and $\mu$. The equilibria or periodic orbits $p_i^*$, $i\in \actI_1$, are hyperbolic at $\mu=0$, i.e., for any $\calC$ the matrices $F_j$ have no eigenvalues on the imaginary axis, except for a simple eigenvalue at the origin (modulo $2\pi\rmi$) if $j\in\virI_\rmP$.
\end{hypothesis}

Here a simple eigenvalue has algebraic and geometric multiplicity one. Hypothesis~\ref{h:basic} implies that the spectrum $\spec(F_j)$ of $F_j$ has the un/stable gaps
\begin{eqnarray*}
\kappa^\rmu_j &:=& \min\left\{\Re(\nu) : 
\nu\in \spec(F_j), \Re(\nu)> 0\right\} > 0\\
\kappa^\rms_j &:=& -\max\left\{\Re(\nu) : 
\nu\in \spec(F_j), \Re(\nu)< 0\right\} > 0.
\end{eqnarray*}
Since $\calC^*_1$ is finite, the gaps $\kappa_j^{\rms/\rmu}$ are uniformly bounded from below in $j\in\virI$. For convenience we choose arbitrary $\kappa_j>0$, $j\in \virI$ with reducible indexing, such that $\kappa_j < \min\{\kappa_j^\rms, \kappa_j^\rmu\}$.
We also need the gap to the next leading eigenvalues/Floquet exponents. Let $\nu_r$, $r=1,\ldots,n$ be the eigenvalues of $F_j$ and define
\begin{eqnarray*}
\rho_j^\rms &:=& \min\{|\Re(\nu_r)|-\kappa_j^\rms : \Re(\nu_r)<0, \Re(\nu_r)\neq \kappa_j^\rms, r=1\ldots,n\},\\
\rho_j^\rmu &:=& \min\{\Re(\nu_r)-\kappa_j^\rmu : \Re(\nu_r)>0, \Re(\nu_r)\neq \kappa_j^\rmu, r=1\ldots,n\}.
\end{eqnarray*}

Leading stable eigenvalues of a matrix are those with the largest strictly negative real part, and leading unstable those with the smallest strictly positive real part. For the main result, we will assume that these leading eigenvalues are simple as expressed in the following hypotheses. 

\begin{hypothesis}\label{h:simplestable}
Consider the leading stable eigenvalues or Floquet exponents at $p_j$. Assume that this is either a simple real eigenvalue $\nu_j$ or a simple complex conjugate pair $\nu_j$, $\bar\nu_j$ with $\Im(\nu_j)\neq 0$.
\end{hypothesis}

\begin{hypothesis}\label{h:simpleunstable}
Consider the leading unstable eigenvalues or Floquet exponents at $p_j$. Assume that this is either a simple real eigenvalue $\nu_j$ or a simple complex conjugate pair $\nu_j$, $\bar\nu_j$ with $\Im(\nu_j)\neq 0$.
\end{hypothesis}

To emphasise where these hypotheses enter we will not assume them globally, which has the effect that a priori exponential rates for estimates are not $\kappa_j^{\rms/\rmu}$, but $\kappa_j^{\rms/\rmu}-\delta_j$ for an arbitrary $\delta_j>0$ due to possible secular growth. In the following $\delta_j$ denotes a priori an arbitrarily small positive number, which may vanish under Hypothesis~\ref{h:simplestable}, \ref{h:simpleunstable}.

Hence, for suitable $x_j\in[0,T_j)$ as well as asymptotic phases $\alpha_j\in [0,T_j)$ of $q_j$ with respect to $p_j$ we obtain the estimates (see, e.g., \cite{CodLev})
\begin{eqnarray}
|q_j(x) - p_j(x+\alpha_j)| &\leq& C \rme^{(\delta_j-\kappa^\rms_j) x}\;,\, x\geq 0\label{e:asyphase}\\
|q_{j+1}(x+x_j) - p_j(x+\alpha_j)| &\leq& C \rme^{(\delta_j-\kappa^\rmu_j) |x|}\;,\, x\leq 0,\nonumber
\end{eqnarray}
where $C>0$ depends only on $q_j, q_{j+1}$ and $\delta_j$. For $j\in \virI_\rmP$ the requirement \eqref{e:asyphase} of equal asymptotic phase for $q_j(0)$ and $q_{j+1}(x_j)$ determines $x_j$ up to multiples of $T_j$ and uniquely in $[0,T_j)$. For $j\in \virI_\rmE$ we have $p_j(x)\equiv p_j$ and set $\alpha_j=x_j=0$.

To distinguish in- and outflow at $p_j$ we denote $\hat q_j(x):=q_{j+1}(x+x_j)$ for $j-1\in\virI^\rmo$. 

Let $\Phi_j(x,y)$ denote the evolution of $\dot v=\partial_u f(q_j(x);0)v$ and $\hat\Phi_j(x,y)$ that of $\dot v=\partial_u f(\hat q_j(x);0)v$. Hyperbolicity of $p_j$ gives the following exponential dichotomies for $j\in \virI_\rmE^\rmo$ and trichotomies for $j\in\virI_\rmP^\rmo$ (see, e.g., \cite{myhombif}). 

\begin{notation}
Indices separated by one or more `slashes' as in $\kappa_j^{\rms/\rmu}$ indicate alternative choices for the statement with all these indices chosen equal at a time.
\end{notation}

There exist projections $\Psi_j^{\rms/\rmc/\rmu}(x)$, continuous in $x\geq 0$, and $\hat\Psi_j^{\rms/\rmc/\rmu}(x)$, continuous in $x\leq 0$, such that the following holds. Set $\Phi_j^{\rms/\rmc/\rmu}(x,y) :=  \Phi_j(x,y)\Psi_j^{\rms/\rmc/\rmu}(y)$ and $\hat\Phi_j^{\rms/\rmc/\rmu}(x,y) :=  \hat\Phi_j(x,y)\hat\Psi_j^{\rms/\rmc/\rmu}(y)$, respectively.

\begin{itemize}
\item For $j\in \virI_\rmP^\rmo$ : $\Rg(\Psi_j^\rmc(x)) = \mathrm{span}\left(
\frac{\rmd}{\rmd x}q_j(x) \right)$,
\item For $j\in \virI_\rmE^\rmo$ : $\Psi_j^\rmc\equiv \hat\Psi_j^\rmc\equiv 0$,
\item The projections are complementary: $\Psi_j^\rms + \Psi_j^\rmu + \Psi_j^\rmc \equiv \Id$, $\Psi^\rms(\Psi^\rmu + \Psi_j^\rmc)\equiv 0$, $\Psi^\rmu(\Psi^\rms + \Psi_j^\rmc)\equiv 0$, $\Psi^\rmc(\Psi^\rms + \Psi_j^\rmu)\equiv 0$; analogous for $\hat\Psi_j^{\rms/\rmc/\rmu}$,
\item The spaces $\Rg(\Psi^\rms(x))$ and $\Rg(\hat\Psi^\rmu(x))$ are unique, the spaces $\Rg(\Psi^\rmu(x))$ and $\Rg(\hat\Psi^\rms(x))$ arbitrary complements such that the previous holds,
\item The projections commute with the linear evolution:\\ 
$\Phi_j^{\rms/\rmc/\rmu}(x,y) =  \Psi_j^{\rms/\rmc/\rmu}(x)\Phi_j(x,y)$ and $\hat\Phi_j^{\rms/\rmc/\rmu}(x,y) =  \hat\Psi_j^{\rms/\rmc/\rmu}(x)\hat\Phi_j(x,y)$,
\item They distinguish un/stable and center direction: there is $C>0$ depending on $\delta_j>0$ and $q_j$, $\hat q_j$ such that for all $u\in \R^n$
\begin{eqnarray}
|\Phi_j^\rms(x,y)u| &\leq& C \rme^{(\delta_j-\kappa_j^\rms)|x-y|}|u|\;,\; x\geq y \geq 0,\nonumber\\
|\Phi_j^\rmu(x,y)u| &\leq& C \rme^{(\delta_j-\kappa_j^\rmu)|x-y|}|u|\;,\; y\geq x\geq 0,\nonumber\\
|\Phi_j^\rmc(x,y)u| &\leq& C |u|\;,\; x, y \geq 0,\label{e:trichest}\\
|\hat\Phi_j^\rmu(x,y)u| &\leq& C \rme^{(\delta_j-\kappa_j^\rmu)|x-y|}|u|\;,\; x\leq y \leq 0,\nonumber\\
|\hat\Phi_j^\rms(x,y)u| &\leq& C \rme^{(\delta_j-\kappa_j^\rms)|x-y|}|u|\;,\; y\leq x\leq 0,\nonumber\\
|\hat\Phi_j^\rmc(x,y)u| &\leq& C |u|\;,\; x, y \leq 0.\nonumber
\end{eqnarray}
\end{itemize}

We denote the un/stable and center spaces, respectively, by 
\[
E_j^{\rmu/\rms/\rmc}(x) := \Rg(\Psi_j^{\rmu/\rms/\rmc}(x))\,,\; \hat E_j^{\rmu/\rms/\rmc}(x) := \Rg(\hat\Psi_j^{\rmu/\rms/\rmc}(x)).
\]

\begin{definition}
For a decomposition $E \oplus F = \R^n$ we denote by $\Proj(E,F):\R^n\to \R^n$ the unique projection with kernel $E$ and image $F$. \end{definition}

In order to link the trichotomies of the in- and outflow near $p_j$, we define 
\begin{eqnarray*}
P_j^\rms(L) &:=& \Proj(E_j^\rmu(L) + \hat E_j^\rmc(-L), \hat E_j^\rms(-L))\\
P_j^\rmu(L) &:=& \Proj(\hat E_j^{\rms\rmc}(-L), E_j^\rmu(L))\\
P_j^\rmc(L) &:=& \Proj(E_j^\rmu(L) + \hat E_j^\rms(-L),\hat E_j^\rmc(-L)).
\end{eqnarray*}
We also define the aforementioned sets of travel time parameters
\begin{equation*}
K_j(L) := 
\begin{cases}
[L,\infty) & \mbox{ for }j\in \virI_\rmE\\
\{\ell T_j/2 : \ell T_j/2\geq L, \ell\in\N\}& \mbox{ for }j\in\virI_\rmP.
\end{cases}
\end{equation*}

It is shown in \cite{myhombif}, Lemma 2, that there is $L^0>0$ such that for $L\in K_j(L^0)$ the $P_j^{\rms/\rmc/\rmu}(L)$ are complementary projections, $P_j^\rmc \equiv 0$ for $j\in \virI_\rmE^\rmo$, and the norms $|P_j^{\rms/\rmc/\rmu}(L)|$ are uniform in $L$. 

\medskip
In order to control the leading order terms in the bifurcations equations we make the following change of coordinates locally near all $p_j$, $j\in \virI_\rmred$. In the new `straight' coordinates the un/stable manifolds locally coincide with the un/stable eigenspaces of the linearization in $p_j$, respectively. For periodic orbits the strong un/stable fibers locally coincide with the un/stable eigenspaces. Since these are graphs over the eigenspaces and tangent at the equilibrium or periodic orbit this is straightforward. See e.g., (3.27) in \cite{bjorndiss}. However, as in \cite{bjorndiss} this change of coordinates is an obstacle to apply the method within the class of semilinear parabolic partial differential equations.
However, in \cite{becks} this problem has been circumvented in a way that should apply here as well.

To emphasise the effect of this coordinate change and to make the notation of estimates throughout the text more readable, we define for $j\in \virI$ and any $\delta_j>0$, and $\delta_j=0$ if explicitly mentioned, the terms

\begin{align*}
R_j &:= 
\begin{cases}
\rme^{(\delta_j-\kappa_j)L_j} \mbox{ a priori},\\
\rme^{(\delta_j-\kappa_j^\rmu)L_j}\mbox{ in straight coordinates},
\end{cases}\\
\hat R_j &:= 
\begin{cases}
\rme^{(\delta_j-\kappa_j)L_j} \mbox{ a priori},\\
\rme^{(\delta_j-\kappa_j^\rms)L_j}\mbox{ in straight coordinates},
\end{cases}
\end{align*}
and set $R_j=\hat R_j=0$ for $j\not\in\virI$.

\begin{notation} In the following we use the usual order notation $a=\calO(b)$ if there is a constant $C>0$ such that $|a|\leq C|b|$ for all large or small enough $b$ and norms as given in the context. In terms of $L_j$ this is always as $L_j\to\infty$. In a chain of inequalities for such order computations we allow the constant $C$ to absorb other constant factors and take maximum values of several constants without giving explicit notice.
\end{notation}

\medskip
The next lemma is the basis for estimating error terms in the following sections.

\begin{lemma}\label{l:proj-est} 
There exist $C>0$ and $L^1\geq L^0$ depending only on $q_j, \hat q_j$ and $\delta_j$ such that for $L_, L_j\in K(L^1)$ the following holds for all $j\in \virI^\rmo$.
\begin{enumerate}
\item  
$|P_j^{\rmc\rmu}(L_j) \Psi_j^{\rms}(L_j)| \leq C R_j\,,\quad |P_j^{\rms\rmc}(L_j) \hat\Psi_j^\rmu(-L_j)| 
\leq C \hat R_j$.
\item $\begin{cases}
|P_j^{\rmc\rmu}(L_j) (p_j(\alpha_j + L_j)-q_j(L_j))| 
 \leq C R_j\hat R_j,\\
|P_j^{\rms\rmc}(L_j) (p_j(\alpha_j-L_j)-\hat{q}_j(-L_j))| 
\leq C R_j\hat R_j.
\end{cases}$
\item The above holds under Hypothesis~\ref{h:simplestable}  with $\delta_j=0$ in $\hat R_j$ and under Hypothesis \ref{h:simpleunstable} with $\delta_j=0$ in $R_j$.
\item Under Hypothesis~\ref{h:simplestable} or~\ref{h:simpleunstable}, respectively, there are vectors $v^{\rmu/\rms}_j$ in the leading un/stable eigenspaces of $F_j$ such that
\begin{eqnarray*}
v^\rms_j \neq 0 &\Leftrightarrow& \limsup_{x\to \infty} \ln(q_j(x))/x = -\kappa_j^\rms,\\
v^\rmu_j \neq 0 &\Leftrightarrow& \limsup_{x\to -\infty} \ln(\hat q_j(x))/x = \kappa_j^\rmu,
\end{eqnarray*} 

and such that under Hypothesis~\ref{h:simplestable} and for any $\delta_j^\rms < \min\{\kappa_j^\rms,\rho_j^\rms\}$ we have
\[
\hat\Phi_j^\rms(0,-L)P_j^\rms(L)(p_j(\alpha_j+L) - q_j(L)) = \rme^{2F_j L}v^\rms_j + \calO(\rme^{-2(\kappa_j^\rms + \delta_j^\rms)L}),
\]

and under Hypothesis~\ref{h:simpleunstable} and for any $\delta_j^\rmu < \min\{\kappa_j^\rmu,\rho_j^\rmu\}$ we have
\[
\Phi_j^\rmu(0,L)P_j^\rmu(L)(p_j(\alpha_j-L) - \hat q_j(-L)) = \rme^{-2F_j L}v^\rmu_j + \calO(\rme^{-2(\kappa_j^\rmu +\delta_j^\rmu)L}).
\]
\end{enumerate}
\end{lemma}

\begin{proof}
For readability, we set $\alpha_j=0$, see \eqref{e:asyphase}. Let $\tilde \Psi_j^{\rms/\rmc/\rmu}(x)$ be the stable/center and unstable eigen- or trichotomy projections \emph{on the whole real line}, $x\in \R$, of 
\[
\dot v = \partial_u f(p_j;0)v,
\]
which trivially exist by the Floquet form. Note that these di/trichotomies differ from those of the linearization in $q_j$.

\medskip
\enum
\item First note that 
\[
\begin{cases}
\Psi_j^\rms(L) = \Proj(E_j^{\rmc\rmu}(L), E_j^\rms(L))\\
P_j^{\rmc\rmu}(L) = \Proj(\hat E_j^\rms(-L),E_j^\rmu(L)+\hat E_j^\rmu(-L)),
\end{cases} 
\]
so that $P_j^{\rmc\rmu}(L)\Psi_j^\rms(L)$ is determined by $E_j^\rms(L)-\hat E_j^\rms(-L)$. (An appropriate norm for estimating this difference goes via suitable bases of these linear spaces.)
Since for $L\in K_j(L^0)$ it holds that $\tilde\Psi_j^\rms(-L) = \tilde\Psi_j^\rms(L)$ for all $j\in\virI^\rmo$ (the projections are constant for $j\in\virI_\rmE^\rmo$) we have
\[
E_j^\rms(L)-\hat E_j^\rms(-L) = E_j^\rms(L)
 - \Rg(\tilde \Psi_j^\rms(L)) + \Rg(\tilde \Psi_j^\rms(-L)) 
 -\hat E_j^\rms(-L),
\]
and we will estimate the two differences on the right hand side separately.

General perturbation estimates of dichotomies (e.g. Lemma 1.2(i) in \cite{bjorndiss}) imply $\Psi_j^\rms(L) -\tilde\Psi_j^\rms(L) = \calO(\rme^{(\delta_j-\kappa_j^\rms)L})$ and $ \hat \Psi_j^{\rmu}(-L) - \tilde\Psi_j^{\rmu}(-L) = \calO(\rme^{(\delta_j-\kappa_j^\rmu)L})$.

Hence, on the one hand $E_j^\rms(L) - \Rg(\tilde\Psi_j^\rms(L)) = \calO(\rme^{(\delta_j-\kappa_j^\rms)L})$.

On the other hand we can write 
\[
\tilde\Psi_j^\rms(-L)-\hat \Psi_j^\rms(-L) = 
\Id - \tilde\Psi_j^{\rmc\rmu}(-L)- (\Id - \hat \Psi_j^{\rmc\rmu}(-L))
= \hat \Psi_j^{\rmc\rmu}(-L) - \tilde\Psi_j^{\rmc\rmu}(-L),
\]
and, due to asymptotic phase we have 
$\frac{\rmd}{\rmd x}\hat q_j(-L) - \frac{\rmd}{\rmd x} p_j(-L) = \calO(\rme^{(\delta_j-\kappa_j^\rmu)L})$ in the center direction. Therefore, $\tilde\Psi_j^\rms(L)-\hat \Psi_j^\rms(-L) = \calO(\rme^{(\delta_j-\kappa_j^\rmu)L})$ and $\Rg(\tilde\Psi_j^\rms(L))-\hat E_j^\rms(-L) = \calO(\rme^{(\delta_j-\kappa_j^\rmu)L})$. (And analogously $\Psi_j^\rms(L)-\tilde\Psi_j^\rms(L)=\calO(\rme^{(\delta_j-\kappa_j^\rms)L})$).

In combination, since $\hat E_j^\rms(-L) \subset \ker P_j^{\rmu}(L)$ the weak version of the first estimate follows. The strong version of this estimate in straight coordinates is a consequence of the fact that then $E_j^\rms(L) = \Rg(\tilde\Psi_j^\rms(L))$ for all $L\geq L^1$ for sufficiently large $L^1\geq L^0$. Hence, for $L\geq L_1$ we have
\[
E_j^\rms(L)-\hat E_j^\rms(-L) = \Rg(\tilde \Psi_j^\rms(-L)) 
 -\hat E_j^\rms(-L),
\]
which implies the stronger estimate.

The proof of the second estimate is completely analogous. 
\item Since the stable manifold is a (at least) quadratic graph over the stable eigenspace for $j\in \virI_\rmE$ and center-stable trichotomy space for $j\in \virI_\rmP$ at $p_j$ we have that 
\[
p_j(L)-q_j(L) = 
\tilde\Psi_j^{\rms\rmc}(L)(p_j(L)-q_j(L)) + 
\calO((p_j(L)-q_j(L))^2).
\]
On the other hand, as in the proof of the previous item, we can replace $\tilde\Psi_j^\rms(L)$ by $\hat \Psi_j^\rms(L)$ with error of order $\rme^{(\delta_j-\kappa^\rmu_j)L}$. Since $\hat E_j^\rms(L)$ lies in the kernel of $P_j^{\rmc\rmu}(L)$ and $p_j(L)-q_j(L)= \calO(\rme^{(\delta_j-\kappa_j^\rms)L})$ there are non-negative constants $C^*$ and $C_*$ such that
\begin{eqnarray}
P_j^{\rmc\rmu}(L)(p_j(L)-q_j(L)) &=& 
P_j^{\rmc\rmu}(L)\tilde \Psi_j^\rmc(L)(p_j(L)-q_j(L)) \label{e:err}\\
&& + \calO(C_* \rme^{(\delta_j-\kappa^\rms_j-\kappa_j^\rmu)L}
+ C^*\rme^{2(\delta_j-\kappa^\rms_j)L}).\nonumber
\end{eqnarray}

For $j\in \virI_\rmP^\rmo$ recall that the asymptotic phase as $L\to\infty$ of $q_j(L)$ is that of $p_j(L)$, hence $q_j(L)$ lies in the strong stable fiber with phase $L$. Strong stable fibers are (at least) quadratic graphs over the (strong) stable trichotomy spaces so that 
\[
\tilde \Psi_j^{\rmc}(L)(p_j(L)-q_j(L)) = \calO((p_j(L)-q_j(L))^2) = 
\calO(\rme^{2(\delta_j-\kappa_j^\rms)L}).
\]
The claimed weak estimate follows from combining this estimate with \eqref{e:err}.

The stronger estimate for straight coordinates means $C^*=0$ in \eqref{e:err}. Indeed, since the strong un/stable fibers and (center) un/stable manifold coincide with their tangent spaces at $p_j$ the higher order corrections disappear and \eqref{e:err} holds with $C^*=0$.

The estimate for $p_j(\alpha_j-L)-\hat{q}_j(-L)$ is completely analogous.
\item For simple eigenvalues or Floquet exponents there is no secular growth so that the rates of convergence are in fact the leading rates.
\item This is a reformulation of results in \cite{myhombif} as follows. 

Concerning the right hand sides without $\hat\Phi_j^\rms(0,-L)$ and $\Phi_j^\rmu(0,L)$, respectively, Lemma 10 in \cite{myhombif} yields the expansions 
\begin{align*}
P_j^\rms(L)(p_j(\alpha_j+L) - q_j(L)) &= A_j(L)\rme^{F_j L}\tilde v^\rms_j + \calO(\rme^{-2(\kappa_j^\rms + \delta_j^\rms)L}),\\
P_j^\rmu(L)(p_j(\alpha_j-L) - \hat q_j(-L)) &= A_j(-L)\rme^{-F_j L}\tilde v^\rmu_j + \calO(\rme^{-2(\kappa_j^\rmu +\delta_j^\rmu)L}),
\end{align*}
for certain $\tilde v_j^{\rmu/\rms}$ in the leading un/stable eigenspace of $F_j$. Note that in the present case $\alpha$ and $v$ from that Lemma are constant, and $L\in K_j(L^1)$ so that $p_j(\alpha_j+L) = p_j(\alpha_j-L)$ and $A_j(L) = A_j(-L)$.

Concerning  $\hat\Phi_j(0,-L)$ and $\Phi_j^\rmu(0,L)$, Equation (5.7) in \cite{myhombif} shows that for any $v$ there are $\hat v^{\rmu/\rms}$ in the un/stable eigenspace of $F_j$ such that
\begin{align*}
\hat\Phi_j^\rms(0,-L)v = \rme^{F_jL}A_j(-L)^{-1}\hat v^\rms + \calO\left(\rme^{-(\kappa_j^\rms+ \delta_j^\rms)L}\right),\\
\Phi_j^\rmu(0,L)v = \rme^{-F_jL}A_j(L)^{-1}\hat v^\rmu + \calO\left(\rme^{-(\kappa_j^\rmu+ \delta_j^\rmu)L}\right).
\end{align*}
Since $A_j(L) = A_j(-L)$ for $L\in K_j(L^1)$, and these expansions only depend on the error to the asymptotic vector fields, combination of this with the previous step proves the claim.
\end{list}
~\hfill{}%\qed
\end{proof}

%%%%%%%%%%%%%%%%%%%%%%%%%%%%%%%%%%%%%%%%%%%%%%%%%%
\section{Coordinates of trajectories}\label{s:coord}
%%%%%%%%%%%%%%%%%%%%%%%%%%%%%%%%%%%%%%%%%%%%%%%%%%

Following and improving \cite{myhombif}, in this section we establish a suitable coordinate system for the $(n-1)$-dimensional set of trajectories that pass nearby $p_j$. We consider the difference $V=(w,\hat w)$ between solutions $u$ to \eqref{e:ode} and $q_j$, $\hat q_j$, where $u(0)\sim q_j(0)$ and $u(2L)\sim \hat q_j(0)$. (In this section $q_j, \hat q_j$ can be any orbits that lie in the stable and unstable manifolds of $p_j$, respectively.) We determine all such $V$ by an implicit function theorem, where the time shift along a trajectory is removed to make $V$ unique. For equilibria this is done in the usual way of \cite{Lin90} by imposing certain boundary value data of $u$ in Poincar\'e sections attached to the in- and outflowing solutions $q_j$ and $\hat q_j$, and adding a continuous parameter $L$ so that $2L$ is the time spent between the section. 

For periodic orbits it would be natural to do the same, only $L$ would not come from a connected unbounded interval in order for $V$ to be small. 
However, due to our approach via a certain variation-of-constants solution operator, we slightly deviate from this, see also \cite{myhombif}. Briefly, in order to control the integrals over the centre part of the trichotomy projections, we use exponentially weighted spaces. This causes difficulty to control the centre projection of another integral term from coupling in- and outflow, which stems from the deviation of the phase with respect to $p_j$ at $x=L$. Due to asymptotic phase this term vanishes in the limit $L\to \infty$, but integrability requires a good estimate. We avoid this and at the same time remove the phase shift by simply requiring that the centre parts of $w(x)$ and $\hat w(\hat x)$ at time $x=L$, $\hat x=-L$ vanish. This precisely disallows time shifts and the result is equivalent to the approach by Poincar\'e sections. In particular, the resulting $V$ give a parameterisation of all orbits near $p_j$ with normalised travel time sequence. A posteriori, the reconstructed solutions approximately satisfy boundary conditions in suitable linear Poincar\'e sections and $2L$ is the approximate travel time between these.

Notably, this parametrises trajectories in a neighborhood of $\{q_j(x):x\geq 0\}\cup\{\hat q_j(\hat x) : \hat x\leq 0\}$. An alternative to this approach is to parametrise trajectories in a small neighborhood of $p_j$ and then insert a `global' trajectory piece between the inflow and outflow  boundaries of these neighborhoods, see \cite{becks,krock,Shilnikov}.

%%%%%%%%%%%%%%
\subsection{Passage coordinates}
%%%%%%%%%%%%%%
We choose the boundary value data in subsets of Poincar\'e sections defined via the trichotomies. For $j\in\virI_\rmP$ the space $(q_j(0),\hat q_j(0)) + E_j^\rms(0)\times \hat E_j^\rmu(0)$ is $(n-1)$-dimensional since the flow direction counts towards stable \emph{and} unstable directions for periodic orbits. On the other hand, for $j\in \virI_\rmE^\rmo$ we have $\dot q_j(x)\in E_j^\rms(x)$ and $\dot{\hat {q}}(x)\in \hat E_j^\rmu(x)$ so that after removing the flow directions for in- and outflow $n-2$ dimensions are left. As mentioned, this is compensated by a continuous `travel time' parameter $L\in K_j(L^1)$. 
To eliminate the flow direction for $j\in \virI_\rmE^\rmo$ we define
\begin{eqnarray*}
%Q_j^\rmu &:=& \Proj\left(E_j^\rms(0)), E_j^\rmu(0)\right),\\
%
Q_j^\rms &:=& \Proj\left(E_j^\rmu(0) + \mathrm{span}(\dot q_j(0)), E_j^\rms(0) \cap \mathrm{span}(\dot q_j(0))^\perp\right),\\
\hat Q_j^\rmu &:=& \Proj\left(\hat E_j^\rms(0) + \mathrm{span}(\dot{\hat q}_j(0)), E_j^\rmu(0) \cap \mathrm{span}(\dot{\hat q}_j(0))^\perp\right).
\end{eqnarray*}
Here $E^\perp$ is the orthogonal complement of a linear space $E$. For the outflow we define $\hat Q_j^\rmu$ accordingly, and set $Q_j^\rms:=\hat Q_j^\rmu:=\Id$ for $j\in \virI_\rmP^\rmo$. 

Now boundary data in $Q_j^\rms E_j^\rms(0)\times \hat Q_j^\rmu \hat E_j^\rmu(0)$ for all $j\in \virI^\rmo$ excludes precisely the flow directions at $q_j(0)$ and $\hat q_j(0)$. Recall that $E_j^\rms(0)$ and $\hat E_j^\rmu(0)$ are determined uniquely by the di-/trichotomy.

Finally, for all $j\in J$ we define the  Poincar\'e sections
\begin{eqnarray*}
\Sigma_j &:=& q_j(0) + Q_j^\rms E_j^\rms(0) + E_j^\rmu(0),\\
\hat\Sigma_j &:=& \hat q_j(0) + \hat E_j^\rms(0) + \hat Q_j^\rmu \hat E_j^\rmu(0).
\end{eqnarray*}

Hence, for $j\in\virI_\rmE^\rmo$ we solve the boundary value problem \eqref{e:ode} subject to $u(0)\in \Sigma_j$ and $u(2L)\in\hat\Sigma_j$.
As mentioned, for $j\in\virI_\rmP^\rmo$ there are technical reasons not to consider this boundary value problem. In addition, the set of travel times has to be `phase coherent' in order for the variations $V$ to be small near the periodic orbit. For convenience we a priori restrict $L$ to the set $K_j(L^1)$ which already appeared in Lemma~\ref{l:proj-est}. For general Poincar\'e sections, the set $K_j$ would need to be adjusted and therefore, in general, differs for each $\mu$, which is inconvenient for the leading order expansion of parameters. The tradeoff is that orbits starting in $\Sigma_j$ do not necessarily lie exactly in $\hat\Sigma_j$ for $L\in K_j(L^1)$: instead of having zero center part at time $2L$ we will require this at time $L$. Since near $q_j(0)$ the flow acts as a diffeomorphism between any choice of hyperplanes transverse to the flow this is no restriction, but rather a normalisation of the discrete travel times.

\medskip
Due to the geometric interpretation, we refer to the boundary data as \emph{coordinate parameters} and denote
\[
\Omega_j:= Q_j^\rms E_j^\rms(0) \times \hat Q_j^\rmu \hat E_j^\rmu(0),
\] 
with elements $\om_j=(\om^\rms_j,\hat\om^\rmu_j)$ for $\om^\rms_j\in Q_j^\rms E_j^\rms(0)$, $\hat\om_j \in \hat Q_j^\rmu \hat E_j^\rmu(0)$, see Figure~\ref{f:notationpar}.
\begin{figure}%[htbp]
\begin{center}
\input{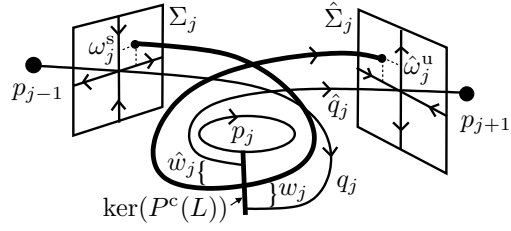}
\caption{Illustration of the notation for the flow near a periodic orbit $p_j$ with minimal period $T_j$, $j\in J_\rmP$. Theorem~\ref{t:1}, see also Corollary~\ref{c:1}, show that trajectories near $\{q_j(x) : x \geq 0\}\cup \{\hat q_j(\hat x) : \hat x \leq 0\}$ are parametrised by $\om_j^\rms\in E_j^\rms(0)$, $\hat \om_j^\rmu\in \hat E_j^\rmu(0)$ and $L\in \{\ell T_j/2 : \ell\in\N, \ell T_j/2 \geq L_0\}$.}\label{f:notationpar}
\end{center}
\end{figure}

Since we also want to parametrise un/stable manifolds including the flow direction for equilibria, we define in addition
\[
\Omega_j^\rmf:=E_j^\rms(0)\times \hat E_j^\rmu(0).
\]

We now turn to the aforementioned difference $V$ between solutions and heteroclinic orbits. Given a solution $u$ for parameters $\mu$ we call any
\[
V(x,\hat x; \mu, u,\sigma, L):= \left(u(\sigma+ x)-q_j(x),u(\sigma + \hat x +2L)-\hat q_j(\hat x)\right)
\] 
a \emph{$j$-variation} of $u$ and denote the components as $V=(w,\hat w)$. A solution $u$ of \eqref{e:ode} for $x\in [0,2L]$ can be reconstructed from a given $j$-variation if on the one hand
\begin{eqnarray}
\frac{\mathrm{d}}{\mathrm{d}x} w(x) &=& \partial_u f(q_j(x);0)w(x) 
+ g_j(w(x),x;\mu),\label{e:vode}\\
\frac{\mathrm{d}}{\mathrm{d}\hat x} \hat w(\hat x) &=& 
\partial_u f(\hat q_j(\hat x);0)\hat w(\hat x) 
+ \hat g_j(\hat w(\hat x),\hat x;\mu),\nonumber
\end{eqnarray}
for $x\in [0,L]$ and $\hat x \in [-L,0]$, where
\begin{eqnarray*}
g_j(w(x),x;\mu) &:=& f(q_j(x) + w(x);\mu) - f(q_j(x);0) 
- \partial_u f(q_j(x);0)w(x),\\
\hat g_j(\hat w(\hat x),\hat x;\mu) &:=& f(\hat q_j(\hat x) 
+ \hat w(\hat x);\mu) - f(\hat q_j(\hat x);0) 
- \partial_u f(\hat q_j(\hat x);0)\hat w(\hat x).
\end{eqnarray*}
On the other hand, the reconstructed orbit is given by
\[
\begin{cases}
q_j(x)+w_j(x) &,\; x\in [0,L]\\
\hat{q}_j(x-2L) + \hat{w}_j(x-2L)&,\; x\in [L,2L]
\end{cases}
\] 
and is continuous, if $w_j(L) - \hat{w}_j(-L) = -q_j(L) + \hat{q}_j(-L)$. Therefore, we define 
\[
b_j(L) := \hat{q}_j(-L) - q_j(L),
\]
which will be the main contribution to the expansion given in \eqref{e:intro}.

We look for solutions that are simultaneously close to $q_j$ and $\hat{q}_j$, so that the $j$-variation $V$ is small and given by an implicit function theorem. Since $b_j(L)$ is asymptotically periodic for $j\in \virI_\rmP^\rmo$ as $L\to \infty$ the aforementioned `phase coherence' condition appears, and for $L\in K_j(L)$ we indeed have smallness:
\[
b_j(L) = p_j(\alpha_j + L) -q_j(L) + \hat q _j(-L) - p_j(\alpha_j + L) 
= \calO(\rme^{-\kappa_j L}).
\]

Moreover, Lemma~\ref{l:proj-est}(4) implies for $L_j\in K_j(L^1)$  that
\begin{eqnarray}
P_j^\rmu(L_j)b_j(L_j) &=& \calO(R_j)\label{e:b-est-u}\\
P_j^{\rms\rmc}(L_j)b_j(L_j) &=& \calO(\hat R_j)\label{e:b-est-s}.
\end{eqnarray}

The trichotomies imply for $x, -\hat x\in [0,L]$ that
\[
\R^n \sim E^\rms(x)\times E^\rmc(x)\times E^\rmu(x)
\sim \hat E^\rms(\hat x) \times \hat E^\rmc(\hat x)
\times \hat E^\rmu(\hat x)
\]
and provide a decomposition $w=w_j^\rms+w_j^\rmc+w_j^\rmu$ into $w_j^{\rms/\rmc/\rmu}(x):= \Psi^{\rms/\rmc/\rmu}_j(x) w(x)$ and $\hat w=\hat w_j^\rms+\hat w_j^\rmc+\hat w_j^\rmu$ into $\hat w_j^{\rms/\rmc/\rmu}(\hat x):= \hat \Psi^{\rms/\rmc/\rmu}_j(\hat x) \hat w(\hat x)$. We use the analogous superscripts for the decomposition of $g_j$ and $\hat g_j$. 

Similar to \cite{bjorndiss}, Lemma 3.4, the following estimates hold. We first change coordinates and rescale time so that (see \cite{myhombif}) there is $\eps_0>0$ such that for $|\mu|\leq \eps_0$ we have $p_j(x;0)=p_j(x;\mu)$, $j\in\virI_\rmred$, i.e., $p_j$ is independent of $\mu$ for $|\mu|\leq \eps_0$.

%%%%%%%%%%%%%%
\begin{lemma}\label{l:g-est} There is $C>0$ depending only on $\delta_j$ and $q_j, \hat q_j$ such that for $x,-\hat x\geq 0$,
\begin{eqnarray}
 |g_j(w,x;\mu)| &\leq& 
   C\left(|w|^2 + |\mu|(|w| + \rme^{(\delta_j-\kappa_j^\rms)x})\right)\\
|g^\rms_j(w,x;\mu)| &\leq& 
C\left( (|w_j^\rms| + \rme^{(\delta_j-\kappa_j^\rms)x})|w|^2 + |\mu|(|w^\rms_j|+\rme^{(\delta_j-\kappa_j^\rms)x} \right) \label{e:g-est-s}\\
 |\hat g_j(w,\hat x;\mu)| &\leq& 
   C\left(|w|^2 + |\mu|(|w| + \rme^{(\delta_j-\kappa_j^\rms)|\hat x|})\right)\\
|\hat g^\rmu_j(\hat w,\hat x;\mu)| &\leq& 
C\left( (|\hat w_j^\rmu| + \rme^{(\delta_j-\kappa_j^\rmu)|\hat x|}|)
|\hat w|^2 + |\mu|(|\hat w^\rmu_j|+ \rme^{(\delta_j-\kappa_j^\rmu)|\hat x|}\right)
\label{e:g-est-u}
\end{eqnarray}
\end{lemma}

\begin{proof}
By definition of straight coordinates, for $w\in\R^n$, we have $\tilde\Psi_j^\rms(x)f(p_j(x) + \tilde\Psi_j^\rmu(x) w;0)\equiv 0$, where $\tilde\Psi_j^{\rms/\rmc/\rmu}(x)$ are the projections of the linear evolution at $p_j$ as in the proof of Lemma~\ref{l:proj-est}. From that proof we know $\Psi_j^\rms(x)-\tilde\Psi_j^\rms(x) = \calO(\rme^{(\delta_j-\kappa_j^\rms)x})$; recall $\Psi^{\rmu/\rms/\rmc}_j(x)$ are the projection of the trichotomies $\Phi(x,y)$, see \eqref{e:trichest}. Therefore, 
\[
\Psi_j^\rms(x)f(p_j(x) + \Psi_j^\rmu(x)w;0)  
= \calO(\rme^{(\delta_j-\kappa_j^\rms)x}).
\]
In the following we omit the argument $x$ for readability and set $f^\rms:= \Psi_j^\rms(x)f$. Hence,
\[
\partial_{uu}f^\rms(p_j + w;0)  
= \calO(|w_j^\rms| + \rme^{(\delta_j-\kappa_j^\rms)x}).
\]
We have
\begin{eqnarray*}
g_j(w,x;\mu) &=& f(q_j+w;\mu) - f(q_j;0) - \partial_u f(q_j;0)w\\
&=& f(q_j+w;\mu) - f(q_j+w;0) + f(q_j+w;0) - f(q_j;0) - \partial_u f(q_j;0)w\\
&=& 
%+ \partial_{u\mu}f(q+\tau s w;\tau s \mu)[w,\mu]\rmd\tau\rmd s\\
%&& 
\int_0^1\partial_\mu f(q+ w;\tau \mu)\mu \rmd \tau
+ \int_0^1\int_0^1\partial_{uu}f(q+\tau s w;0) \tau w^2 \rmd\tau\rmd s,
\end{eqnarray*}
and $\partial_\mu f(p_j+w;\mu) = \calO(|w|)$ since $f(p_j;\mu)=f(p_j;0)$. It follows that $\partial_\mu f = \calO(|w| + \rme^{(\delta_j-\kappa_j^\rms)x})$ which proves the first claimed estimate. We also infer $\partial_\mu f^\rms = \calO(|w^\rms| + \rme^{(\delta_j-\kappa_j^\rms)x})$ which proves the second claimed estimate. The proof of the remaining estimates is completely analogous.~\hfill{}
\end{proof}

Based on this, $V$ can be found with uniform estimates in $L$ for $j\in\virI_\rmP$ in the weighted space
\begin{eqnarray*}
X_{\eta,L} &=& \left(C^0([0,L],\R^n)\times C^0([-L,0],\R^n),\|\cdot\|_{\eta,L}\right),\\
\|(w,\hat w)\|_{\eta,L} &=& \| w\|_{\eta,L}^+ 
+ \| \hat w\|_{\eta,L}^-,\\
\| w\|_{\eta,L}^+ &=& \sup\{|\rme^{\eta x}w(x)| : x\in [0,L]\},\\
\| \hat w\|_{\eta,L}^- &=& \sup\{|\rme^{\eta |\hat x|}\hat w(\hat x)| : \hat x\in [-L,0]\}.
\end{eqnarray*}

By Lemma~\ref{l:g-est}, for any $0\leq \eta_j<\kappa_j$ there is a constant $C>0$ independent of $L$ such that
\begin{equation}
\begin{cases}\label{e:g-est}
   \|g_j(w,\cdot;\mu,L)\|_{\eta_j,L}^+ &\leq 
 	C( (\|w\|_{\eta_j,L}^+)^2 + |\mu|),\\
	   \|\hat g_j(\hat w,\cdot;\mu,L)\|_{\eta_j,L}^- &\leq 
 	C( (\|\hat w\|_{\eta_j,L}^-)^2 + |\mu|).%\\
\end{cases}
\end{equation}

The coordinates of trajectories will be those $(\om_j,\mu,L)$ with $L\in K_j(L_1)$, $\om_j=(\om^\rms_j,\hat\om^\rmu_j) \in \Omega_j$ that generate a fixed point of the map
\[
\calG_j(w,\hat w;\om_j,\mu,L) : X_{\eta,L} \to X_{\eta,L}, \quad j\in \virI^\rmo,
\]
defined next. The maps $\calG_j$ can be derived from a variation-of-constants solution of \eqref{e:vode} decomposed suitably by the trichotomies and are given by
\[
\calG_j(w,\hat w;\om_j,\mu,L)\dblx :=\left(
  \begin{array}{*{1}{l}}
    \Phi_j^\rms(x,0)\om^\rms_j\\ + 
    \int_0^x\Phi_j^\rms(x,y)g_j(w(y),y;\mu)\dy\\
    +\int_L^x\Phi_j^{\rmc\rmu}(x,y)g_j(w(y),y;\mu)\dy\\
    + \Phi_j^\rmu(x,L)P_j^\rmu(L)(\couplin_j(L)\om 
    + \coupnl_j(w,\hat w;\mu,L) + b_j(L))\\
    \hline
    \Phi_j^\rmu(\hat x,0)\hat\om^\rmu_j\\ 
    +\int_{-L}^{\hat x}\hat\Phi_j^{\rms\rmc}(\hat x,y)\hat 
    g_j(\hat w(y),y;\mu)\dy\\
    +\int_0^{\hat x}\hat\Phi_j^\rmu(\hat x,y)\hat 
    g_j(\hat w(y),y;\mu)\dy\\    
    -\hat\Phi_j^{\rms\rmc}(\hat x,-L)P_j^{\rms\rmc}(L)(\couplin_j(L)\om
    + \coupnl_j(w,\hat w;\mu,L) + b_j(L))
  \end{array}
\right)
\]
where the horizontal line separates first and second component $\calG_j = (\calG_{j,1}, \calG_{j,2})$. The terms coupling these components are
\begin{eqnarray*}
	\couplin_j(L)\om_j &=& \hat\Phi_j^\rmu(-L,0)\hat\om^\rmu_j - 
	\Phi_j^\rms(L,0)\om^\rms_j,\\
	\coupnl_j(w,\hat w;\mu,L) &=& \int_0^{-L}\hat\Phi_j^\rmu(-L,y)\hat g_j(\hat w(y),y;\mu) \rmd y - \int_0^L \Phi_j^\rms(L,y)g_j(w(y),y;\mu) \rmd y.
\end{eqnarray*}

By Lemma~\ref{l:proj-est} there is $C>0$ depending only on $\delta_j$ and $q_j, \hat q_j$ such that
\begin{eqnarray}
	|P_j^\rmu(L)\couplin_j(L)\om_j| &\leq& C(R_j\rme^{(\delta_j-\kappa_j^\rms) L}|\om^\rms_j| 
	+ \rme^{(\delta_j-\kappa_j^\rmu) L}|\hat\om^\rmu_j|)\label{e:couplin-est} \\
	|P_j^\rms(L)\couplin_j(L)\om_j| &\leq& C(\rme^{(\delta_j-\kappa_j^\rms) L}|\om^\rms_j| + \hat R_j\rme^{(\delta_j-\kappa_j^\rmu) L}|\hat\om^\rmu_j|)\label{e:couplin-est2}.
\end{eqnarray}

Note that $\Psi_j^\rmc(L)\calG_{j,1}(x=L)=0$ and $\hat\Psi_j^\rmc(-L)\calG_{j,2}(\hat x=-L)=0$. As mentioned above, for $j\in J_\rmP$, this is equivalent to fixing the phase on a reconstructed orbit from a fixed point of $\calG_j$. 
It is shown in \cite{myhombif}, Lemma 4, that fixed points of $\calG_j$ indeed generate the aforementioned reconstructed orbits $u(x)$ of \eqref{e:ode} for $x\in[0,2L]$. The following theorem proves that all orbits can be obtained in this way. Recall that $\om_j\in\Omega^\rmf_j$ contains the flow direction for $j\in J_\rmE$ so that the map from fixed points to trajectories at $x=\hat x=0$ is not injective, while for $\om_j\in\Omega_j$ it is. 

\begin{notation} $B(X,\rho) := \{x\in X : |x|_X \leq \rho\}$ where $|\cdot|_X$ is the norm of $X$.
\end{notation}

\medskip
The following theorem provides the coordinates of trajectories near $q_j$, $\hat q_j$ (in fact near any in- outflow pair at $p_j$); this is formulated more explicitly in Corollary~\ref{c:1} below.

\medskip
\begin{theorem}\label{t:1} Assume Hypothesis~\ref{h:basic} and take $j\in \virI$, as well as any $\eta_j\in(0,\kappa_j)$ if $j\in \virI_\rmP$, and $\eta_j\in[0,\kappa_j)$ if $j\in \virI_\rmE$. There exist $\eps>0$, $L^*\geq L^1$ depending only on $q_j, \hat q_j$ such that the following hold for $L\in K_j(L^*)$, $\mu\in B(\R^d, \eps)$, $\om_j\in B(\Omega^\rmf_j,\eps)$.

\enum
	\item % Existence and uniqueness
	There exists a unique $V_j=V_j(\om_j,\mu,L)\in X_{\eta_j,L}$ 
  such that $V_j = \calG_j(V_j;\om_j,\mu,L)$. In addition, $V_j$ is $C^k$ smooth in $(\om_j,\mu)$ and for $j\in \virI_\rmE^\rmo$ also 
  $C^k$ smooth in $(\om_j,\mu,L)$.
  \item % Completeness 
  Let $V(x,\hat x;\mu,u,0,L)=(w(x),\hat w(\hat x))$ be the $j$-variation of a solution $u$ of \eqref{e:ode} with $u(0)\in \Sigma_j$ and, for $j\in \virI_\rmE^\rmo$, $u(2L)\in \hat\Sigma_j$. If $|V(0;\mu,u,\sigma,L)|\leq \eps$ then there is unique $\sigma^*=
  \calO(|P_j^\rmc(L)w(L;\mu,u,\sigma)|)$ such that
    \[
  	V(\cdot;\mu,u,\sigma^*,L) \equiv V_j(\om_j,\mu,L)
	\] 
	for $\om_j = \left(Q_j^\rms\Psi_j^{\rms}(0)w(0;\mu,u,\sigma^*,L),\hat Q_j^\rmu\hat \Psi_j^\rmu(0)\hat w(0;\mu,u,\sigma^*,L)\right)$.
  \item % expansion
    For any $\delta_j\in(\eta_j,\kappa_j)$, there is $C>0$ such that a fixed point $(W,\hat W)(\om_j,\mu,L)$ of $\calG_j(\cdot,\om_j,\mu,L)$ satisfies, for $x,-\hat x \in [0,L]$, 
   \begin{eqnarray}
    \|W(\om_j,\mu,L)\|_{\eta_j,L}^+ &\leq& 
    C(|\mu|+|\om_j^\rms| + R_j),\nonumber\\
    \|\hat W(\om_j,\mu,L)\|_{\eta_j,L}^- &\leq& 
     C(|\mu|+|\hat\om_j^\rmu| + \hat R_j),\nonumber\\
    |\Psi_j^\rms(x)W_j(\om_j,\mu,L)(x)| &\leq& 
    C \rme^{(\delta_j-\kappa_j^\rms)x},\label{e:est-stab}\\
    |\hat\Psi_j^\rmu(\hat x)\hat W_j(\om_j,\mu,L)(\hat x)| &\leq& 
    C \rme^{(\delta_j-\kappa_j^\rmu)\hat x}.\label{e:est-unstab}
  \end{eqnarray}
  \item Under Hypothesis~\ref{h:simplestable} or~\ref{h:simpleunstable} the estimates \eqref{e:est-stab} and \eqref{e:est-unstab} hold with $\delta_j=\eta_j$ in $\hat{R}_j$ or in $R_j$, respectively.
\end{list}
\end{theorem}

Concerning the required smoothness of $f$, the loss of two degrees of differentiability is due to the coordinate changes (which can be performed simultaneously) that involve $f'$ and that $g$ contains $f'$.

\medskip
Before proving this theorem, to emphasise the coordinate system character we reformulate part of Theorem~\ref{t:1} in the following corollary taking $\om_j\in\Omega_j$. Recall from the beginning of this section that then the set of parameters $(\om_j,L)$ is $(n-1)$-dimensional for all $j\in\virI^\rmo$. The theorem in particular shows that this is also true for the set of fixed points.

\begin{corollary}\label{c:1}
For any $j\in \virI^\rmo$ there is $\eps>0$ and $L^*\geq L^1$ as well as a neighbourhood $\calU$ of $\{q_j(x) : x\in [0,L]\}\cup \{\hat q_j(x) : x \in [-L,0]\}$ such that the following holds. The set of solutions $u$ of \eqref{e:ode} that lie in $\calU$ and satisfy $u(0)\in\Sigma_j$, and, if $j\in \virI_\rmE^\rmo$, $u(2L)\in \hat\Sigma_j$ for $L\in K_j(L^*)$ is in one-to-one correspondence with the parameters $\{(\om_j,\mu,L) : \om_j \in B(\Omega_j,\eps), L\in K_j(L^*), |\mu|\leq \eps \}$ of fixed points of $\calG_j$.
\end{corollary}

\medskip
\emph{Proof of Theorem~\ref{t:1}.}
Items 1 and 2 are a consequence of Theorem 1 in \cite{myhombif} as follows. There it was assumed that $|\mu|\leq \eps\rme^{-\eta_j L}$, but in the present case we can take $|\mu|\leq \eps$ due to the following estimate. For any $0<\eta_j<\kappa_j$ we have
\[
 \|\partial_{w}g_j(w,\cdot;\mu)\|_{\eta,L}^+ 
 \leq K_1(\|w\|_{\eta,L}^+ + |\mu|),
\]
and the same with hats in the $\|\cdot\|_{\eta,L}^-$-norm (see \cite{bjorndiss} Lemma 3.1). This estimate improves the corresponding estimate at the end of the proof of Lemma 5 in \cite{myhombif} as needed. All constants are uniform in $L\in K_j(L^*)$. 

\medskip
Items 3 and 4 improve the estimate of Theorem 1 in  \cite{myhombif} using Lemma~\ref{l:proj-est} as follows. Theorem 1 in \cite{myhombif} states
\[
\|(W_j,\hat W_j)\|_{\eta,L} \leq C(|\mu| + |\om_j| + \rme^{(\eta -\kappa_j)L}),
\]
where for $j\in\virI_\rmE$ we can set $\eta=0$ since there is no centre direction.

To improve this we consider $W_j$ and $\hat W_j$ separately and note that the coupling of these only enters through $\coupnl_j$, while the dependence of $W_j$ on $\hat \om^\rmu_j$ and $\hat W_j$ on $\om_j^\rms$ is only by $\couplin_j$. On account of~\eqref{e:couplin-est}, \eqref{e:couplin-est2}, the decomposition of $\om$ in the claimed separate estimates of $W_j$ and $\hat W_j$ follows. The estimates~\eqref{e:b-est-u},~\eqref{e:b-est-s} prove the claimed separation for the remainder term $b_j(L)$.  Again note that $\eta=0$ is allowed for $j\in\virI_\rmE$. 

It remains to estimate $\coupnl_j$ and to prove the pointwise estimates for $W_j^\rms = \Phi_j^\rms W$ and $\hat W_j^\rmu = \hat \Phi_j^\rms \hat W$. 
We first consider the pointwise estimates. By definition of $\calG_{j,1}$
\[
W_j^\rms(x) = \Phi_j^\rms(x,0)\om_j^\rms + \int_0^x\Phi_j^\rms(x,y)g_j(W_j(y),y;\mu)\rmd y,
\]
and so using \eqref{e:trichest} and \eqref{e:g-est-s}, for a weight $\eta$ to be determined there is $C>0$ such that,
\begin{eqnarray*}
\rme^{\eta x}|W_j^\rms(x)| &\leq& C\left(\rme^{(\eta-\kappa_j^\rms)x}|\om_j^\rms| + \int_0^x\rme^{(\eta-\kappa_j^\rms)x +\kappa_j^\rms y}
\left((|W_j^\rms(y)| + \rme^{(\delta_j-\kappa_j^\rms) y})|W_j(y)|^2 \right.\right.\\
&& \left.\left.+ |\mu|(|W_j(y)| + \rme^{(\delta_j-\kappa_j^\rms)y})\right) \rmd y
\right)\\
&\leq& C\left(\rme^{(\eta-\kappa_j^\rms)x}|\om_j^\rms| 
+ \int_0^x\rme^{(\eta-\kappa_j^\rms)x +\kappa_j^\rms y}
\left((\rme^{-\eta y}\|W_j^\rms\|^+_{\eta,L} + \rme^{(\delta_j-\kappa_j^\rms) y}) \right.\right.\\
&& \left.\left.\times\;\rme^{-2\eta y}(\|W_j\|_{\eta,L}^+)^2
 + |\mu|(\rme^{-\eta y}\|W_j^\rms\|^+_{\eta,L} + \rme^{(\delta_j-\kappa_j^\rms)y})\right) \rmd y\right).
\end{eqnarray*}
Taking maxima over $x$ this implies for any $0<\eta<\kappa_j^\rms$ (and $0\leq \eta < \kappa_j^\rms$ for $j\in\virI_\rmE$) that
\[
\|W_j^\rms\|_{\eta,L}^+ \leq C\left(|\om_j^\rms| + (1+\|W_j^\rms\|_{\eta,L}^+)(\|W_j\|^+_{\eta,L})^2 + |\mu|\right),
\]
and in particular for any $\delta>0$ there exists $C>0$ such that
\[
\|W_j^\rms\|_{(\delta-\kappa_j^\rms),L}^+ \leq C \; \Leftrightarrow \; |W_j^\rms(x)| \leq C\rme^{(\delta - \kappa_j^\rms)x}, \; x\in [0,L].
\]
The estimate for $\hat W_j(x)$ is completely analogous, only now $\kappa_j^\rmu$ bounds $\eta$.

We now turn to the required estimate involving $\coupnl_j$.
Substituting \eqref{e:est-stab}, \eqref{e:est-unstab} into \eqref{e:g-est-s} and~\eqref{e:g-est-u} gives $C>0$ such that
\begin{eqnarray}
|g^\rms_j(w,x;\mu)| &\leq& 
C\left( \rme^{(\delta_j-\kappa_j^\rms)x}(|w|^2 + |\mu|) \right) \label{e:g-est-s2}\\
|\hat g^\rmu_j(\hat w,\hat x;\mu)| &\leq& 
C\left( \rme^{(\delta_j-\kappa_j^\rmu)|\hat x|}|(
|\hat w|^2 + |\mu|) \right).
\label{e:g-est-u2}
\end{eqnarray}

Concerning $\coupnl_j$ these estimates yield (suppressing $y,\mu$ in $\hat g_j$)
\begin{eqnarray*}
|\int_0^{-L}\hat \Phi_j^\rmu(-L,y)\hat g_j(\hat W_j)\rmd y| &\leq& C\int_0^{-L}\rme^{-\kappa_j^\rmu(L+y)}\rme^{(\delta_j-\kappa_j^\rmu)|y|}(|\hat W_j|(y)|^2 + |\mu|)\rmd y\\
&\leq& C(\int_0^{-L}\rme^{-\kappa_j^\rmu L}\rme^{-\delta_j |y|}(\|\hat W_j\|^-_{\delta,L})^2\rmd y + |\mu|)\\
&\leq& C\left( 
\rme^{-\kappa_j^\rmu L}(\|\hat W_j\|^-_{\eta_j,L})^2 + |\mu| \right).
\end{eqnarray*}

Analogously, 
\[
|\int_0^L \Phi_j^\rms(L,y) g_j(W_j,y,\mu) \rmd y |\leq
 C \left(\rme^{-\kappa_j^\rms L}(\|W_j\|^+_{\eta_j,L})^2 + |\mu|\right).
\]
Hence, using Lemma~\ref{l:proj-est} there is $C>0$ depending on $\delta_j>0$ such that
\begin{eqnarray}
|P_j^\rmu(L)\coupnl_j(W_j,\hat W_j;L,\mu)| &\leq& 
C(R_j(\|\hat W_j\|^-_{\eta_j,L})^2 + R_j\hat R_j (\|W_j\|^+_{\eta_j,L})^2 + |\mu|),
\label{e:coupnl-est}\\
|P_j^{\rms\rmc}(L)\coupnl_j(W_j,\hat W_j;L,\mu)| &\leq& 
C(\hat R_j R_j(\|\hat W_j\|^-_{\eta_j,L})^2 + \hat R_j(\|W_j\|^+_{\eta_j,L})^2 + |\mu|).
\label{e:coupnl-est2}
\end{eqnarray}
This completes the proof of the claimed estimates.
~\hfill{}$\Box$%\qed
%\end{proof}

Note that the differences in this result between an equilibrium ($j\in\virI_\rmE$) and a periodic orbit ($j\in\virI_\rmP$) are twofold: 
\enum
\item The ranges of $L$ values measuring the time spent near $p_j$ are a semi-infinite interval for $j\in \virI_\rmE^\rmo$ and a discrete infinite (phase coherent) sequence for $j\in \virI_\rmP^\rmo$,
\item In the periodic case the exponential weight $\eta_j$ must be strictly positive.
\end{list}

%%%%%%%%%%%%%%%%%%%%%%%%%%%%%%%%%%%%%%%%%%%%%%%%%%
\subsection{Stable and unstable manifolds}

For homoclinic or heteroclinic connections it is helpful to parametrise the stable and unstable manifolds of $p_j$ in the same way as above. This is in fact simpler than the general case and we can set $L$ in the definition of $\calG_j$ to infinity, which gives, for $j\in\virI^\rmo$,
\[
\calG_j^\infty(w;\om_j^\rms,\mu)(x) :=\left(
  \begin{array}{*{1}{l}}
    \Phi_j^\rms(x,0)\om_j^\rms\\ + 
    \int_0^x\Phi_j^\rms(x,y)
    g_j(w(y),y;\mu)\dy\\
    +\int_\infty^x\Phi_j^{\rmc\rmu}(x,y)
    g_j(w(y),y;\mu)\dy
  \end{array}
\right)
\]

\[
\hat\calG_j^\infty(w;\hat\om_j^\rmu,\mu)(\hat x) :=\left(
  \begin{array}{*{1}{l}}
    \hat\Phi_j^\rmu(\hat x,0)\hat\om_j^\rmu\\
    +\int_{-\infty}^{\hat x}\hat\Phi_j^{\rms\rmc}(\hat x,y) 
    \hat g_j(\hat w(y),y;\mu)\dy\\
    +\int_0^{\hat x}\hat\Phi_j^\rmu(\hat x,y)
    \hat g_j(\hat w(y),y;\mu)\dy\\    
  \end{array}
\right)
\]
The same change in the proof of Theorem~\ref{t:1} gives the following corollary concerning the parametrization of un/stable manifolds $\calW^{\rms/\rmu}(p_j)$.

%++++++++++++++++++++++++++++++++++++++
\begin{corollary}\label{c:fibers}
	There exists $\eps>0$ such that for all $(\mu,\om_j)\in \R^d\times \Omega^\rmf_j$ with $|\mu| + |\om_j|\leq \eps$ the operators $\calG_j^\infty$ and $\hat\calG_j^\infty$ have unique $C^k$ smooth fixed points $W_j^\infty(x;\om_j^\rms,\mu)$ and $\hat W_j^\infty(\hat x;\hat\om_j^\rms,\mu)$ satisfying the following.

\enum
\item For $j\in \virI_\rmE$ the maps $\om_j^\rms \mapsto W_j^\infty(0;\om_j^\rms,\mu)(0)$ and $\hat\om_j^\rmu \mapsto \hat W_j^\infty(0;\hat\om_j^\rmu,\mu)(0)$ parametrise $\calW^\rms(p_j)$ and $\calW^\rmu(p_j)$ near $q_j(0)$ and $\hat q_j(0)$ over $E_j^\rms(0)$ and  $\hat E_j^\rmu(0)$, respectively. 
\item For $j\in \virI_\rmP$, $\alpha\in[0,T_j)$ the maps $\om_j^\rms \mapsto W_j^\infty(\alpha;\om_j^\rms,\mu)$ and $\hat\om_j^\rmu \mapsto \hat W_j^\infty(\alpha;\hat\om_j^\rmu,\mu)$ parametrise the strong stable and unstable fibers of $p_j$ with phase $\alpha_j+\alpha$ over $E_j^\rms(0)$ and $\hat E_j^\rmu(0)$, respectively.
\item Theorem~\ref{t:1}(3) holds for $W_j^\infty$ with $R_j=0$ and $\hat W_j^\infty$ with $\hat R_j=0$.
\end{list}
\end{corollary}

%%%%%%%%%%%%%%%%%%%%%%%%%%%%%%%%%%%%%%%%%%%%%%%%%%
\section{Bifurcation equations}\label{s:bifeq}
%%%%%%%%%%%%%%%%%%%%%%%%%%%%%%%%%%%%%%%%%%%%%%%%%%

\begin{figure}%[htbp]
\begin{center}
\input{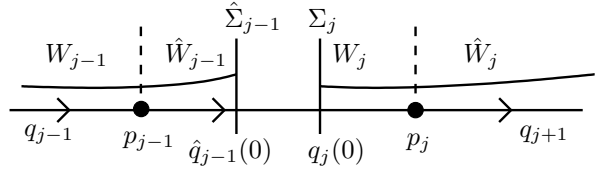}
\caption{Schematic illustration of adjacent $j$-variations.}
\label{f:notation}
\end{center}
\end{figure}

Based on the results of the previous section we derive reduced
equations whose solutions parametrise all solutions of \eqref{e:ode} that are near the chosen itinerary $\calC$. Throughout this section we take $\om_j=(\om_j^\rms,\hat\om_j^\rmu)\in\Omega_j$.

In order to reconstruct solutions of \eqref{e:ode} from the variations $(W_j,\hat W_j)$ about adjacent $q_j$ these need to fit together continuously. By definition of the variations this means solving (up to the flow direction as shown below) the system of equations
\begin{equation}\label{e:1}
W_j(x=0;\om^\rms_j,\hat\om^\rmu_j,\mu,L_j) = \hat{W}_{j-1}(\hat x=0;\om_{j-1}^\rms,\hat\om_{j-1}^\rmu,\mu,L_{j-1}),
\end{equation}
where $j\in \virI^\rmo$. System \eqref{e:1} is closed if $\virI=\mathbb{Z}$, and closing conditions are required for $\virI$ with upper or lower bound.

In case of finite $\virI$ reconstructed solutions are either heteroclinic from $p_1$ to $p_m$ (`het' in short) and, for $p_1=p_m$, homoclinic to $p_1$ (`hom') or periodic orbits (`per'). Note that the same periodic orbit is a solution for any periodically prolonged itinerary, and in this case we implicitly assume $\calC$ is a heteroclinic cycle. The remaining cases are semi-unbounded $\virI$ for which we require the corresponding solution to lie in the un/stable manifold of $p_j$ with the largest or smallest index, respectively.

More formally, this means

\begin{list}{}{
\setlength{\labelwidth}{20mm}
\setlength{\leftmargin}{20mm}}
\item[`het':] $q_2(0) + W_2(0) \in \calW^\rmu(p_1)$ and $\hat{q}_{m-1}(0) + \hat{W}_{m-1}(0) \in \calW^\rms(p_m)$, i.e.,  
\begin{eqnarray*}
W_2(0;\om_2,\mu,L_2) &=& \hat W_1^\infty(\alpha_1;\om_1^\rmu,\mu),\\ \hat W_{m-1}(0;\om_{m-1},\mu,L_{m-1}) &=& W_m^\infty(\alpha_m;\om_m^\rms,\mu),
\end{eqnarray*}
\item[`hom':]  (really the same as `het', here extra only for clarity) \\ $q_2(0) + W_2(0) \in \calW^\rmu(p_1)$ and $\hat{q}_{m-1}(0) + \hat{W}_{m-1}(0) \in \calW^\rms(p_1)$, i.e.,
\begin{eqnarray*}
W_2(0;\om_2,\mu,L_2) &=& \hat W_1^\infty(\alpha_1;\om_1^\rmu,\mu),\\ \hat W_{m-1}(0;\om_{m-1},\mu,L_{m-1}) &=& W_1^\infty(\alpha_1;\om_1^\rms,\mu),
\end{eqnarray*}
\item[`per':] $W_1(0;\om_1,\mu,L_1) = \hat W_m(0;\om_m,\mu,L_m)$,
\item[`semi-':] $\hat{q}_{0}(0) + \hat{W}_{0}(0) \in \calW^\rms(p_1)$, i.e.,
\begin{eqnarray*}
\hat W_{0}(0;\om_{0},\mu,L_{0}) &=& W_1^\infty(\alpha_1;\om_1^\rms,\mu),
\end{eqnarray*}
\item[`semi+':] $q_2(0) + W_2(0) \in \calW^\rmu(p_1)$, i.e.,  
\begin{eqnarray*}
W_2(0;\om_2,\mu,L_2) &=& \hat W_1^\infty(\alpha_1;\om_1^\rmu,\mu).
\end{eqnarray*}
\end{list}

In order to unify notation for these cases, we set $L_1=\infty$ for `semi$\pm$' and $L_1=L_m=\infty$ for `het' and `hom' so that all equations are of the form \eqref{e:1}. We thus omit the superscript `$\infty$' and take indices modulo $m+1$ for `per'. System \eqref{e:1} then needs to be solved for $j\in \virI^\rmo$ with the modified definition 
\[
\virI^\rmo:=\begin{cases}
\virI \mod m+1\mbox{ for `per'},\\
\virI \setminus \{1\}\mbox{ for `semi$-$', `hom' and `het'},\\
\virI\mbox{ for `semi$+$' and when }\virI=\mathbb{Z},
\end{cases}
\]
Free travel time parameters are then $L_j$ with $j\in \virI^\rmL$ where
\[
\virI^\rmL:=\begin{cases}
\virI^\rmo\mbox{ for `semi$\pm$', `per' and if }\virI^\rmo = \mathbb{Z},\\
\{2,\ldots,m-1\}\mbox{ for `hom' and `het'}.
\end{cases}
\]
The parameters of fixed points of $(\calG_j)_{j\in\virI^\rmo}$ are thus $\om_j$, $j\in \virI^\rmo$, and $L_{j}$, $j\in \virI^\rmL$, and the actual system parameters $\mu\in \R^d$. In the Lyapunov-Schmidt reduction we first use the coordinate parameters $\om_j$, and then, if needed, express the system parameters $\mu$ through the time parameters $L_j$ and possibly remaining coordinate parameters. This also determines the generic minimum number of parameters needed for the unfolding of the part of the network that is visited by the selected itinerary.

\medskip
If $\calC$ contains a sequence of adjacent periodic orbits, the requirement of equal asymptotic phase in Theorem~\ref{t:1} for in- and outflow at each of these may require different $q_{j+1}(0)$ and $\hat q_j(0)$, i.e.\ $x_j\neq 0$ in the definition of $\hat q_j$. In that case solving \eqref{e:1} requires a nontrivial shift in the flow direction. However, this direction is not directly available since we removed the flow direction from the coordinate parameters $\om_j$.

To trivialise the matching in this direction we change coordinates locally in a neighbourhood of the trajectory segments $Y_j:=\{q_j(x) : 0\leq x \leq \max\{T_j,T_{j-1}\}\}$ for all $j\in\virI_\rmP$, to obtain `flow box' coordinates so that the flow is parallel to $Y_j$ in a tube about it, see Figure~\ref{f:match}. Since $\calC^*$ is finite we can choose a uniform tube radius. Note that this change of coordinates is independent of the changes near $p_j$ performed above.

\begin{figure}%[htbp]
\begin{center}
\input{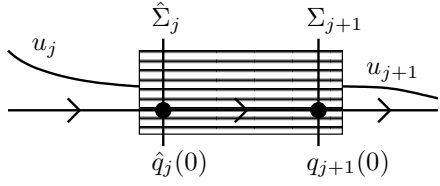}
\caption{Schematic plot of the flow box near the trajectory between $\hat q_j(0)$ and $q_{j+1}(0)$. A left orbit $u_j$ that enters the box at the same coordinate in $\hat \Sigma_j$ as a right orbit $u_{j+1}$ in $\Sigma_{j+1}$ lies on the same trajectory as $u_{j+1}$.}
\label{f:match}
\end{center}
\end{figure}

\begin{remark}\label{r:flow}
\enum
\item Since fixed points of $\calG_j$ are coordinates of trajectories (Corollary~\ref{c:1}) there is a bijection between solutions (up to time shifts) of system \eqref{e:1} with all closing conditions and solutions of~\eqref{e:ode} that stay in a certain neighbourhood of the itinerary, if we require minimal period for periodic solutions.
\item Due to the flow box coordinates near problematic $q_j(0)$, it is in fact not necessary to solve \eqref{e:1} in the flow direction $\dot q_j(0)$: Even if the orbits reconstructed from fixed point components $W_j$ of $\calG_j$ and $\hat W_{j-1}$ of $\calG_{j-1}$ do not fit together in the flow direction, a unique trajectory is selected, see also Figure~\ref{f:match}. Proof: By construction, all trajectory segments 
\[
u_j(x) = \left\{
\parbox{80mm}{$q_j(x) + W_j(x)\,,\quad x\in[0,L_j]$\\
$\hat q_j(x-2L_j) + \hat W_j(x-2L_j)\,,\quad x\in [L_j,2L_j]$,}
\right.
\]
$j\in \virI^\rmL$, are continuous at $L_j$. 
A priori we would require the jumps $u_j(2L)- u_{j+1}(0)$ to vanish. However, since the vector field in the flow box is parallel to $q_j$ the coordinates in the un/stable trichotomy spaces $E_{j+1}^{\rms/\rmu}(x)$ and $\hat E_j^{\rms/\rmu}(x)$ do not change within the flow box. Therefore, the segments $u_j$ and $u_{j+1}$ lie on the same trajectory, if and only if their $j$- and $(j+1)$-variations have same coordinates in $E_{j+1}^{\rms/\rmu}(0)$ and $\hat E_j^{\rms/\rmu}(0)$. 
\end{list}
\end{remark}

\medskip
Recall the spaces of coordinate parameters $H_j^\rms:=Q_j^\rms E_j^\rms(0)$ and $\hat H_j^\rmu:=\hat Q_j^\rmu E_j^\rmu(0)$, and that these do not contain the flow direction.

To motivate the following definitions, notice that $\om_j^\rms$ and $\hat\om_{j-1}^\rmu$ explicitly appear in \eqref{e:1} when substituting the definition of $\calG_j$ at $x=\hat x =0$. In fact, \eqref{e:1} projected onto $E_j := H_j^\rms + \hat H_{j-1}^\rmu$ can be solved using $\om_j$, and we therefore define
\[
P_j := \Proj( E_j^\perp , E_j).
\]
Lyapunov-Schmidt reduction now consists of solving the system \eqref{e:1} projected first by $P_j$ and then substituting the result into the projection by $\Id-P_j$. In this process the flow direction need not be considered as shown in Remark~\ref{r:flow}. Therefore, the directions that are unreachable by coordinate parameters are 
\[
E_j^\rmb := \left( E_j^\rms(0) + E_j^\rmu(0) \right)^\perp.
\]
Hence, $d_j := \dim(E_j^\rmb)$ is the number of reduced equations at $q_j(0)$ that need to be solved by system parameters. 

\begin{definition}
Let $\virI^\rmb\subset \virI^\rmo$ be the set of indices for which $d_j \geq 1$. 
\end{definition}

\medskip
We set $d:=\sum_{j\in \virI^\rmb\cap \virI_\rmred} d_j$ and will show that this is the number of parameters needed to unfold $\calC_\rmred$ and thus to locate the solutions selected by the choice of $\calC$. We call additional parameters \emph{auxiliary}. 

\medskip
If $H_j^\rms\cap \hat H_{j-1}^\rmu$ is non-trivial, then the representation of $\Rg P_j$ by $\om_j^\rms +\hat \om_{j-1}^\rmu$ is not unique. To make it unique, we remove the intersection from $H_j^\rms$ and denote the remaining coordinate parameters by $v_j\in H_j^\rms\cap \hat H_{j-1}^\rmu$. More precisely,
we define 
\begin{eqnarray}
\tilde P_j &:=& \Proj( [H_j^\rms\cap \hat H_{j-1}^\rmu]^\perp, H_j^\rms\cap \hat H^\rmu_{j-1} )\label{e:tildeP}\\
\tilde H_j^\rms &:=& \ker(\tilde P_j) \cap H_j^\rms,\nonumber
\end{eqnarray}
so that for $\om_j^\rms \in H_j^\rms$ there exist unique $v_j\in \Rg \tilde P_j$ and $\tilde\om_j^\rms\in \tilde H_j^\rms$ with $\om_j^\rms = v_j + \tilde\om_j^\rms$.

\begin{definition}
Let $\virI^\rmt\subset \virI^\rmo$ be the set of indices for which $\dim(\Rg(\tilde P_j))\geq 1$.
\end{definition}

\medskip
We denote the collection of all these coordinate parameters by
\[
\bar v = (v_j)_{j\in \virI^\rmt}\in \calV:=\prod_{j\in \virI^\rmt} \Rg \tilde P_j,
\] 
and endow $\calV$ with the sup-norm. Parameters $v_j$ occur if the tangent spaces of stable and unstable manifolds coincide in more than just the flow direction. A transverse heteroclinic set of two or more dimensions occurs for $j\in \virI^\rmt \setminus \virI^\rmb$, which means that the `linear' codimension 
\[
	n+1-\dim{\calW^{\rms}(p_j)} - \dim{\calW^{\rmu}(p_{j-1})} 
\]
is negative and gives the generic dimension of tangency minus the flow direction. This only uses information from the un/stable dimensions at the asymptotic states, and tangency of the manifolds may be higher dimensional, and can also occur for positive linear codimension. The above defined $d_j$ includes this by accounting for the intersection of $H_j^\rms$ and $\hat H_j^\rmu$, and is always larger than or equal to the linear codimension. We therefore refer to $d_j$ as the \emph{codimension} of $q_j$. Note that transverse heteroclinic connections occur for $j\in\virI^\rmo\setminus\virI^\rmb$ and tangent directions transverse to the flow (a tangent heteroclinic connection) occurs for $j\in \virI^\rmt\cap \virI^\rmb$, see Figure~\ref{f:tang}. 

\begin{figure}%[htbp]
\begin{center}
\input{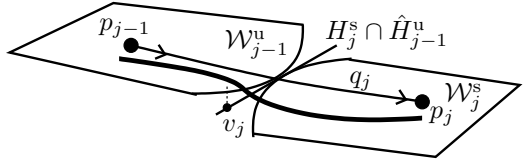}
\caption{Illustration of the notation for a tangent heteroclinic connection.}
\label{f:tang}
\end{center}
\end{figure}

\medskip
To capture the leading order effect of parameter variations on the $j$-variation we define the Melnikov-type integral maps  for $j\in \virI^\rmb$
\[
\calM_j : \R^d \to E_j\,,\; \mu \mapsto \sum_{r=1}^{d_j} \int_\R\langle \partial_{\mu} f(q_j(y);0)\mu\,,\,a_{j,r}(y)\rangle\,\rmd y\; a_{j,r}^0,
\]
where  $a_{j,r}^0\in\R^n$, $r=1,\ldots,d_j$ is a basis of $E_j^\rmb$ with reducible indexing, and $a_{j,r}(y)$ is the solution to the adjoint linear equation 
\[
\dot a = -\partial_u (f(q_j(y);0))^t a,
\]
with $a(0) = a_{j,r}^0$. On account of \eqref{e:trichest} $\calM_j$ is well-defined.

Note that auxiliary parameters $\tilde\mu$ lead to a modified map
\[
(\mu,\tilde\mu) \mapsto \calM_j\mu + \tilde\calM_j\tilde\mu \in E_j^\rmb.
\]
The complete Melnikov-map for $d$ parameters is then
\begin{equation}\label{e:mel}
\calM: \R^d \to \bar E^\rmb := \prod_{j\in \virI^\rmb} E_j^\rmb\,,\; \mu\mapsto (\calM_j \mu)_{j\in \virI^\rmb}.
\end{equation}

\begin{hypothesis}\label{h:mel}~
$\ker(\calM)=\{0\}$.
\end{hypothesis}

Repeated elements in $\calC$ mean repeated rows in $\calM$. For an equation $\calM\mu=X$ this means that the solvability conditions on the coordinates $X_j$ of $X$ in $E_j^\rmb$ are
$X_j = X_{j'}$ whenever $q_j=q_{j'}$, $j,j'\in \virI^\rmb$.
To solve the remaining part of $\calM\mu = X$ separately in each $E_j^\rmb$ as far as possible we change parameters as follows. Under Hypothesis~\ref{h:mel} the Melnikov map $\calM_\rmred: \R^d\to \prod_{j\in\virI^\rmb\cap \virI_\rmred} E_j^\rmb$ s invertible. Set $\check\mu =\calM_\rmred \mu$ and $\check f(u;\check\mu) := f(u;\calM_\rmred^{-1} \check\mu)$ and omit the `check' in the following. In the new parameters $\calM_\rmred$ is the identity on $E_j^\rmb$ in the sense that 
\[
\calM_j\mu = \sum_{r=1}^{d_j}\mu_{j_r} a^0_{j,r},
\]
for a unique subcollection $\bar\mu_j=(\mu_{j_r})_{r=1,\ldots,d_j}$ of parameters, and $\mu \cong (\bar\mu_j)_{j\in \virI_\rmred\cap \virI^\rmb}$. 

To unify notation of bifurcation equations for parameters and solvability conditions we define \emph{itinerary parameters} $\mu_j^*$ for all $j\in \virI^\rmb$ as follows. Set $\mu_j^*=\bar\mu_j$ for $j\in \virI_\rmred\cap \virI^\rmb$ and, for $j\in \virI^\rmb\setminus \virI_\rmred$, $\mu_j^*=\bar\mu_{j'}$ whenever $j'\in\virI^\rmb$ is such that $q_j=q_{j'}$.
Due to the above change of parameters, solutions to $\calM_j\mu = X$ can be cast simply as
\[
\mu_j^* = X_j, \quad j\in \virI^\rmb.
\]

\medskip
In preparation of the main theorem statement we define for $j\in \virI^\rmb\cap \virI^\rmt$ the map
\begin{eqnarray*}
\calT_j(v_j) &:=& -\sum_{r=1}^{d_j}\int_\R \langle\partial_{uu} f(q_j(y);0)(\Phi_j^*(y) v_j)^2 \,,\, a_{j,r}(y)\rangle \, \rmd y \,a_{j,r}^0,%% \\
\end{eqnarray*}
which measures the quadratic separation of the tangent manifolds by $v_j\in \Rg(\tilde P_j)$ and
\[
\Phi_j^*(y):=
\begin{cases}
\Phi_j^\rms(y), &y>0\\
\hat\Phi_{j-1}^\rmu(y), &y\leq 0.
\end{cases}
\]
On account of  \eqref{e:trichest} $\calT_j(v_j)$ is well-defined and $\calT_j(v_j) = \calO(|v_j|^2)$. 
\medskip
The following terms will capture the leading order effect of the neighbouring itinerary elements and give rise to the expansion of the bifurcation equations.
\begin{eqnarray*}
B_j^\rmu(L_j) &:=& \Phi_j^\rmu(0,L_j)P_j^\rmu(L_j)b_j(L_j),\\
B_j^\rms(L_j) &:=& \hat\Phi_j^\rms(0,-L_j)P_j^\rms(L_j)b_j(L_j).
\end{eqnarray*}

From \eqref{e:trichest}, \eqref{e:b-est-u} and \eqref{e:b-est-s} we infer
\begin{eqnarray}
|B_j^\rmu(L_j)| &=& \calO(\rme^{(\delta-\kappa_j^\rmu)L_j}R_j) = \calO(R_j^2),\label{e:est17}\\
|B_j^\rms(L_j)| &=& \calO(\rme^{(\delta-\kappa_j^\rms)L_j}\hat R_j) = \calO(\hat R_j^2).\label{e:est18}
\end{eqnarray}

The following hypothesis concerns intersections of the heteroclinic orbits and the spaces $E_j^\rmb$ as well as $E_j$ with leading un/stable fibers and trichotomy spaces, respectively, and excludes flip bifurcations.

\begin{hypothesis}\label{h:lead} Let $\nu_j^{\rms/\rmu}$ be the leading stable/unstable eigenvalues or Floquet exponents at $p_j$ with $\Im(\nu_j^{\rms/\rmu})\in[0,2\pi)$, and $\kappa_j^{\rms/\rmu}:=\Re(\nu_j^{\rms/\rmu})$. 
\begin{enumerate}
\item $\displaystyle\limsup_{x\to -\infty}\frac{\ln(\hat q_j(x))}{x} = \kappa_j^\rmu$,
\item $\displaystyle \limsup_{x\to \infty}\frac{\ln(q_{j-1}(x))}{-x} = \kappa_{j-1}^\rms$,
\item $\exists r^\rmu \in\{1,\ldots,d_j\}$ such that $\displaystyle \limsup_{x\to \infty}\frac{\ln(a_{j,r^\rmu}(x))}{-x} = \nu_j^\rmu$, 
\item $\exists r^\rms \in\{1,\ldots,d_j\}$ such that 
$\displaystyle \limsup_{x\to -\infty}\frac{\ln(a_{j,r^\rms}(x))}{x} = \nu_{j-1}^\rms$,
\item $\exists v\in E_j$ such that $\displaystyle \limsup_{x\to \infty}\frac{\ln(\Phi_j^\rmu(x,0)v)}{x} = \nu_j^\rmu$,
\item $\exists v\in E_j$ such that $\displaystyle \limsup_{x\to -\infty}\frac{\ln(\hat\Phi_{j}^\rms(x,0)v)}{-x} = \nu_{j-1}^\rms$.
\end{enumerate}
\end{hypothesis}

\medskip
To emphasise the local coupling in the itinerary and to show conjugacy to symbolic dynamical systems, see \S\ref{s:chaos}, for an arbitrary $0<\lambda<1$, we use the weighted norm 
\begin{equation}\label{e:normL}
\|\bar L\|_\calL := \sum_{j\in J^\rmL} \lambda^j |L_j|
\end{equation}
on the space $\calL(L^*):=\prod_{r\in \virI^\rmL}K_r(L^*)$ of the sequence of travel time parameters $L_j$.

In the precise formulation of the main result given next, we identify $E_j^\rmb$ with the isomorphic $\R^{d_j}$ and for $j\in\virI^\rmb$ the following denote linear maps, where $M\subset\R^{d_j\times d_j}$ are diagonal matrices.

$\beta_j': \Rg(\tilde P_{j+1})\to \R^{d_j}$, 
$\gamma_j': \Rg(\tilde P_{j-1})\to \R^{d_j}$, 

$\beta_j'': E_{j+1}^\rmu(0) \to \R^{d_j}$, 
$\gamma_j'': \hat E_{j-2}^\rms(0) \to \R^{d_j}$, 

$\eqa_j, \eqb_j \in M$,
$\eqa_j': \Rg(\tilde P_{j+1})\to M$, 
$\eqb_j': \Rg(\tilde P_{j-1})\to M$,

$\eqa_j'': E_{j+1}^\rmu(0)\to M$, 
$\eqb_j'': \hat E_{j-2}^\rms(0) \to M$.

\noindent Here $\ker(\beta_j''), \ker(\eqa_j'') \supset \ker P_{j+1}$, $\ker(\gamma_j''), \ker(\eqb_j'') \supset \ker P_{j-1}^\rmb$ so $\beta_j'' P_{j+1} = \beta_j''$ etc. 

\noindent Let $\bar{\calR} = (\calR_j)_{j\in\virI^\rmb}$, and $\Cos(\beta)= (\cos(\beta_1),\ldots,\cos(\beta_r))$ for $\beta\in\R^r$. 

\noindent Finally, $\om= (\om_j)_{j\in \virI^\rmo}$, $\tilde\om = (\tilde\om^\rms, \hat\om^\rmu) 
\in \tilde\Omega := \prod_{j\in\virI^\rmo}(\tilde H_j^\rms\times \hat H_{j-1}^\rmu)$, and $\bar E: \prod_{j\in \virI^\rmo} E_j$, where $\tilde\Omega$ and $\bar E$ are endowed with the sup-norm.

Recall that $L_j=\infty$, if $j\not\in \virI^\rmL$ and $v_j=0$ if $j\not\in\virI^\rmt$. Similarly, we make the convention that a quantity vanishes if its label is outside its range.

%------------
\begin{theorem}\label{t:2}
Under Hypotheses~\ref{h:basic} and~\ref{h:mel}, for a given itinerary $\calC$ 
with closing conditions, if required,
there exist $L^*$, $\eps_*,\eps^*>0$ depending only on $\calC_\rmred$ so that for all $\delta>0$ the following hold.

\enum
\item For all $j\in \virI^\rmb$ assume Hypothesis~\ref{h:simplestable} and~\ref{h:simpleunstable} at $p_r$ for $r=j, j\pm 1, j-2$ with leading un/stable eigenvalues $\nu_r^{\rmu/\rms} = -\kappa_r^{\rmu/\rms} + \rmi \sigma_r^{\rmu/\rms}$, respectively. Take any $\delta_{j-1}^\rms<\min\{\kappa_{j-1}^\rms,\kappa_{j-1}^\rmu,\rho_{j-1}^\rms\}$ and $\delta_j^\rmu<\min\{\kappa_j^\rmu,\kappa_j^\rms,\rho_j^\rmu\}$. For $r=j,j\pm1, j-2$ set $\delta_r=0$ if $r\in\virI^\rmb\cap\virI_\rmE$ and otherwise $\delta_r=\delta$.

There exists unique $\beta_j, \gamma_j \in\R^{d_j}$ and
$\beta_j'$, $\beta_j''$, $\gamma_j'$, $\gamma_j''$, $\eqa_j$,
$\eqb_j,$
$\eqa_j'$, 
$\eqb_j'$,
$\eqa_j''$, 
$\eqb_j''$ as above,
as well as unique $C^k$ smooth $(\mu,\tilde\om^\rms,\hat\om^\rmu,\bar\calR): B(\calV,\eps_*)\times \calL(L^*) \to B(\R^d \times \tilde\Omega \times \bar E^\rmb, \eps^*)$ such that $(\mu,\tilde\om^\rms+\bar v,\hat\om^\rmu)$ solves \eqref{e:1} for $j\in \virI^\rmo$ if and only if $\mu_j^*(\bar v, \bar L)$ satisfy \eqref{e:intro} with \eqref{e:intro2} for $j\in \virI^\rmb$. All quantities except $(\bar v, \bar L)$ have reducible indexing and
\begin{align*}
\calR_j &= \calO\left( |v_j|^3 + 
\rme^{-2\kappa_j^\rmu L_j}\left[\rme^{-\delta_j^\rmu L_j} + |v_{j+1}|(R_{j+1} + |v_{j+1}|) + R_{j+1}^3\right] \right.\\ 
&\left.+ \rme^{-2\kappa_{j-1}^\rms L_{j-1}}\left[ \rme^{-\delta_{j-1}^\rms L_{j-1}} + |v_{j-1}|(R_{j-1} + |v_{j-1}|)+ \hat R_{j-2}^3\right] 
\right),\\
\frac{\rmd}{\rmd v_j}\calR_j &= \calO(
\rme^{-(2\kappa_j^\rmu +\delta_j^\rmu)L_j} 
+ \rme^{-(2\kappa_{j-1}^\rms +\delta_{j-1}^\rms)L_{j-1}}
+ |v_j|^2).
\end{align*}

Finally, $\mathrm{Rank}(\eqa_j)\geq 1$ under Hypotheses~\ref{h:lead}(1) and (3), $\mathrm{Rank}(\eqb_j)\geq 1$ under Hypotheses~\ref{h:lead}(2) and (4). The analogous statement holds for $\eqa_{j+1}^{\rmu}$, $\eqa_{j-2}^{\rms}$.
\item For $j\in \virI^\rmo$, $\delta_j=\delta$ and $\eta_j$ as in Theorem~\ref{t:1}, and $\bar R = \sum_{i\in \virI^\rmb\cap \virI_\rmred}\hat R_{i-1} + R_i$ solutions to \eqref{e:1} with $|(\mu,\om)| \leq \eps_*$ and $|\bar L|\geq L^*$ satisfy
\begin{eqnarray*}
\|W_j(\bar v,\bar L)\|_{\eta_j,L_j}^+ &=& \calO(|v_j| + R_j + |\bar v|^2 + \bar R^2),\\
\|\hat W_j(\bar v,\bar L)\|_{\eta_j,L_j}^- &=& \calO(|v_j| + \hat R_j + |\bar v|^2  + \bar R^2).
\end{eqnarray*}
If $j\in\virI_\rmE$, then under Hypotheses~\ref{h:simplestable} and~\ref{h:simpleunstable} at $p_j$ we can take $\delta_j=0$ in $R_j$ and $\hat R_j$.
\item There exists a neighbourhood $\calU$ of $\cup_{j\in \virI^\rmo_\rmred}\{q_j(x) : x\in \R\}$ such that the set of $(\bar L,\bar v)$, $|\bar v|\leq \eps_*$, $\bar L\in \calL(L^*)$ for which there is a solution to \eqref{e:1} with $|\mu|\leq \eps_*$ is bijective to the following set of $(\mu, u)\in B(\R^d,\eps_*) \times C^0(\R,\R^n)$: $u$ solves \eqref{e:ode} with $u(0)\in\Sigma_1$ and $u(x)\in \calU$ for all $x\in \R$, and there exists $(x_j)_{j\in \virI^\rmo}\subset [0,T_u)$ with $x_{j+1}-x_j>0$ minimal such that $u(x_j)\in\Sigma_j$, where $T_u\in\R\cup \{\infty\}$ is the minimal period of $u$. 
\end{list}
\end{theorem}
%-----------
In \S\ref{s:bifana} the use of this somewhat abstract result for concrete cases is illustrated by a number of examples. See \S\ref{s:maindescr} for a discussion of Theorem~\ref{t:2}.

\medskip
The remainder of this section is devoted to the proof of Theorem~\ref{t:2}, which proceeds in the two Lyapunov-Schmidt reduction steps 1. solve \eqref{e:1} by the coordinate parameters $\om_j$, 2. solve the remaining equations except the flow direction by system parameters $\mu$. 

%--------------------------------------------------------
\subsection{Solvability by coordinate parameters}~

In $E_j$, the leading order dependence on $\mu$ will stem from
\[
\underline{\calM}_j \mu := \int_\R \tilde\Phi_j(y)\partial_\mu f(q_j(y);0)\mu \rmd y,
\]
which is well-defined due to \eqref{e:trichest} for
$
\tilde\Phi_j =
\begin{cases}
P_j \Phi_j^\rmu(0,y) &,\,y\geq 0,\\
P_j \hat\Phi_{j-1}^\rms(0,y) &,\, y\leq 0.
\end{cases}
$

\medskip
%------------
\begin{lemma}\label{l:match1}
There are $C, L^*, \eps_1, \eps_2, \eps_3 >0$ depending only on $\calC_\rmred$ and $\delta_j$ such that 
there exist unique $C^k$ smooth $(\tilde\om^\rms,\hat\om^\rmu,\bar\calR)(\bar v, \bar L, \mu)$, $(\tilde\om^\rms,\hat\om^\rmu,\bar\calR): B(\calV, \eps_1) \times \bar \calL(L^*) \times B(\R^d, \eps_2) \to \tilde\Omega \times \bar E$, 
that satisfy,  for $j\in \virI^\rmo$, 
\begin{enumerate}
\item 
$P_j\left(W_j(0;\tilde\om_j^\rms + v_j,\hat\om_j^\rmu,\mu,L_j) - 
\hat W_{j-1}(0;\tilde\om_{j-1}^\rms + v_{j-1},\hat\om_{j-1}^\rmu,\mu,L_{j-1})\right) = 0$,
\item $|\tilde\om_j^\rms| + 
|\hat\om_{j-1}^\rmu - v_j| \leq C \left( |\mu| + |v_j|^2 + R_j^2 + \hat R_{j-1}^2\right)$, 
\item 
$\tilde\om_j^\rms-\hat\om_{j-1}^\rmu = -v_j + \underline\calM_j\mu - B_j^\rmu(L_j) - B_{j-1}^\rms(L_{j-1}) +\calR_{j,1}$, where
\begin{eqnarray*}
\calR_{j,1} &=& \calO\left(R_j^3 + \hat R_{j-1}^3 + (R_j+\hat R_{j-1} + |v_j| + |\mu|)|\mu|
 + |v_j|(\hat R_{j-1}+R_j + |v_j|)\right).
\end{eqnarray*}
\item For any solution of \eqref{e:1} multiplied by $P_j$ that has $|\mu|\leq \eps_2$ and $|\om| \leq \eps_3$ there exists unique $(\bar v, \bar L)$ such that $\om = (\tilde \om^\rms(\bar v, \bar L,\mu) + \bar v, \hat\om^\rmu(\bar v, \bar L,\mu))$.
\end{enumerate}
\end{lemma}
%----------

\noindent Note that $\underline{\calM}_j$, $B_j^\rmu$, $B_{j-1}^\rms$ have reducible indexing.

\medskip
\begin{proof}
For $j\in J^\rmo$ we need to solve
\[
P_j\left(\calG_{j,1}(W_j,\hat W_j;\om_j,\mu,L_j)(0) - \calG_{j-1,2}(W_{j-1},\hat W_{j-1};\om_{j-1},\mu,L_{j-1})(0)\right)=0.
\]
Let $v_j\in \Rg(\tilde P_j)$ and $\tilde\om_j^\rms \in \tilde E_j^\rms$ for the decomposition of $\om_j^\rms = \tilde\om_j^\rms + v_j$. Reordering terms, using $\Rg(P_j) = E_j$, and that the center direction lies in the kernel of $P_j$ this equation becomes
\begin{equation}\label{e:match1}
	\tilde\om_j^\rms - \hat\om_{j-1}^\rmu 
    = -v_j + \calR_{j,2} - B_j^\rmu(L_j) - B_{j-1}^\rms(L_{j-1}),
\end{equation}
where%
\begin{eqnarray}
     \calR_{j,2} &=& 
    -P_j\int_{L_j}^0\Phi_j^{\rmu}(0,y)g_j(W_j(y),y;\mu)\dy\label{e:partA}\\
    &&
    +P_j\int_{-L_{j-1}}^0\hat\Phi_{j-1}^{\rms}(0,y)\hat 
    g_{j-1}(\hat W_{j-1}(y),y;\mu)\dy\label{e:partB}\\
    &&-P_j\Phi_j^\rmu(0,L_j)P_j^\rmu(L_j)\couplin_j(L_j)\om_j \label{e:partC}\\
    &&-P_j\Phi_j^\rmu(0,L_j)P_j^\rmu(L_j)\coupnl_j(W_j,\hat W_j;\mu,L_j)\label{e:partC2}\\
    &&
    -P_j\hat\Phi_{j-1}^{\rms}(0,-L_{j-1})P_{j-1}^{\rms}
    (L_{j-1}) \couplin_{j-1}(L_{j-1})\om_{j-1}\label{e:partD}\\
    &&    -P_j\hat\Phi_{j-1}^{\rms}(0,-L_{j-1})P_{j-1}^{\rms}
    (L_{j-1})\coupnl_{j-1}(W_{j-1},\hat W_{j-1};\mu,L_{j-1}) 
    \label{e:partE}.
\end{eqnarray}

By construction, the left hand side of \eqref{e:match1} is invertible as a map from $\tilde E_j^\rms(0)\times \hat E_{j-1}^\rmu(0)$ to $E_j$. We will next estimate the terms in $\calR_{j,2}$ and show that the right hand side is a perturbation.

Note that for $r=j-1,j$ the terms $W_r$ and $\hat W_r$ depend on $\omega_r$ as in Theorem~\ref{t:1}. In particular, the right hand side of \eqref{e:match1} depends on $\tilde\om_{j-1}^\rms$ and $v_{j-1}$ (through $\om_{j-1}$ which involves $\om_{j-1}^\rms$) as well as $\tilde\om_{j+1}^\rms$ and $v_{j+1}$ (through $\om_j$ which involves $\hat\om_j^\rmu$) so that each equation is coupled to its left and right neighbour in the itinerary.
 
\medskip
\noindent \underline{Step 1:} estimating  \eqref{e:partA} and \eqref{e:partB}.
 
We expand \eqref{e:partA} and \eqref{e:partB}, and determine the difference to $\underline\calM_j \mu$. As a shorthand we use $f_j = f(q_j;0)$. Upon expanding $g_j$ in $w$ and $\mu$ the terms $f_j$ and $\partial_u f_j W_j$ cancel; similarly for $\hat g_{j-1}$. Thus, \eqref{e:partA} and \eqref{e:partB}, respectively, equal
\begin{eqnarray}
&&\quad P_j\int_{L_j}^0 \Phi_j^\rmu(0,y)
(\partial_{uu} f_j W_j^2 + \partial_\mu f_j \mu 
) \rmd y + \calR_{j,3}, \label{e:step1a}\\
&& -P_j\int_{-L_{j-1}}^0 \hat\Phi_{j-1}^\rms(0,y)
(\partial_{uu} f_j \hat W_{j-1}^2 + \partial_\mu f_j \mu
) \rmd y
+ \calR_{j,4}, \label{e:step1b}
\end{eqnarray}
where
\begin{eqnarray*}
\calR_{j,3} &=& \calO(|\mu|^2 + |\mu|\|W_j\|^+_{\eta_j,L_j} + (\|W_j\|_{\eta_j,L_j}^+)^3),\\
\calR_{j,4} &=& \calO(|\mu|^2 + |\mu|\|\hat W_{j-1}\|^-_{\eta_{j-1},L_{j-1}} + (\|\hat W_{j-1}\|^-_{\eta_{j-1},L_{j-1}})^3).
\end{eqnarray*}
The estimates in Theorem~\ref{t:1} imply
\begin{eqnarray*}
	\calR_{j,3} &=& \calO(|\mu|^2 + |\om^\rms_j|^3 + R_j^3	+ (|\om^\rms_j| + R_j) |\mu|),\\
	\calR_{j,4} &=& \calO(|\mu|^2 + |\hat\om^\rmu_j|^3 + \hat R_{j-1}^3	+ (|\hat\om^\rms_j| + \hat R_{j-1}) |\mu|).
\end{eqnarray*}
In the following, the estimates of Theorem~\ref{t:1} will be substituted directly without mentioning.

\medskip
We next estimate the term in \eqref {e:step1a} involving $\partial_{uu}f_jW_j^2$. Since $W_j$ is a fixed point of $\calG_{j,1}$ we have
\begin{eqnarray}
W_j(y) &=&
\Phi_j^\rms(y,0)\om^\rms_j + \Phi_j^\rmu(y,L_j)P_j^\rmu(L_j)[b_j(L_j) + \calR_{j,5}] + \calR_{j,6}.\label{e:wj}
\end{eqnarray}
Looking at the definition of $\calG_{j,1}$, the remainder term $\calR_{j,6}$ consists of the integrals involving $g_j$, while $\calR_{j,5}$ contains the terms involving $\couplin_j$ as well as $\coupnl_j$. 

When estimating the order of $\calR_{j,6}$, \eqref{e:trichest} and the weight $\eta_j$ allow to ignore the integrals and $\Phi^\rmu_j(x,L_j)$ in the sense that the integrals are of order $\|g_j\|^+_{\eta_j,L_j}$. This is estimated in \eqref{e:g-est} to be of order $|\mu|+|\om_j^\rms|^2 + R_j^2$. Concerning $\calR_{j,5}$, \eqref{e:couplin-est} and \eqref{e:coupnl-est} show that $\calR_{j,5} = \calO(R_j)$. 

In order to get good estimates of the other terms in \eqref{e:wj} we decompose 
\begin{equation}\label{e:wj2}
\begin{array}{rl}
\partial_{uu} f_j (W_j-\Phi_j^\rms(y,0)\om_j^\rms + \Phi_j^\rms(y,0)\om_j^\rms)^2 = \partial_{uu} f_j (W_j-\Phi_j^\rms(y,0)\om_j^\rms)^2\\
  + 2\partial_{uu} f_j [W_j,\Phi_j^\rms(y,0)\om_j^\rms]
  - \partial_{uu} f_j (\Phi_j^\rms(y,0)\om_j^\rms)^2,
\end{array}
\end{equation}
and substitute this into \eqref{e:step1a}. Consider $X=\int_{L_j}^0 \Phi_j^\rmu(0,y)\partial_{uu} f_j (W_j-\Phi_j^\rms(y,0)\om_j^\rms)^2 \rmd y$, and substitute \eqref{e:wj} for $W_j$. Using \eqref{e:trichest}, \eqref{e:b-est-u} as well as $\calR_{j,5}= \calO(R_j)$ we find
\begin{eqnarray}
|X| &\leq& C
\int_{L_j}^0\rme^{-\kappa_j^\rmu y}|\partial_{uu}f_j|
\left(\rme^{-\kappa_j^\rmu(L_j-y)}|P_j^\rmu(L_j)| R_j + |\calR_{j,6}|\right)^2\rmd y \nonumber \\
&\leq& C(\int_{L_j}^0\rme^{-\kappa_j^\rmu(y+2L_j-2y)}\rmd yR_j^2
  + |\calR_{j,6}|^2) \nonumber \\
&\leq& C (R_j^3 + |\mu|^2 + |\om_j^\rms|^4).\label{e:x}
\end{eqnarray}
The remaining terms from \eqref{e:wj2} are directly estimated to be of the order $|\om_j^\rms|^2$ and $\|W_j\|_{\eta,L_j}^+ |\om_j^\rms|$, i.e., of order  $|\om_j^\rms|(|\mu| + |\om_j^\rms| + R_j)$. 

\medskip
In summary, including $\calR_{j,3}$, the term \eqref{e:partA} can be written as
\[
-P_j\int_{L_j}^0 \Phi_j^\rmu(0,y)\partial_\mu f_j \mu \rmd y + \calR_{j,7},
\]
where
\[
\calR_{j,7} = \calO\left(|\mu|(|\mu| + |\om_j^\rms| + R_j) + |\om_j^\rms|^2 + R_j|\om_j^\rms|  + R_j^3\right).
\]
The completely analogous computation for $\hat W_{j-1}$ and \eqref{e:partB} shows that \eqref{e:partB} equals
\[
-P_j\int_{L_{j-1}}^0 \hat\Phi_{j-1}^\rms(0,y)\partial_\mu f_j \mu \rmd y 
+\calR_{j,8},
\]
where
\[
\calR_{j,8} = \calO\left(|\mu|(|\mu| + |\hat\om_{j-1}^\rmu| + 
\hat R_{j-1}) + |\hat\om_{j-1}^\rmu|^2 + \hat R_{j-1}|\hat\om_{j-1}^\rmu| + \hat R_{j-1}^3\right).
\]
Letting the bounds in these integrals tend to infinity generates an error of order $(R_j+ \hat R_{j-1})|\mu|$. Hence, the sum of \eqref{e:partA} and \eqref{e:partB} is of the form
\begin{equation}\label{e:form1}
\underline\calM_j \mu + \calO(\calR_{j,7} + \calR_{j,8}).
\end{equation}

\medskip
\noindent \underline{Step 2:} Estimate \eqref{e:partC}-\eqref{e:partE}.

Using \eqref{e:coupnl-est} and \eqref{e:trichest} shows that \eqref{e:partC2} is of order
\begin{equation}\label{e:est13}
R_j^2\left(|\hat\om_j^\rmu|^2 + 
\hat R_j|\om_j^\rms|^2 + \hat R_jR_j^2 + \hat R_j^2\right) + R_j|\mu|,
\end{equation}
and similarly \eqref{e:coupnl-est2} implies \eqref{e:partE} is of order
\begin{equation}\label{e:est14}
\hat R_{j-1}^2\left(|\om_{j-1}^\rms|^2 + R_{j-1}^2 
+ R_{j-1}\hat R_{j-1}^2 + R_{j-1}|\hat\om_{j-1}^\rmu|^2\right) + \hat R_{j-1}|\mu|.
\end{equation}
By \eqref{e:couplin-est} it follows that \eqref{e:partC} is of order
\begin{equation}\label{e:est15}
R_j^2\left(|\hat\om_j^\rmu| + \hat R_j |\om_j^\rms|\right),
\end{equation}
and by \eqref{e:coupnl-est2} it follows that \eqref{e:partD}  is of order
\begin{equation}\label{e:est16}
\hat R_{j-1}^2\left(R_{j-1}|\hat\om_{j-1}^\rmu| + |\om_{j-1}^\rms|\right).
\end{equation}
In summary,
$\calR_{j,2} = \underline\calM_j \mu + \calO( \calR_{j,7} + \calR_{j,8} + R_j^2|\hat\om_j^\rmu| + \hat R_{j-1}^2|\om_{j-1}^\rms|)$, and in particular, using \eqref{e:est17} and \eqref{e:est18}, the right hand side of \eqref{e:match1} is of the form
\begin{equation}\label{e:match1est}
-v_j + \calO\left(|v_j|^2 + R_j^2 + \hat R_{j-1}^2 + |\mu| + |\om_j^\rms|^2 + |\hat\om_{j-1}^\rmu|^2)\right)
\end{equation}

\medskip
\noindent \underline{Step 3:} Apply the implicit function theorem. 

On account of \eqref{e:match1est} the right hand side is a perturbation of the invertible left hand side for large $L_j$ and $L_{j-1}$ and small $|v_j|$. Notably, the constants in the above order estimates all depend only on $C$ from Theorem~\ref{t:1} and are uniform in $L_j$, $L_{j-1}$.

Hence, for finite $\virI$ and any closing condition there are $\eps_1, \eps_2$ and $L^*$ such that the implicit function theorem immediately applies to the then finite system \eqref{e:match1} for $j\in \virI^\rmo$. This gives unique $C^k$ smooth solutions $\tilde\om_j^\rms, \hat \om_{j-1}^\rmu$ and, due to \eqref{e:match1est}, there is $C>0$ such that these satisfy
\begin{equation*}%\label{e:iftest}
|\tilde\om_j^\rms| + |\hat w_{j-1}^\rmu -v_j| \leq 
C\left( |\mu| + |v_j|^2 + R_j^2 + \hat R_{j-1}^2\right).
\end{equation*}
Since $\bar\calR$ is a function of $\tilde\om, \bar v, \bar L, \mu$ this proves items 1 and 2 of the lemma. Item 3 follows from the above form of $\calR_{j,2}$, and item 4 from this and the implicit function theorem.

For infinite $\virI$ the inverse of the map generating the left hand side of \eqref{e:match1},
\[
\tilde\Omega \to \prod_{j\in \virI^\rmo}(\tilde H_j^\rms + \hat H_{j-1}^\rmu), \,
(\tilde\om_j^\rms, \hat\om_{j-1}^\rmu)\mapsto \tilde\om_j^\rms - \hat\om_{j-1}^\rmu,
\] 
is given componentwise by $(\bar P_j, (\bar P_j - \Id))$ with norm measured by that of $\bar P_j := \Proj(\tilde H_j^\rms, \hat H_{j-1}^\rmu)$. Since $\calC_\rmred$ is finite, this is uniformly bounded in $j$. The estimates of the right hand side of \eqref{e:match1} and that this involves only $j$ and $j-1$ immediately gives continuity in $\bar L, \tilde\om, \bar v$ in the spaces $\calL, \tilde\Omega, \calV$. Smoothness of order $k$ follows again using that \eqref{e:match1} is local in the index $j$; smoothness in $\mu$ is straightforward. Hence, the implicit function theorem applies as in the finite case.~\hfill{}
\end{proof}

%-------------------------------------------------
\subsection{Completing the solution using system parameters} 

Upon substituting the solutions $\tilde\om_j^\rms$, $\hat\om_{j-1}^\rmu$, $j\in \virI^\rmo$, from Lemma~\ref{l:match1} into the fixed points $W_j,\hat W_j$ of Theorem~\ref{t:1} the remaining parameters are $\bar v$, $\mu$ and $\bar L$. 

As explained in Remark~\ref{r:flow}, the projection of \eqref{e:1} to the flow direction $\dot q_j(0)$ is trivially solved. Therefore, the spaces $E_j^\rmb\subset \ker P_j$ identify the remaining so-called \emph{bifurcation equations} that determine $\mu$ via \eqref{e:1} multiplied by all $a_{j,r}^0$, $j\in \virI^\rmb$, $r\in \{1,\dots,d_j\}$, which yields
\begin{equation}\label{e:2}
\sum_{r=1,\ldots,d_j}\langle (W_j - \hat W_{j-1})(0;\bar v,\bar L,\mu), a_{j,r}^0\rangle a_{j,r}^0 =0.
\end{equation}
Recall that $a_{j,r}^0$ form a basis of $E_j^\rmb$ so that each summand  has to vanish.

We will show that the boundary terms $b_j$ enter the $j$-th bifurcation equation via
\begin{eqnarray*}
\hat B_{j,r}(L) &:=& \langle B_j^\rms(L)\,,\, a_{j+1,r}^0\rangle,\\ 
B_{j,r}(L) &:=& \langle B_j^\rmu(L)\,,\, a_{j,r}^0\rangle,\\
\calB_j(L_{j-1}, L_j) &=& \sum_{r=1,\ldots,d_j}(  \hat B_{j-1,r}(L_{j-1}) + B_{j,r}(L_j)) \,a_{j,r}^0.
\end{eqnarray*}

\medskip
To capture the leading order dependence of the $j$-th bifurcation equation on the neighbouring $(j\pm 1)$-st we define (recall $\bar P_j$ from the proof of Lemma~\ref{l:match1})
\begin{eqnarray*}
G_j &:=& \bar P_{j+1}(v_{j+1} + B_{j+1}^\rmu(L_{j+1}) + B_j^\rms(L_j))\\
\hat G_{j-1} &:=& v_{j-1} + (\Id - \bar P_{j-1})(B_j^\rmu(L_j) + B_{j-1}^\rms(L_{j-1})),\\
\calS^+_{j,r} &:=& - \langle
\Phi_j^\rmu(0,L_j)P_j^\rmu(L_j)\hat\Phi_j^\rmu(-L_j,0)G_j
,a_{j,r}^0\rangle \, \rmd y, \\
\calS^-_{j,r} &:=& \langle
\hat\Phi_{j-1}^\rms(0,-L_{j-1})P_{j-1}^\rms(L_{j-1})\Phi_{j-1}^\rms(L_{j-1},0)\hat G_{j-1},a_{j,r}^0
\rangle \, \rmd y,\\
\calS_j &=& \calS_j(v_{j+1},v_{j-1},L_j, L_{j-1},L_{j+1})  :=
\sum_{r=1,\ldots,d_j} (\calS^-_{j,r} + \calS^+_{j,r})a_{j,r}^0.
\end{eqnarray*}
Note how in these terms the coordinate parameters $v_{j\pm 1}$ are transported by the linear evolution from $q_{j\pm 1}(0)$ to $q_j(0)$, respectively. 

The estimates \eqref{e:trichest}, \eqref{e:est17} and \eqref{e:est18} imply
\begin{eqnarray}
\calB_j(L_{j-1}, L_j) &=& \calO(\hat R^2_{j-1} + R^2_j),\label{e:B-est}\\
\calS^+_{j,r} &=& \calO(R^2_j(|v_{j+1}| + R_{j+1}^2 + \hat R_j^2)),\label{e:S1-est}\\
\calS^-_{j,r} &=& \calO(\hat R^2_{j-1}(|v_{j-1}| + R_j^2 + \hat R_{j-1}^2)).\label{e:S2-est}
\end{eqnarray}

In the following lemma a subtlety for the correction $\delta_j$ in the error terms $R_j$, $\hat R_j$ arises. For simple leading eigenvalues $\delta_j=0$ is possible everywhere, except in the estimates of Theorem~\ref{t:1} for $j\in\virI_\rmP$, which requires an exponential weight $\eta_j>0$. So far this was irrelevant, but now it becomes important, and therefore we let $\underline{R}_j=R_j$, $\hat{\underline{R}}_j=\hat R_j$ denote error terms where $\delta_j$ can be set to zero for simple leading eigenvalues.
%----------------------------------------
\begin{lemma}\label{l:bifeq1}
\begin{enumerate}
\item For $j\in \virI^\rmb$ and $\bar L$, $\bar v$ as in Lemma~\ref{l:match1}, 
  equation \eqref{e:2} is of the form
  \begin{eqnarray*}
    \calM_j \mu &=& \calT_j (v_j) - \calB_j(L_{j-1}, L_j)%\\ 
    + \calS_j(v_{j+1},v_{j-1},L_j, L_{j-1},L_{j+1}) %&& 
    + \calR_{j,9} + \calR_{j,10} 
\end{eqnarray*}
where the remainder terms are $C^k$ smooth and
\begin{eqnarray*}
\calR_{j,9} &=& \calO\left( |\mu|(R_j + \hat R_{j-1} + |v_j| + |\mu|) \right),\\ 
\calR_{j,10} &=& \calO\left(
|v_j|\left(|v_j|^2 + R_j^2\hat R_j+\hat R_{j-1}^2 R_{j-1}\right)
\right.\\
&&+\left. 
|v_{j+1}|R_j^2(\hat R_j + R_{j+1} + |v_{j+1}|) + |v_{j-1}|\hat R_{j-1}^2(\hat R_{j-2} + R_{j-1} + |v_{j-1}|)
\right.\\
&& \left.
+\underline{R}_j^2(R_j + \hat R_j^2 + R_{j+1}^3) +\hat{\underline{R}}_{j-1}^2(\hat R_{j-1} + R_{j-1}^2 + \hat R_{j-2}^3) 
\right).
\end{eqnarray*} 
\item Under Hypothesis~\ref{h:mel}, for sufficiently small $|\mu|$, $|\bar v|$ and large $\bar L$, there exist unique $C^k$ smooth $\calR_{j,11}$ of the same order as $\calR_{j,10}$ such that \eqref{e:2} is solved if and only if for $j\in \virI^\rmb$ the itinerary parameters $\mu_j^*$ satisfy
\[
\mu_j^* = \calT_j(v_j) - \calB_j(L_{j-1}, L_j) + \calS_j(v_{j+1},v_{j-1},L_j, L_{j-1},L_{j+1}) + \calR_{j,11}.
\]
\end{enumerate}
\end{lemma}

Here, $\calM_j,\calT_j, \calB_j, \calS_j$ have reducible indexing (they do not differ on repeated parts of the heteroclinic network in the itinerary).

\begin{proof}
Substituting the definition of $\calG_j$, and suppressing some variables for readability, the summands in \eqref{e:2} are
\begin{eqnarray}
- \langle \Phi_j^\rmu(0,L_j)P_j^\rmu(L_j)b_j + \hat \Phi_{j-1}^\rms(0,-L_{j-1})P_{j-1}^\rms(L_{j-1})b_{j-1}, a_{j,r}^0 \rangle \label{e:part1}\\
- \langle \Phi_j^\rmu(0,L_j)P_j^\rmu(L_j)\couplin_j, a_{j,r}^0 \rangle\label{e:part2} \\
- \langle \Phi_j^\rmu(0,L_j)P_j^\rmu(L_j)\coupnl_j, a_{j,r}^0 \rangle\label{e:part2a} \\
- \langle\hat \Phi_{j-1}^\rms(0,-L_{j-1})P_{j-1}^\rms(L_{j-1})\couplin_{j-1}, a_{j,r}^0 \rangle \label{e:part3}\\
- \langle\hat \Phi_{j-1}^\rms(0,-L_{j-1})P_{j-1}^\rms(L_{j-1})\coupnl_{j-1}, a_{j,r}^0 \rangle \label{e:part3a}\\
+ \langle \int_{L_j}^0 \Phi_j^\rms(0,y) g_j(W_j,y;\mu) \rmd y , a_{j,r}^0 \rangle \label{e:part4}\\ - 
\langle\int_{-L_{j-1}}^0 \hat\Phi_{j-1}^\rms(0,y)\hat g_{j-1}(\hat W_{j-1},y;\mu)\rmd y , a_{j,r}^0 \rangle \label{e:part5}
\end{eqnarray}
Note first that term \eqref{e:part1} is precisely
$-\hat B_{j-1,r}(L_{j-1}) - B_{j,r}(L_j)$.

\medskip
\noindent \underline{Step 1:} Estimate \eqref{e:part2}--\eqref{e:part3a}.

From the estimate \eqref{e:est13} of \eqref{e:partC2} and Lemma~\ref{l:match1}(2) we infer that \eqref{e:part2a}  is of the order
\[
R_j^2\left(\hat R_j^2 + |v_{j+1}|^2 + R_{j+1}^4 + \hat R_jR_j^2 +\hat R_j(|v_j|^2 + R_{j}^4 + \hat R_{j-1}^4) \right)  + R_j|\mu|.
\]
Similarly, now using \eqref{e:est14}, \eqref{e:part3a} is of theorder
\[
\hat R_{j-1}^2\left(R_{j-1}^2 + |v_{j-1}|^2 + \hat R_{j-2}^4 + R_{j-1}\hat R_{j-1}^2 + R_{j-1}(|v_j|^2 + R_j^4+ \hat R_{j-1}^4)\right) + \hat R_{j-1}|\mu|.
\]

\medskip
Substituting the expansion of $\hat\om_j^\rmu$ from Lemma~\ref{l:match1}(3) into \eqref{e:part2} and using \eqref{e:couplin-est} as well as Lemma~\ref{l:match1}(2) for $|\tilde\om_j^\rms|$ gives
\[
\calS_{j,r}^+ + \calO\left(\underline{R}_j^2\left(|\mu|+ |\calR_{j+1,1}| + \hat R_j(|v_j|+R_j^2+\hat R_{j-1}^2)\right)\right).
\]
Similarly, using  the expansion of $\tilde\om_{j-1}^\rms$, \eqref{e:part3} contains $\calS^-_{j,r}$, and, by \eqref{e:couplin-est2}, the rest is of the order
\[
\hat{\underline{R}}_{j-1}^2\left(|\mu| + |\calR_{j-1,1}| + R_{j-1}(|v_j|+R_j^2+\hat R_{j-1}^2)\right).
\]
In summary, after some computation, \eqref{e:part2}$+$\eqref{e:part2a}$+$\eqref{e:part3}$+$\eqref{e:part3a} minus $\calS^\pm_{j,r}$ is of the order
\begin{eqnarray}
|\mu|(R_j + \hat R_{j-1}) 
+ |v_j|(R_j^2\hat R_j + \hat R_{j-1}^2 R_{j-1})\nonumber\\
+ |v_{j+1}|R_j^2(\hat R_j + R_{j+1} + |v_{j+1}|) 
+ |v_{j-1}|\hat R_{j-1}^2(\hat R_{j-2} + R_{j-1} + |v_{j-1}|)\label{e:rem}\\
+ \underline{R}_j^2(\hat R_j^2 + R_j^2 + R_{j+1}^3)+ \hat{\underline{R}}_{j-1}^2(R_{j-1}^2 + \hat R_{j-1}^2 + \hat R_{j-2}^3).\nonumber
\end{eqnarray}
Note how remainder terms come from the local piece of the itinerary, but also from one and two steps further along the itinerary, if the itinerary is that long. If the itinerary is shorter and not `per' then all the terms with indices outside the range of the itinerary vanish by definition. 

\medskip
\noindent \underline{Step 2:} Expand and estimate  \eqref{e:part4} and \eqref{e:part5}. 

Similar to the proof of Lemma~\ref{l:match1} the idea is to expand \eqref{e:part4} and \eqref{e:part5}, so that the sum can be written as
\[
-\calM_j \mu + \calT_j (v_j) + \calR_{j,11}.
\]
We write $f_j = f(q_j;0)$ and expand $g_j$ and $\hat g_{j-1}$ so that \eqref{e:part4} + \eqref{e:part5} equals
\begin{eqnarray}
&& \langle \int_{L_j}^0 \Phi_j^\rmu(0,y)
(\partial_{uu} f_j W_j^2 + \partial_\mu f_j \mu 
) \rmd y , a_{j,r}^0 \rangle \label{e:fuu1}\\
&& -\langle \int_{-L_{j-1}}^0 \hat\Phi_{j-1}^\rms(0,y)
(\partial_{uu} f_j \hat W_{j-1}^2 + \partial_\mu f_j \mu
) \rmd y , a_{j,r}^0 \rangle
+\calR_{j,12},\label{e:fuu2}
\end{eqnarray}
where
\begin{eqnarray*}
\calR_{j,12} &=& \calO(|\mu|^2 + |\mu|(\|W_j\|_{\eta_j,L_j} + \|\hat W_{j-1}\|_{\eta_{j-1},L_{j-1}})\\ 
&&+ (\|W_j\|_{\eta_j,L_j})^3 
+ (\|\hat W_{j-1}\|_{\eta_{j-1},L_{j-1}})^3.
\end{eqnarray*}
Theorem~\ref{t:1} and Lemma~\ref{l:match1} imply
\begin{eqnarray*}
\|W_j(\om_j^\rms,\mu,L_j)\|_{\eta_j,L_j} &=& 
\calO(|\mu| + |v_j| + R_j + \hat R_{j-1}^2),\\
\|\hat W_{j-1}(\om_{j-1}^\rmu,\mu,L_{j-1})\|_{\eta_{j-1},L_{j-1}} &=& 
\calO(|\mu| + |v_j| + R_j^2 + \hat R_{j-1}),
\end{eqnarray*}
which yields
\[
	\calR_{j,12} = \calO(|\mu|^2 + |v_j|^3 + R_j^3 + \hat R_{j-1}^3 
	+ (R_j + \hat R_{j-1} + |v_j|) |\mu|).
\]
In \eqref{e:fuu1} we write
\begin{equation}\label{e:wj3}
\begin{array}{rl}
\partial_{uu}f_jW_j^2 &= \partial_{uu}f_j[W_j-\Phi_j^\rms(x,0)\om_j^\rms]^2 + \partial_{uu}f_j(\Phi_j^\rms(x,0)\om_j^\rms)^2\\ 
& + 2\partial_{uu}f_j[W_j-\Phi_j^\rms(x,0)\om_j^\rms,\Phi_j^\rms(x,0)\om_j^\rms],
\end{array}
\end{equation} 
and consider the resulting integrals in  \eqref{e:fuu1} from each of the three summands (I--III) of the right hand side of \eqref{e:wj3}. Since we already have an error of order of $\calR_{j,12}$, we focus on the additional contributions.

\medskip
\noindent (I) The integral over the first summand gives rise to the term $X$ from the proof of Lemma~\ref{l:match1}. Using \eqref{e:x} it does not contribute further to the order of $\calR_{j,12}$. 

\noindent (II) The integral over the second summand $\partial_{uu}f_j(\Phi_j^\rms(x,0)\om_j^\rms)^2$ reads
\begin{eqnarray*}
\int_{L_j}^0 \langle 
\partial_{uu} f_j (\Phi_j^\rms(y,0)v_j)^2, a_{j,r}(y) \rangle \rmd y + \calO(|\tilde\om_j^\rms|(|v_j| + |\tilde\om_j^\rms|)).
\end{eqnarray*}
Substituting the estimate of Lemma~\ref{l:match1} implies that the remainder term of this is of order
\[
(|\mu| + |v_j|^2 + R_j^2 + \hat R_{j-1}^2)|v_j| + |\mu|^2 + R_j^4 + \hat R_{j-1}^4 \leq C | \calR_{j,12}|.
\]

\noindent (III) Using \eqref{e:wj} and the same integral computation as for the estimate of $X$ in Lemma~\ref{l:match1}, the integral over the third summand $\partial_{uu}f_j[W_j-\Phi_j^\rms(x,0)\om_j^\rms,\Phi_j^\rms(x,0)\om_j^\rms]$ is of order $|\calR_{j,6}| |\om_j^\rms|$ with $\calR_{j,6}$ from that proof. Lemma~\ref{l:match1}(2) shows
\begin{eqnarray*}
\calR_{j,6} &=&\calO\left(|\mu|+|v_j|^2 + \hat R_{j-1}^4 + R_j^2\right) = \calO(\calR_{j,12}).
\end{eqnarray*}

The same holds for the corresponding terms in \eqref{e:part5}. Hence, going to infinite integral bounds as in the proof of Lemma~\ref{l:match1}, we  obtain that \eqref{e:part4}+\eqref{e:part5} can be written as
\[
-\calM_j\mu + \calT_j(v_j)
+\calR_{j,12}.
\]
Combining this with the remainder term \eqref{e:rem} of step 1 proves part (1) of the lemma statement. 

\medskip
\noindent \underline{Step 3:} Apply the implicit function theorem.

Recall the change of parameters in the Lyapunov-Schmidt reduction of $\calM\mu = X$ discussed after the definition of $\calM$ in \eqref{e:mel}. The itinerary parameters are such that $\calM_j|_{\mu_j}:\R^{d_j}\to E_j^\rmb$ is the identity (on the chosen basis) and generates the solvability conditions $X_j = X_{j'}$. Using part 1 of the lemma, estimates \eqref{e:B-est}, \eqref{e:S1-est} and \eqref{e:S2-est} and Hypothesis~\ref{h:mel} allow to apply the implicit function theorem immediately for finite $\virI$. 

As in Lemma~\ref{l:match1} all estimates are local in the itinerary, i.e., in the index $j$, so that smoothness in $(\bar v,\bar L)\in \calV \times \calL(L^*)$ follow. Hence, the implicit function theorem also applies for infinite $\virI$, which proves part (2) of the lemma statement.
~\hfill{}%\qed
\end{proof}

The basis to study the leading order geometry of bifurcation sets is the following. 

%%%%%%%%%%%%%%%%%%%%%%%%%%%%%%%%
\begin{lemma}\label{l:exp1} There exists $L^*\geq L^1$ such that the following holds.
\begin{enumerate}
\item Assume Hypothesis~\ref{h:simplestable} with $\nu_j = -\kappa^\rms + \rmi \sigma^\rms_j$. There exist $\hat\beta_{j,r}, \hat h_{j,r}\in \R, \beta_j^{\rms} \in \R^{n}$, $h_j^\rms\in E_j$ such that for all $L\in K_j(L^*)$ and $\delta_j^\rms < \min\{\kappa_j,\rho_j^\rms\}$
\begin{eqnarray*}
  B_j^\rms(L) &=& \rme^{-2\kappa_j^\rms L} \Cos(2\sigma^\rms_j L + \beta_{j}^\rms)h_j^\rms + \calO(\rme^{-(2\kappa_j^\rms + \delta_j^\rms)L}),\\
  \hat B_{j,r}(L) &=&   
  \rme^{-2\kappa_j^\rms L} \cos(2\sigma^\rms_j L + \hat\beta_{j,r})\hat h_{j,r} + \calO(\rme^{-(2\kappa_j^\rms + \delta_j^\rms)L}).
  \end{eqnarray*}
If Hypothesis~\ref{h:lead}(2) and~\ref{h:lead}(4) hold, then there is $r^\rms\in \{1,\ldots,d_{j}\}$ such that $\hat h_{j,r^\rms}\neq 0$. If Hypotheses~\ref{h:lead}(2) and~\ref{h:lead}(6) hold, then $h_{j}^\rms\neq 0$.
\item Assume Hypothesis~\ref{h:simpleunstable} with $\nu_j = -\kappa^\rmu \pm\rmi \sigma^\rmu_j$. There exist $\check\beta_{j,r}, \check h_{j,r} \in\R, \beta_j^{\rmu} \in \R^{n}$, $h_j^\rms\in E_j$ such that for all $L\in K_j(L^*)$ and $\delta_j^\rmu < \min\{\kappa_j,\rho_j^\rmu\}$
  \begin{eqnarray*}
  B_j^\rmu(L) &=&   
  \rme^{-2\kappa_j^\rmu L} \Cos(2\sigma^\rmu_j L + \beta_{j}^\rmu)h_j^\rmu + \calO(\rme^{-(2\kappa_j^\rmu + \delta_j^\rmu)L}),\\
    B_{j,r}(L) &=&   
  \rme^{-2\kappa_j^\rmu L} \cos(2\sigma^\rmu_j L + \check\beta_{j,r})\check h_{j,r} +
  \calO(\rme^{-(2\kappa_j^\rmu + \delta_j^\rmu)L}).
  \end{eqnarray*}
  If Hypotheses~\ref{h:lead}(1) and~\ref{h:lead}(3) hold (replacing $j-1$ by $j$), then there is $r^\rmu\in \{1,\ldots,d_{j}\}$ such that $\check h_{j-1,r^\rmu}\neq 0$. If Hypothesis~\ref{h:lead}(1) and (5) hold, then $h_{j}^\rmu\neq 0$.
\end{enumerate}
\end{lemma}
  
\begin{proof}
The expansions are a direct consequence of substituting Lemma~\ref{l:proj-est}(4) into the definitions of $\hat B_{j,r}(L)$, $B_{j,r}(L)$, $B_j^{\rms/\rmu}(L)$, using Lemma~\ref{l:proj-est}(2), and angle addition formulae.
~\hfill{}%\qed
\end{proof} 

\textbf{Proof of Theorem~\ref{t:2}.}
Concerning the expansion of $\mu_j^*$, we substitute the results of Lemma~\ref{l:exp1} into the expansion in Lemma~\ref{l:bifeq1}(2) and solve the resulting equation for $\mu$. The exponents $\delta_j^{\rms/\rmu}$ in the remainder term are not restricted further than in Lemma~\ref{l:exp1} since the relevant terms in $\calR_{j,10}$ of Lemma~\ref{l:bifeq1} are $R_j^3$ and $R_{j-1}^3$.

The terms from $\calS^\pm_j$ are, using placeholder terms $Y$, $\hat Y$, of the form
\begin{eqnarray*}
&\langle \Phi_j^\rmu(0,L_j)P_j^\rmu(L_j)\hat\Phi_j^\rmu(-L_j,0)Y, a_{j,r}^0 \rangle\\
&\langle \Phi_{j-1}^\rms(0,-L_{j-1})P_{j-1}^\rms(L_{j-1})\Phi_{j-1}^\rms(L_{j-1},0)\hat Y, a_{j,r}^0 \rangle.
\end{eqnarray*}
As noted in the proof of Lemma~\ref{l:proj-est}(4) we can replace $\Phi^\rms_{j-1}(x,0)$, $\hat \Phi^\rmu_j(x,0)$, $\hat\Phi^\rms_{j-1}(x,0)$, $\Phi^\rmu_j(x,0)$ by the evolution and projections of the variation about $p_j$ to leading order. These linear evolutions generate terms of the same form as the leading order part of the expansion of the nonlinear terms in Lemma~\ref{l:exp1}. 

The cosine terms stem from resolving $F_j$, $F_{j-1}$ in terms of the leading eigenvalues, which gives sine and cosine terms with the \emph{same} angle arguments appear with $v_{j\pm 1}$-dependent coefficients, respectively. After multiplication with $\calM_j^{-1}$ in each component all terms can be joined into one cosine term for $\nu_j^\rmu$ and one for $\nu_{j-1}^\rms$ with $v_{j\pm 1}$-dependent phase shift, respectively. This gives the diagonal nature of the linear coefficient maps. Expanding phase shift and cosine, only the linear $v_{j\pm 1}$-dependence is kept and the rest subsumed into the remainder terms. 

Finally, the order of $\frac{\rmd}{\rmd v_j}\calR_j$ follows from inspecting \eqref{e:part2}-\eqref{e:part5} as in the proof of Lemma~\ref{l:bifeq1}. 
This completes the proof of Part 1, and Part 2 follows from this in combination with the estimates of Theorem~\ref{t:1} and Lemmata~\ref{l:match1}, \ref{l:bifeq1}.

\medskip
Part 3 of this theorem is a consequence of the fact that fixed points of $\calG_j$ are coordinates of trajectories, see Corollary~\ref{c:1}, and the uniqueness of solutions in Lemma~\ref{l:match1} and \ref{l:bifeq1}.
Injectivity from $(\bar v,\bar L)$ to trajectories for fixed $\mu$, follows the expansion  in Lemma~\ref{l:match1}(3) and Corollary~\ref{c:1}; injectivity in $\mu$ follows from invertibility of $\calM_\rmred$. It remains to argue that this covers all solutions in a neighbourhood of the visited heteroclinic connections.

The assumptions on $x_j$ guarantee that the $j$-variations obtained from $u$ are coherent with $\calC$ and so small that Theorem~\ref{t:1} and Lemmas~\ref{l:match1}, \ref{l:bifeq1} apply. Since fixed points of $\calG_j$ are surjective onto such $j$-variations in a neighbourhood of each heteroclinic connection, each of these corresponds to a unique fixed point of $\calG_j$ near $p_j$. Hence, the coordinate parameters must solve \eqref{e:1} for the given $\mu$. The Lyapunov-Schmidt reduction of this system done in Lemmas~\ref{l:match1}, \ref{l:bifeq1} provides necessary and sufficient conditions on all such solutions. Therefore, the coordinates parameters $\bar v$ and travel times $\bar L$ derived from $u$ must solve part (1) of Theorem~\ref{t:2}. In particular, these solutions indeed cover all trajectories of \eqref{e:ode} that remain in a neighbourhood of the selected itinerary.
%\qed

%%%%%%%%%%%%%%%%%%%%%%%%%%%%%%%%%%%%%%%%%%%%%%%%%%
\section{Sample bifurcation analyses}\label{s:bifana}
%%%%%%%%%%%%%%%%%%%%%%%%%%%%%%%%%%%%%%%%%%%%%%%%%%

If the itinerary $\calC$ does not have repeated elements of $\calC^*_2$, then $\bar\mu = \mu^*$ so there are no solvability conditions for $(\bar v, \bar L)$ in Theorem~\ref{t:2}. Hence, it already proves existence, uniqueness, and the parameter expansions of certain periodic, homoclinic and heteroclinic orbits as follows. If the heteroclinic network allows for recurrence, e.g., if $p_1=p_m$, these solutions are \emph{simply recurrent}, or $1$-recurrent, in the sense that during their minimal (possibly infinite) period they intersect each section $\Sigma_j$ at $q_j(0)$ for all $q_j\in\calC_2$ precisely once. We call such solutions $1$-periodic, $1$-homoclinic or $1$-heteroclinic orbits.

%--------------------------------
\begin{corollary}\label{c:basic}
Assume Hypotheses~\ref{h:basic} and~\ref{h:mel}. If $\virI^\rmo = \virI_\rmred$ (in particular $\virI=\{1,\ldots,m\}$ is finite and $\virI_\rmred$ unique) then there exists $\eps>0$ and a neighbourhood $\calU$ of the itinerary such that the following holds.
\begin{enumerate}
\item If $p_1\neq p_m$ then the set of 1-heteroclinic orbits from $p_1$ to $p_m$ in $\calU$ is bijective to the non-empty set of fixed points of Lemma~\ref{l:match1} and parameters as in Lemma~\ref{l:bifeq1} under the `het' closing conditions $L_1=L_m=\infty$.
\item If $p_1=p_m$ then the same holds for the set of 1-homoclinic and 1-periodic orbits under the `hom' ($L_1=L_m=\infty$) or `per' ($L_1=L_m$) closing conditions, respectively.
\item Under the assumptions of Theorem~\ref{t:2}(1) in each case $\bar v$ and $\bar L$ parametrise the corresponding solution set, and $\bar\mu\in \R^d$ satisfy the given expansion.
\end{enumerate}
\end{corollary}

In the following we illustrate for some basic heteroclinic networks how to determine and use the equations for $\mu_j$ of Theorem~\ref{t:2}. Most of the results are well known, but we hope to show the ease in obtaining them using that result. The last class of examples are for heteroclinic cycles between one equilibrium and one periodic orbit mentioned in the introduction.

%---------------------
\subsection{Heteroclinic orbits}

The simplest heteroclinic network consists of two elements $p_1\neq p_2$ connected by one heteroclinic orbit $q_2$, which enforces the `het' closing condition. We thus investigate the existence and variation of the heteroclinic connection $q_2$ upon parameter changes. Note that $L_1=L_2=\infty$ so that the bifurcation equations only contain terms of order $|\mu_2|^\ell,|v_2|^{\ell+1}$, $\ell\geq 1$.

In this context $E_2 = H_2^\rms + \hat H_1^\rmu \subset  E_2^\rms(0) + \hat E_1^\rmu(0)$, where the inclusion is due to the flow direction, which was removed for equilibria. The number of bifurcation equations is $d_2 = \dim E_2^\rmb=n-\dim(H_2^\rms + \hat H_1^\rmu)-1$, and the number of remaining coordinate parameters $v_2$ for potential tangent directions is given by $\dim(H_2^\rms \cap \hat H_1^\rmu)$.

Typically, $\dim(E_2)$ is maximal, so that no tangencies occur and $\# \virI^\rmt$ is minimal, which implies $d_2 = n +1 - \dim(\calW^\rmu(p_1)) - \dim(\calW^\rms(p_2))$ is the codimension of the heteroclinic connection and $\# \virI^\rmt = \dim(\calW^\rmu(p_1) \cap \calW^\rms(p_2))-1$ is the dimension of the set of heteroclinic trajectories.

%---------------------
\subsubsection{Saddle-saddle connection}\label{s:hetsaddle}

Suppose both $p_1$ and $p_2$ are equilibria or periodic orbits connected in a saddle-saddle situation $\dim(\calW^\rmu(p_1)) + \dim(\calW^\rms(p_2)) = n$ for which the linear codimension is $1$, i.e., $d_2\geq 1$ and typically $d_2=1$. Since the heteroclinic connection cannot be transverse $\dim\Rg\tilde P_j$ counts tangent directions (except the flow) in $\calW^\rmu(p_1) \cap \calW^\rms(p_2))$ and typically is zero. 

In the typical case the bifurcation equation from Lemma~\ref{l:bifeq1} reads $\calM_2\mu_2 = \calO(\mu_2^2)$, $\mu_2\in\R$, and if $\calM_2\neq 0$ a heteroclinic connection exists only at $\mu_2=0$. If there is an auxiliary parameter $\tilde\mu_2$ we obtain a second contribution to the Melnikov integral and the leading order bifurcation equation
\[
\calM_2\mu_2 + \tilde\calM_2\tilde\mu_2 = 0.
\]
Hence, if $\calM_2$ and $\tilde\calM_2$ are non-zero, then heteroclinic orbits exist locally on a curve in the parameter plane.

%---------------------
\subsubsection{Tangent source-sink connection}\label{s:hetsource}

In a source-sink case, where the heteroclinic connection has non-positive linear codimension, generically $d_2 = 0$, i.e., there are no bifurcation equations for parameters. In that case, the coordinate parameters for heteroclinic orbits are given in Lemma~\ref{l:match1}. Note, that for negative linear codimension coordinate parameters $v_j$ appear.

If the heteroclinic connection is tangent, i.e., $d_2\geq 1$ we have $\virI^\rmt=\{2\}$ and for single tangent direction $d_2=1$, $v_2,\mu_2\in\R$. The quadratic function $\calT_2(v_2)$ is then scalar and can be written as $a v_2^2$ for $a\in \R$ so that the bifurcation equation reads
\[
\calM_2\mu_2 = a v_2^2 + \calO(|v_2|^3 + \mu_2^2).
\]
Hence, heteroclinic connections typically ($a\neq 0$) occur on a parabola in the $(v_2,\mu_2)$-parameter plane at leading order. 

Including an auxiliary parameter $\tilde\mu_2$ we can trace tangent heteroclinic connections in the $(\mu_2,\tilde\mu_2)$-parameter plane. Tangencies are located at roots of the derivative of the bifurcation equation with respect to $v_2$. At leading order this gives $v_2=0$ so that to leading order tangent heteroclinic orbits lie at $\calM_2\mu_2 + \tilde\calM_2\tilde\mu_2=0$.

%-----------------
\subsection{Bifurcations from homoclinic orbits}

The situation of the previous section for $p_1=p_2$ allows for more interesting solutions and the `hom' as well as `per' closing conditions. To serve readability we omit the subscript which labels the single equilibrium or periodic orbits in the following.

We consider the generic transverse case $H^\rms \cap \hat H^\rmu = \{0\}$ where no parameter $v_2$ occurs. Since $d_2=1$ in that case, all reduced index sets contain only one element, and for those we omit the labels. Hence, for any itinerary, each of the bifurcation equations reads, with $\mu, \beta, \gamma\in\R$,
\begin{equation}\label{e:homeq}
\mu = 
\rme^{-2\kappa^\rmu L_j}\cos(2\sigma^\rmu L_j + \beta)\eqa
+ \rme^{-2\kappa^\rms L_{j-1}}\cos(2\sigma^\rms L_{j-1} + \gamma)\eqb.
\end{equation}
Here we omitted the term $\calR_j$ and set $\zeta_j''=\xi_j''=0$ since these terms do not appear at leading order in the following considerations. The occurrence of parameters $L_r$ and the range of indices depends on the choice of itinerary.

We first consider an equilibrium $p$ where $\calL(L^*)$ is continuous, and then a periodic orbit $p$ where $\calL(L^*)$ is discrete (and the un/stable dimensions change).

%-----------------
\subsubsection{Equilibrium at $p$}\label{s:hombifeq}

The analysis of homoclinic orbits that do not pass by the equilibrium  $p$, i.e., of $1$-homoclinic orbits, is the same as in Section~\ref{s:hetsaddle}.

\medskip
%-----------------
\paragraph{2-homoclinic orbits} Homoclinic orbits could pass by $p$ any number of times before connecting to the un/stable manifold. Striving for illustration, we only consider $2$-homoclinic orbits that pass by $p$ once.

The itinerary $\calC$ for these orbits is as in Figure~\ref{f:2-hom} in \S\ref{s:intro} and contains three equilibria $p_j=p_1^*$, $j=1,2,3$, under the `hom' closing condition. The un/stable dimensions are all the same, respectively, so $d_2=d_3=1$. When applying Theorem~\ref{t:2} with the `hom' closing condition, one free parameter $L_2$ appears (note $L_1=L_3=\infty$), see Figure~\ref{f:2-hom}, and from \eqref{e:homeq} we obtain the system of bifurcation equations
\begin{eqnarray*}
\mu &=& \rme^{-2\kappa^\rmu L}\eqa\cos(2\sigma^\rmu L + \beta)\\
\mu &=& \rme^{-2\kappa^\rms L}\eqb\cos(2\sigma^\rms L +\gamma),
\end{eqnarray*} 
the first for $j=2$ and the second for $j=3$. Equating the right hand sides gives the solvability condition
\[
\rme^{-2\kappa^\rmu L}\eqa\cos(2\sigma^\rmu L + \beta)= \rme^{-2\kappa^\rms L}\eqb\cos(2\sigma^\rms L +\gamma).
\] 
If $\kappa^\rms\neq \kappa^\rmu$, say $\kappa^\rmu< \kappa^\rms$, the leading order equation is
\[
\rme^{-2\kappa^\rmu L}\eqa\cos(2\sigma^\rmu L + \beta)= 0,
\]
Hypothesis~\ref{h:lead} implies $\eqa\neq 0$ so that solutions at leading order exist if and only if $\sigma^\rmu \neq 0$, which is the well-known Shil'nikov saddle-focus configuration. The arising infinite sequence of solutions persists (due to transversality) under the higher order perturbation of the remainder term.

Concerning vanishing coefficients of leading terms, we outline the result mentioned in Remark~\ref{r:lead} in case $\bar \calR$ is higher order with respect to the terms in \eqref{e:homeq} (roughly speaking this is valid for small difference of leading stable and unstable rate, and large gap to the next leading rates).
We thus keep all terms from \eqref{e:homeq} and obtain
\begin{equation}\label{e:2hom}
\rme^{2L(\kappa^\rmu-\kappa^\rms)} = \frac{\eqa\cos(2\sigma^\rmu L + \beta)}{\eqb\cos(2\sigma^\rms L +\gamma)}
\end{equation}
For real leading eigenvalues a sign change of $\zeta$ for $\xi\neq 0$ implies the emergence of a solution from $L=\infty$, i.e., for sufficiently large $L$ for Theorem~\ref{t:2} to apply. As mentioned, the results in \cite{bjorndiss} provide a complete study of this situation and, in particular, do not require such restrictive spectral configurations. 

In the resonant situation $\kappa^\rms = \kappa^\rmu$ equation~\eqref{e:2hom} applies as well with left hand side equal to $1$. This yields solutions if either $\sigma^\rms=0$ or $\sigma^\rmu=0$, or else under non-resonance conditions on these and $\gamma, \beta$. Typically, infinitely many solutions persist when including the higher order terms. For real leading eigenvalues there is no solution if $\eqa\neq \eqb$.

\medskip
%-----------------
\paragraph{$1$- and $2$-periodic orbits} Periodic orbits could pass by $p$ any number of times each period, but here we only consider the cases where this number is one or two. Typically stable and unstable rates  differ, say $\kappa^\rmu<\kappa^\rms$, and the coefficient $\xi\neq 0$. For the $1$-periodic orbits the leading order equation according to \eqref{e:homeq} is
\[
\mu = \rme^{-2\kappa^\rms L}\eqb \cos(2\sigma^\rms L+\gamma).
\]
The well-known result follows that only for $\sigma^\rms\neq 0$ periodic orbits at $\mu=0$ co-exist with the homoclinic orbit, and in fact accumulate on it. 

The $2$-periodic orbits are encoded in the itinerary of the $2$-homoclinic orbits with `per' closing conditions so that indices of $L$ need to be taken mod $2 +1$ and \eqref{e:homeq} yields

\begin{eqnarray*}
\mu &=& \rme^{-2\kappa^\rmu L_2}\eqa\cos(2\sigma^\rmu L_2 + \beta) + \rme^{-2\kappa^\rms L_1}\eqb\cos(2\sigma^\rms L_1 +\gamma)\\
\mu &=& \rme^{-2\kappa^\rmu L_1}\eqa\cos(2\sigma^\rmu L_1 +\beta) + \rme^{-2\kappa^\rms L_2}\eqb\cos(2\sigma^\rms L_2 + \gamma),
\end{eqnarray*}
with the two parameters $L_2$ and $L_1$. Both equations coincide with that for $1$-periodic orbits if $L_2=L_1$.

In case $\kappa^\rms > \kappa^\rmu$ the solvability condition to leading order as $L^*\to\infty$ is 
\[
\rme^{-2\kappa^\rmu L_2}\eqa\cos(2\sigma^\rmu L_2 + \beta) 
= \rme^{-2\kappa^\rmu L_1}\eqa\cos(2\sigma^\rmu L_1 + \beta),
\]
or equivalently, for $\eqa\neq 0$,
\[
\rme^{2\kappa^\rmu (L_1-L_2)}
= \frac{\cos(2\sigma^\rmu L_1 + \beta)}{\cos(2\sigma^\rmu L_2 + \beta)},
\]
which shows that for the Shil'nikov saddle-focus with $\sigma^\rmu\neq 0$ there are infinitely many persistent solutions.

Assuming that $\bar\calR$ is higher order with respect to all terms in \eqref{e:homeq} (i.e., the spectral configuration is as mentioned in the $2$-homoclinic case) the resulting leading order solvability condition for $\sigma^\rmu=\sigma^\rms=0$ and $\eqb\neq 0$, can be written as
\begin{equation}\label{e:hom2per}
\frac
{\rme^{-2\kappa^\rms L_2} - \rme^{-2\kappa^\rms L_1}}
{\rme^{-2\kappa^\rmu L_2}-\rme^{-2\kappa^\rmu L_1}}
=\frac{\eqa}{\eqb}.
\end{equation}
We observe that a sign change of $\zeta$ leads to the period doubling bifurcation of a solution curve $L_2\sim L_1$ with either $L_2>L_1$ or vice versa. Note the analogy to the $2$-homoclinic case. 

In the resonant case $\kappa^\rms = \kappa^\rmu$, for $\sigma^{\rms/\rmu}=0$ the left hand side of \eqref{e:hom2per} is $1$ so that solutions do not exist for $\zeta\neq \xi$. 

%-------------------
\subsubsection{Periodic orbit at $p$}\label{s:chaos}

The form of the abstract bifurcation equations does not change much when $p$ is a periodic orbit, only $L$ is discrete, but the set of solutions near the homoclinic orbit to $p$ may change dramatically.

The reason is that such a homoclinic orbit is generically codimension-0 since the flow direction is the center direction which counts towards stable and unstable dimensions. Indeed, typically the complement to $\hat H^\rmu + H^\rms$ only contains the flow direction so that $d=0$. Hence, there is no parameter needed and no solvability condition. This means that under the assumptions of Lemma~\ref{l:match1} solutions for all itineraries and for any small parameter perturbation can be constructed by adjusting the coordinate parameters $(\om^\rms, \hat\om^\rmu)$ according to the expansion in that lemma; note that all constants depend only on $p$ and $q$, and in particular are uniform for all $\calC$.

Complicated dynamics typically occurs since the diffeomorphism generated by any suitable Poincar\'e map has a transverse homoclinic orbit which is one of the paradigms of chaotic dynamics \cite{palistakens}. Note that in the present setup the ambient dimension is arbitrary.

We next show how Theorem~\ref{t:2} can be used to prove conjugacy of the dynamics on the local invariant set to shift dynamics on two symbols. Let $\calY\subset\calU$ be the set of \emph{trajectories} $\{u(x): x\in\R\}$ contained entirely in the neighbourhood $\calU$ of $q$ from Theorem~\ref{t:2}(3). Take a suitably large Poincar\'e section $\Sigma_p$ transverse to the flow containing $p(0)$. For all solutions in $\calY$ with $u(0)\in\Sigma_1$ (without loss of generality), we find a unique sequence $x_s\in \R$,  $s\in\mathbb{Z}$ such that $x_0=0$, $x_s<x_{s+1}$ and $x_{s+1}-x_s$ is minimal so that $u(x_s)\in\Sigma_1\cup\Sigma_p$ for all $s$.

This defines a unique symbol sequence $(a_s)_{s\in\mathbb{Z}}$, $a_s\in \{X,Y\}$ by setting $a_s=Y$ if $u(x_s)\in\Sigma_1$, and $a_s=X$ if $u(x_s)\in\Sigma_p$. 
Since we require a minimum travel time $L^*$ from $\Sigma$ to $\hat\Sigma$ (the time from $\hat\Sigma$ to $\Sigma$ is constant) there is a well defined minimal number $j_Y(L^*)$ of $X$'s after each $Y$ in the sequence.

\begin{corollary}\label{c:homper}
Assume \eqref{e:ode} possesses a homoclinic orbit $q$ to a hyperbolic periodic orbit $p$, and suppose that $\dim(\hat H^\rmu + H^\rms) = n-1$. Then there is a number $L^*>0$ and a neighbourhood $\calU$ of $q$ such that the invariant set $\calY$ in $\calU$ is bijective to the set of sequences $\{(a_s)_{s\in\mathbb{Z}}\}$ defined above for which there are at least $j_Y(L^*)$ symbols $X$ after each $Y$.
\end{corollary}

\begin{proof}
In this case there is no system parameter $\mu$ and no coordinate parameter $v$ appears in Theorem~\ref{t:2}. Hence, there is a minimal travel time $L^*$ and a neighbourhood $\calU$ of $q$ such that the solutions for all itineraries in that neighbourhood are bijective to the sequences of travel times in $\calL(L^*)$. Since any orbit that lies in $\calU$ for all time has a unique such sequence, the entire invariant set $\calY$ in $\calU$ is bijective to the sequences in 
\[
\calL^*:= \{\calL(L^*): \calC\mbox{ is an itinerary}\}.
\]
In particular, any orbit in $\calY$ has a unique itinerary of intersections with $\Sigma_p$ and $\Sigma_1$ as defined above, i.e., the map from travel time to these symbol sequences is injective.

To prove surjectivity consider a symbol sequence $\{(a_s)_{s\in\mathbb{Z}}\}$ with at least $j_Y(L^*)$ symbols $X$ after each $Y$. We construct an itinerary that generates a solution with that sequence. For this we define a subsequence $(b_s)_{s\in B}$, $B\subset \mathbb{Z}$ of $a_s$ and then consider the itinerary generated by the sequence of $Y$'s in $b_s$. First remove $j_Y(L^*)$ symbols $X$ after each zero in $(a_s)$. If the resulting sequence is periodic, then $b_s$ is a minimally periodic subsequence, say of period $m$, and we employ the `per' closing conditions. If the resulting sequence is constant for $s\geq s_+$ and/or for $s\leq s_-$ then (these must be constant $X$'s) $b_s$ is defined as the sequence between such maximally chosen $s_-$ and/or minimally chosen $s_+$. If the resulting $B$ is bounded we employ the `hom' closing conditions, else the corresponding `semi' condition. 

Let $\calC$ be the itinerary $p_j=p_1^*$ and $q_j=q_1^*$, where $\virI$ is bijective to $\{s : b_s=X\}$ and $\min \virI=1$ if it exists. There is a unique solution obtained by Theorem~\ref{t:2} for that itinerary with $L_j=\ell_j T$ where $\ell_j \geq j_Y(L^*)$ is the number of consecutive $X$'s in $b_s$ that follow the $Y$ corresponding to $j$ in the bijection between $\virI$ and $\{s : b_s=X\}$. The above defined map from $\calY$ to these symbol sequences thus surjective.~\hfill{}
\end{proof}

\begin{corollary}
The dynamics of \eqref{e:ode} on the invariant set near a transverse homoclinic orbit to a hyperbolic periodic orbit (i.e., the trichotomy spaces satisfy $\dim(\hat H^\rmu + H^\rms) = n-1$) is conjugate to (suspended) shift dynamics on two symbols.
\end{corollary}

\begin{proof}
Let $L^*$ be as in Theorem~\ref{t:2} in this setting and $\calY$ the local invariant set. Let $(c_r)_{r\in \mathbb{Z}}$ be a bi-infinite sequence of symbols $0$ and $1$. We define the sequence $a_s$ with symbols and meaning as above: replace all $1$ by $Y$ followed by $j_Y(L^*)$ $X$'s, and replace all $0$ by $X$.

By Corollary~\ref{c:homper} there exists a unique orbit solving \eqref{e:ode} corresponding to that sequence. By construction, there is a unique sequence of time steps $x_r$, $r\in\mathbb{Z}$ such that the time evolution of the trajectory $u$ for the unique discrete times of intersection with $\Sigma_1$ and $\Sigma_p$ is precisely the shift of the index $c_r \mapsto c_{r+1}$.

Concerning continuity of this map from trajectories to symbol sequences, we consider the usual product topology on symbol sequences where cylinders are open sets, i.e., sets with some prescribed finite sequence of adjacent symbols. The norm generating this topology is given by \eqref{e:normL} when taking $L_j\in\{0,1\}$, see, e.g., \cite{Kathass}. 

In the direction from $\calY$ to symbol sequences continuity follows from the construction: convergence in $\calY$ means that initial conditions converge, which implies convergence of travel time sequences $\bar L$ in $\calL(L^*)$. By construction of the symbol sequences this implies their convergence in cylinders.

Conversely, let $a$ and $a'$ be symbol sequences so that $a\to a'$ in the cylinder topology. By construction, the travel time sequences $\bar L$ and $\bar L'$ of the corresponding solutions in $\calY$ converge in $\calL(L^*)$. Since the coordinate parameters $\tilde\om_j^\rms$ and $\hat \om_{j-1}^\rmu$ are continuous in $\bar L$ this implies that the coordinate parameters converge as in Lemma~\ref{l:match1}(2). Therefore, the solutions in $\calY$ converge as well.~\hfill{}
\end{proof}

\medskip
If the homoclinic orbit $q_1$ is tangent, e.g., $\dim(H^\rmu + H^\rms)=n-2$, the bifurcation equation for 1-homoclinic orbits is the same as for the tangent source-sink heteroclinic in Section~\ref{s:hetsource}. The dynamics near such a homoclinic tangency is more complicated than in the above case, see, e.g.,  \cite{palistakens}.

%%%%%%%%%%%%%%%%%%%%%%%%%%%%%%%%%%%%%%%%%%%%%%%%%%%%%%%
\subsection{Bifurcations from EP heteroclinic cycles}

In this final section we consider heteroclinic cycles between one equilibrium $p_1=E$ and one periodic orbit $p_2=P$ with heteroclinic connections $q_{EP}=q_2$ from $E$ to $P$ and $q_{PE}= q_1$ from $P$ to $E$. Such cycles have been recently found in a number of models, see \cite{becks,krock,myhombif,sieber} and the references therein. EP cycles are also called singular cycles, and have been studied from an ergodic theory point of view in \cite{Bam94,Labarca, LaSa97,MorPac98, PaRo93}, and further papers by these authors, looking for instance at properties of non-wandering sets. 

Generally, in an EP cycle one connection is generically transverse, while the other has a positive codimension, see \cite{myhombif}. Here we consider the following three cases:

\begin{list}{\labelitemi}{\leftmargin=5em}
\item[EP1:~] the connection from $E$ to $P$ is transverse and one-dimensional, and the connection from $P$ to $E$ is codimension-1,
\item[EP2:~] the connection from $E$ to $P$ is transverse and two-dimensional and the connections from $P$ to $E$ is codimension-2,
\item[EP1t:] the connection from $E$ to $P$ is codimension-1 and the connection from $P$ to $E$ is tangent.
\end{list}

Concerning stable and unstable dimensions at $E$ and $P$ let $i_E$ be the number of unstable dimensions at $E$ and $i_P$ the number of unstable dimensions at $P$ including the flow direction. Let $d_{EP} = d_2$ be the codimension of the heteroclinic connection from $E$ to $P$ and $d_{PE}=d_1$ that for the connection from $P$ to $E$. The three cases are then as follows: 
\begin{list}{\labelitemi}{\leftmargin=5em}
\item[EP1:~]~ $i_\rmE=i_\rmP$, $d_{PE} = 1$, $d_{EP} = 0$,
\item[EP2:~]~ $i_\rmP=i_\rmE-1$, $d_{PE} = 2$, $d_{EP} = -1$, 
\item[EP1t:]~ $i_\rmE=i_\rmP-1$, $d_{PE} = d_{EP} = 1$.
\end{list}

As in \eqref{e:homeq} for the homoclinic orbits to an equilibrium, we use this and Theorem~\ref{t:2} to obtain the form of the bifurcation equations without choosing a specific itinerary now, and omitting $\calR_j$ and $B_{j+1}^\rmu$, $B_{j-2}^\rms$ for the moment. The occurrence of parameters $v_j$, $L_j$ and the range of indices then depends on the choice of itinerary.

\begin{figure}%[htbp]
\begin{center}
\input{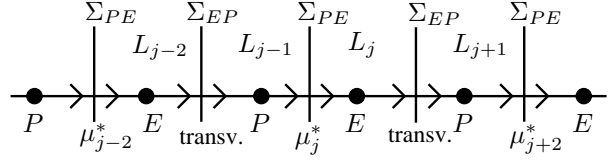}
\caption{Schematic plot of a segment of a general itinerary in an EP1 or EP2 heteroclinic cycle. Passing through $\Sigma_{EP}$ does not require a parameter $\mu_{j\pm 1}$ since the heteroclinic connection from $E$ to $P$ is transverse, but in the EP2 case a parameter $v_{j\pm1}\in\R$ appears. On the other hand, passing through $\Sigma_{PE}$ at position $j$ generates itinerary parameters $\mu_j\in\R^{d}$ for the EP$d$ case, since this connection has codimension $d=1$ or $d=2$.}
\label{f:EP1gen}
\end{center}
\end{figure}

For the EP1 case only the connection from $P$ to $E$ generates a parameter, i.e., if index $j$ corresponds to this connection then there is no bifurcation equation for indices $j\pm 1$ for any itinerary, see Figure~\ref{f:EP1gen}. Therefore, we denote the leading eigenvalues with subindices $E$ and $P$ as they always appear in the same order in the bifurcation equations which have the form
\begin{equation}
\mu = 
\rme^{-2\kappa^\rmu_E L_j}\cos(2\sigma^\rmu_E L_j + \beta)\eqa
+ \rme^{-2\kappa^\rms_P L_{j-1}}\cos(2\sigma^\rms_P L_{j-1} + \gamma)\eqb.\label{e:EP1}
\end{equation} 

For the EP2 case the connection from $E$ to $P$ is transverse and hence does not require system parameters. But the set of heteroclinic points from $E$ to $P$ is two-dimensional, i.e., a parameter $v_{j'}\in\R$ arises for $j'$ corresponding to that connection. In any itinerary the index $j=j'\pm 1$ then corresponds to the connection from $P$ to $E$, which has codimension-2 so that $\mu_j^*$ from Theorem~\ref{t:2} is two-dimensional and we write $\mu_j^* = (\mu_1,\mu_2)$ since this itinerary parameter always corresponds to the same system parameters. See Figure~\ref{f:EP1gen} for illustration. 
Note that in this case $\eqa, \eqb$, and the image of $\eqa', \eqb'$ are diagonal 2-by-2 matrices.
With subindices $E$ and $P$ as for the EP1 case, the bifurcation equations read, for $r=1,2$,
\begin{eqnarray}
\mu_r &=& 
\rme^{-2\kappa^\rmu_E L_j}\cos(2\sigma^\rmu_E L_j +\beta_{r} + \beta_{r}'v_{j+1})(\eqa_{r} + \eqa_{r}'v_{j+1})
 \nonumber\\
&&+ \rme^{-2\kappa^\rms_P L_{j-1}}\cos(2\sigma^\rmu_E L_j +\gamma_{r} + \gamma_{r}'v_{j-1})(\eqb_{r} + \eqb_{r}'v_{j-1}), \label{e:EP2}
\end{eqnarray}
where $L_j\in[L^*,\infty)$ measures the time spent between $\Sigma_{PE}$ and $\Sigma_{EP}$, while the discrete $L_{j-1}\in K_P(L^*)$ approximately measures that between $\Sigma_{EP}$ and $\Sigma_{PE}$, see Figure~\ref{f:EP1gen}.

The parameters $v_{j\pm1}\in\R$ can be viewed as varying the underlying reference heteroclinic connection from $E$ to $P$. If the heteroclinic set from $E$ to $P$ has a winding number this can be used to obtain a continuous parameter $L_{j-1}$ for the travel time near $P$, see \cite{myhombif}.

\medskip
In the EP1t case all $\mu_j^*\in\R$ are one-dimensional and the connection from $E$ to $P$ is codimension 1 while the connection from $P$ to $E$ is tangent. The tangency generates coordinate parameters $v_{j\pm1}\in\R$ and we write $\calT_2(v)= a v^2$ for $a\in \R$. Note that $v_{j\pm1}$ also appear in the bifurcation equation for the codimension-1 connection which is neighbouring this in any itinerary. Due to this coupling the resulting bifurcation equations cannot be reduced to the same basic building block form above, but to
\begin{eqnarray}
\mu_2 &=& a v_{j-1}^2 
+ \rme^{-2\kappa_E^\rmu L_{j-1}}\cos(2\sigma_E^\rmu L_{j-1} + \beta_2)\eqa_2 
+ \rme^{-2\kappa_P^\rms L_{j-2}}\cos(2\sigma_P^\rms L_{j-2} + \gamma_2)\eqb_2\nonumber\\
\mu_1 &=& \rme^{-2\kappa_P^\rmu L_j}\cos(2\sigma_P^\rmu L_j + \beta_1 + \beta_1' v_{j+1})(\eqa_1 + \eqa_1'v_{j+1})\nonumber\\ 
&&+ \rme^{-2\kappa_E^\rms L_{j-1}}\cos(2\sigma_E^\rms L_{j-1} + \gamma_1 + \gamma_1' v_{j-1})(\eqb_1+\eqb_1' v_{j-1}), \label{e:ep1t}\\
\mu_2 &=& a v_{j+1}^2 
+ \rme^{-2\kappa_E^\rmu L_{j+1}}\cos(2\sigma_E^\rmu L_{j+1} + \beta_2)\eqa_2 
+ \rme^{-2\kappa_P^\rms L_j}\cos(2\sigma_P^\rms L_j + \gamma_2)\eqb_2. \nonumber
\end{eqnarray}
Here $\mu_1$ unfolds the connection from $E$ to $P$ and $\mu_2$ from $P$ to $E$. The first and last equation form a solvability condition, if the itinerary under consideration is that long. 

\subsubsection{EP1 and EP2}
The loci of simply recurrent solutions whose itinerary has no repetitions are explicitly given for the EP1 and EP2 case by the above equations: for `hom' set either $L_j=\infty$ or $L_{j-1}=\infty$ and note that $L_{j-1}$ is discrete; for `per' any $(L_j,L_{j-1})$ for $j=1$ with $L_{j-1}$ discrete generates a solution. 

For illustration we next consider $2$-homoclinic orbits in the EP1 case; this also indicates complications arising from the terms omitted in \eqref{e:EP1}.

\begin{figure}
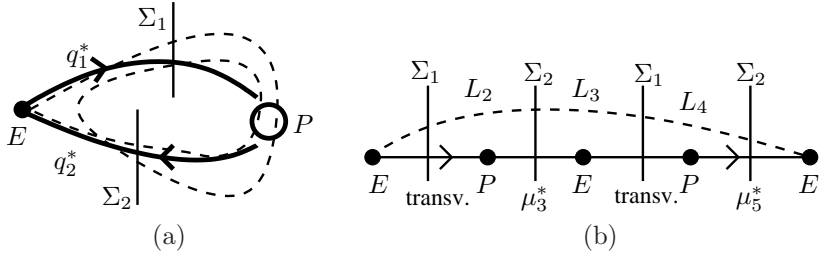
%[htbp]
\begin{center}
\begin{tabular}{cc}
\input{EP1-2-hom.pstex_t}
& \input{EP1-2-hom-b.pstex_t}\\
(a) & (b)
\end{tabular}
\caption{Sketch of an itinerary for a 2-homoclinic orbit for an EP1 heteroclinic cycle. (a) the heteroclinic cycle composed of $q_1^*$, $q_2^*$ (solid), and a 2-homoclinic orbit (dashed). (b) Schematic plot of the itinerary (solid) for a 2-homoclinic orbit (dashed) for which $L_1=L_5=\infty$. Note that the passing through $\Sigma_1$ does not generate a parameter since the heteroclinic connection from $E$ to $P$ is transverse. However, each passing through $\Sigma_2$ generates a single itinerary parameter.}
\label{f:EP1}
\end{center}
\end{figure}

$2$-homoclinic orbits to $E$ pass by $E$ once and $P$ twice so that the itinerary is as in Figure~\ref{f:EP1} and gives three parameters $L_2, L_3, L_4$ where $L_2, L_4$ are discrete. Since in this case we will get a solvability condition where the error terms have exponents from different $p_j$, they need to be treated more carefully. From~\eqref{e:EP1} the bifurcation equations including $\calR_j$ and $B_\rmE^\rms$, $B_\rmP^\rmu$ as error terms, are 
\begin{eqnarray*}
\mu &=& 
\rme^{-2\kappa^\rmu_E L_3}\cos(2\sigma^\rmu_E L_3 + \beta)\eqa
+ \rme^{-2\kappa^\rms_P L_2}\cos(2\sigma^\rms_P L_2 + \gamma)\eqb\\
&& + \calO\left(
\rme^{-(2\kappa_\rmE^\rmu + \delta_\rmE^\rmu)L_3} 
+ \rme^{-(2\kappa_\rmP^\rms + \delta_\rmP^\rms)L_2} 
+ \rme^{-2(\kappa_\rmE^\rmu L_3 + \kappa_\rmP^\rmu L_4)}\right)\\
\mu &=& \rme^{-2\kappa^\rms_P L_4}\cos(2\sigma^\rms_P L_4 + \gamma)\eqb
 + \calO\left(
 \rme^{-(2\kappa_\rmP^\rms + \delta_\rmP^\rms)L_4}
 + \rme^{-2(\kappa_\rmE^\rms L_3 + \kappa_\rmP^\rms L_2)}\right).
\end{eqnarray*}

Subtracting these equations gives the solvability condition
\begin{eqnarray*}
0 &=& \left(\rme^{-2\kappa^\rms_P L_2}\cos(2\sigma^\rms_P L_2 + \gamma)
- \rme^{-2\kappa^\rms_P L_4}\cos(2\sigma^\rms_P L_4 + \gamma)\right)\eqb\\
&+& \rme^{-2\kappa^\rmu_E L_3}\cos(2\sigma^\rmu_E L_3 + \beta)\eqa\\
&&+ \calO\left(\rme^{-(2\kappa_\rmP^\rms + \delta_\rmP^\rms)L_4}
 + \rme^{-(2\kappa_\rmP^\rms + \delta_\rmP^\rms)L_2}
 + \rme^{-(2\kappa_\rmE^\rmu + \delta_\rmE^\rmu)L_3} 
 \right) \\
&&+ \calO\left( \rme^{-2(\kappa_\rmE^\rmu L_3 + \kappa_\rmP^\rmu L_4)}
 + \rme^{-2(\kappa_\rmE^\rms L_3 + \kappa_\rmP^\rms L_2)}\right).
\end{eqnarray*}

Since $L_2,L_3,L_4$ are free parameters (for $\min\{L_2,L_3,L_4\}\geq L^*$ and within their domains) a natural starting point to find solutions are the asymptotic regimes:
\begin{enumerate}
\item (a) $L_3, L_2 \gg L_4$; (b) $L_3, L_4 \gg L_2$; (c) $L_2, L_4 \gg L_3$;
\item $L_3 \gg L_2, L_4$;
\item $L_3 \sim L_2 \sim L_4$;
\end{enumerate}
This is at first independent of the relative sizes of $\kappa_{\rmP/\rmE}^{\rms/\rmu}$; however, these are relevant when estimating the constants in the specific meaning of the `$\gg$' symbols. 

\paragraph{Case 1} For subcase (a) the asymptotic solvability condition reads
\begin{eqnarray*}
0 &=& \rme^{-2\kappa^\rms_P L_4}\cos(2\sigma^\rms_P L_4 + \gamma)\eqb
+ \calO\left(\rme^{-(2\kappa_\rmP^\rms + \delta_\rmP^\rms)L_4}\right).
\end{eqnarray*}
For $\xi\neq 0$ solutions are close to roots of the cosine-term, which, for $\sigma_P^\rms\neq 0$, are $L_4=((2m+1)\pi-\gamma)/2\sigma_P^\rms$, $m \in \N$. However, $L_4= \ell T_P/2$, $\ell\in\N$ and for generic $T_P$ the previous fails for any $m, \ell$, so that, also for $\sigma_P^\rms =0$, typically there is no solution to the asymptotic equation.  In the resonant case $T_P = ((2m_*+1)\pi-\gamma)/\ell_*\sigma_P^\rms$,  and with $\ell = m\ell_*$ in $L_4$, the full solvability condition reads
\begin{eqnarray*}
0 &=& \rme^{-2\kappa^\rms_P L_2}\cos(2\sigma^\rms_P L_2 + \gamma)\eqb + \rme^{-2\kappa^\rmu_E L_3}\cos(2\sigma^\rmu_E L_3 + \beta)\eqa\\
&&+ \calO\left( \rme^{-2(\kappa_E^\rmu L_3 + \kappa_P^\rmu L_4)}
 + \rme^{-2(\kappa_E^\rms L_3 + \kappa_P^\rms L_2)}\right)\\
&&+ \calO\left(\rme^{-(2\kappa_P^\rms + \delta_P^\rms)L_4}
 + \rme^{-(2\kappa_P^\rms + \delta_P^\rms)L_2}
 + \rme^{-(2\kappa_E^\rmu + \delta_E^\rmu)L_3} 
 \right).
\end{eqnarray*}

This can be solved under the condition $L_3, L_2 \gg L_4$ if the cosine term with the continuous $L_3$ is leading order. If the difference in `$\gg$' is not too big, this is possible if $\eqa\neq 0$, $\kappa_\rmP^\rms + \delta_\rmP^\rms/2 > \kappa_\rmE^\rmu$, and requires $\ell=m_2\ell_*$, unless $\kappa_P^\rms>\kappa^\rmu_E$ or $L_2\gg L_3$. Looking at the remainder term, $L_3\gg L_2$ is possible if $\kappa_\rmE^\rms + \kappa_\rmP^\rms> \kappa_\rmE^\rmu$.
In all these cases, if $\sigma_\rmE^\rmu\neq 0$ and if $L_2, L_4$ are sufficiently large, the implicit function theorem yields a sequence of solutions in $L_3$. 

Subcase (b) is analogous, but in subcase (c) $L_3$ is a continuous parameter so that, if $\zeta\neq 0$ and $\sigma_\rmE^\rmu\neq 0$, there exists an infinite sequence of robust solutions to the asymptotic equation without the resonance assumption and the constraint on the spectral gaps.

\paragraph{Case 2} The asymptotic solvability condition reads
\begin{eqnarray*}
0 &=& \left(\rme^{-2\kappa^\rms_P L_2}\cos(2\sigma^\rms_P L_2 + \gamma)
- \rme^{-2\kappa^\rms_P L_4}\cos(2\sigma^\rms_P L_4 + \gamma)\right)\eqb\\
&&+ \calO\left(\rme^{-(2\kappa_P^\rms + \delta_P^\rms)L_4}
 + \rme^{-(2\kappa_P^\rms + \delta_P^\rms)L_2} \right).
\end{eqnarray*}
If $\xi\neq 0$ and in the non-resonant case, the sum of the cosine-term vanishes if and only if $L_2=L_4$. This means that the corresponding orbit revolves about $P$ the same number of times each passing. However,  discreteness implies that $L_2, L_4$ cannot be adjusted to compensate the error terms. Hence, we look at the full solvability condition with $L_2=L_4$.
Since $L_3\gg L_4=L_2$ this provides a solvable condition if $\zeta\neq 0$ and if the cosine term in $L_3$ is leading order. Similar to Case 1(a),(b), this can be achieved if $\kappa_E^\rmu< \min\{\kappa_\rmP^\rms + \delta_\rmP^\rms/2,  \kappa_E^\rms + \kappa_P^\rms\}$, and the result is as in Case 1(c). Hence, if under these conditions a long time is spent near $E$, then the number of rotations about $P$ are locked: they must be the same at each passing. 

\paragraph{Case 3} If $\kappa_E^\rmu < \kappa_P^\rms$ the situation is as in Case 1(c), and if $\kappa_E^\rmu > \kappa_P^\rms$, typically as in Case 2, and under resonance analogous to Case 1(a).

%-------------------
\subsubsection{EP1t}
For simply recurrent solutions in the EP1t case, the loci of parameters are given explicitly in Corollary~\ref{c:basic}, but one is also interested in the location of turning points in the parameter curves and folds of solutions. In \cite{krock} EP1t cycles are studied in $\R^3$ using a not entirely rigorous, but geometrically intuitive approach to obtain bifurcation equations for simply recurrent solutions. Note that in $\R^3$ either $\sigma_P^{\rms/\rmu}=0$ (positive Floquet multipliers) or $\sigma_P^{\rms/\rmu} = \pi/T_P$ (negative ones) and without loss of generality one-dimensional unstable manifold so that $\sigma_E^\rmu=0$. It turns out that the present rigorous approach confirms the results of \cite{krock} in arbitrary ambient dimensions. 

Here we briefly illustrate the location of turning points and the bifurcation set for $1$-homoclinic orbits.

\medskip
\paragraph{$1$-homoclinic orbits to $E$}
This case yields the bifurcation equations
\begin{eqnarray}
\mu_1 &=& \rme^{-2\kappa_P^\rmu L}\cos(2\sigma_P^\rmu L + \beta_1 + \beta_1'v)(\eqa_1 + \eqa_1'v)\nonumber\\ 
\mu_2 &=& a v^2 
+ \rme^{-2\kappa_P^\rms L}\cos(2\sigma_P^\rms L + \gamma_2)\eqb_2.   \nonumber
\end{eqnarray}
where $v$ is continuous with $|v|< \eps$ and $L = \ell T_P/2$ ($T_P$ the period of $P$) for $\ell$ sufficiently large counting the number of oscillations about $P$. Up to an error of order $v^2$, we can write the equation for $\mu_1$ as
\[
\mu_1 = \rme^{-2\kappa_P^\rmu L}(\cos(2\sigma_P^\rmu L + \beta_1)\eqa_1 + (\cos(2\sigma_P^\rmu L + \beta_1)\eqa_1' - \sin(2\sigma_P^\rmu L + \beta_1)\eqa_1\beta_1')v).\nonumber \\
\]
Solving for $v$ and substituting the result into the equation for $\mu_2$ gives
\[
\mu_2 = a\left(\frac{\rme^{2\kappa_P^\rmu L}\mu_1 - \cos(2\sigma_P^\rmu L + \beta_1)\eqa_1}
{\cos(2\sigma_P^\rmu L + \beta_1)\eqa'_1 - \sin(2\sigma_P^\rmu L + \beta_1)\eqa_1\beta_1'}\right)^2 + 
\rme^{-2\kappa_P^\rms L}\cos(2\sigma_P^\rms L + \gamma_2)\eqb_2,
\]
note that the denominator typically never vanishes if $\eqa_1'\neq 0$ or $\eqa_1\beta'_1 \neq 0$ because $L=\ell T_P/2$ is constrained to a equi-distance discrete sequence.

Solving $\partial_{\mu_1}\mu_2=0$ gives the turning points of solution curves
\[
\mu_1=\mu_* := \rme^{-2\kappa_P^\rmu L}\cos(2\sigma_P^\rmu L + \beta_1)\eqa_1.
\]

Hence, the solution set typically is the union of parabolas with critical points at $\mu_1=\mu_*$. Depending on $\sigma_P^{\rmu}$ and $\sigma_P^\rms$ the critical points lie on the discrete evaluation of either a monotone or a `snaking' curve in the $\mu_1$ and $\mu_2$ direction, respectively, which can generate a spiraling sequence in the parameter plane.

\paragraph{$1$-homoclinic orbits to $P$} The bifurcation equations become
\begin{eqnarray}
\mu_1 &=& \rme^{-2\kappa_E^\rms L}\cos(2\sigma_E^\rms L + \gamma_1Ê+ \gamma_1'v)(\eqb_1+\eqb_1' v) \nonumber\\
\mu_2 &=& a v^2 
+ \rme^{-2\kappa_E^\rmu L}\cos(2\sigma_E^\rmu L + \beta_2)\eqa_2,\nonumber
\end{eqnarray}
where $L \geq L^*$ and $v$, $|v|< \eps$ are both continuous so that solutions typically come in two-dimensional sheets connected by folds or corners.

Eliminating $v$ as in the $E$-homoclinic case gives (for non-resonant $L$)
\[
\mu_2 = a\left(\frac{\rme^{2\kappa_E^\rms L}\mu_1 - \cos(2\sigma_E^\rms L + \gamma_1)\eqb_1}
{\cos(2\sigma_E^\rms L + \gamma_1)\eqb'_1 - \sin(2\sigma_E^\rms L + \gamma_1)\eqb_1\gamma_1'}\right)^2 + 
\rme^{-2\kappa_E^\rmu L}\cos(2\sigma_E^\rmu L + \beta_2)\eqa_2.
\]
Solving $\partial_{\mu_1}\mu_2=0$ yields the turning points
\[
\mu_1=\mu_* := \rme^{-2\kappa_E^\rms L}\cos(2\sigma_E^\rms L + \gamma_1)\eqb_1,
\]
which give the location of a fold curve of the solution sheet in the parameter plane.

For each non-resonant sequence of $L$-values the parameter locations are analogous to those of $1$-homoclinic orbits to $E$.  If $\sigma_E^\rms = \sigma_E^\rmu =0$ there are no resonances and the monotone $\mu_*(L)$ provides all folds.
Otherwise, for each resonant value of $L$, note that $\mu_2$ is quadratic in $v$ with critical point at $v=0$, which means $\mu_1=\mu_*$. Hence, to leading order, fold curves are given on the one hand via $\mu_*(L)$, and on the other hand via $\mu_2(v)$ at $\mu_1=\mu_*(L)$ for resonant $L$.
A more detailed description of the solution set for $n=3$ is given in \cite{krock}.


\begin{thebibliography}{00}
\bibitem{Bam94} R.~Bam\'on, R.~Labarca, R.~Ma\~n\'e, and M.J.~Pac\'ifico.
The explosion of singular cycles. 
Inst. Hautes \'Etudes Sci. Publ. Math., 78 (1994), pp. 207--232.

\bibitem{becks} M. Beck, J. Knobloch, D.J.B. Lloyd, B. Sandstede, T. Wagenknecht.
Snakes, ladders, and isolas of localised patterns. SIAM Journal on Mathematical Analysis 41 (2009) 936-972.
%
\bibitem{bykov} V.V. Bykov. The bifurcations of separatrix contours and chaos. Physica D 62 (1993) 290--299.
%
\bibitem{krock} A.R. Champneys, V. Kirk, E. Knoblock, B. Oldeman, J.D.M. Rademacher. Unfolding a tangent equilibrium-to-periodic heteroclinic cycle. SIAM J. Appl. Dyn. Sys. 8 (2009) 1261--1304.
%
\bibitem{chow1} S.-N. Chow, B. Deng, D. Terman. The bifurcation of a homoclinic orbit from two heteroclinic orbits, SIAM J. Math. Anal. 21 (1990) 179--204.
%
\bibitem{chow} S.-N. Chow, B. Deng, B. Fiedler. Homoclinic bifurcation with resonant eigenvalues. J. Dyn. Diff. Eqs. 2 (1990) 177--244
%
\bibitem{CodLev} E.A. Coddington, N. Levinson, Theory of
ordinary differential equations, New York, Toronto, London:
McGill-Hill Book Company, Inc. XII, 1955.
%
\bibitem{conley} C. Conley, On travelling wave solutions of nonlinear diffusion equations, Dynamical Systems Theory
and Applications (J. Moser, ed.), Lecture Notes in Physics, Vol. 38, Springer-Verlag, Berlin, 1975.
%
\bibitem{deng} Bo Deng. The bifurcations of countable connections from a twisted heteroclinic loop. 
SIAM J. Math. Anal. 22 (1991) 653--679.
%
\bibitem{denghom} Bo Deng. The transverse homoclinic dynamics and their bifurcations at nonhyperbolic fixed points. Trans. Am. Math. Soc. 331  (1992) 15--53.
%
\bibitem{Diaz} L.J. Diaz, J. Rocha, Heterodimensional cycles,
partial hyperbolicity and limit dynamics, Fundam. Math. 174 (2002)
127--186.
%
\bibitem{glenSpa} P. Glendinning, C. Sparrow, T-points: a
codimension two heteroclinic bifurcation, J. Stat. Phys. 43  (1986)
479--488.
%
\bibitem{HaleLin} J.K. Hale, X.-B. Lin, Heteroclinic orbits
for retarded functional differential equations, J. Diff.
Equations 65  (1986) 175--202.
%
\bibitem{homkraus} A.-J. Homburg, B. Krauskopf.
Resonant homoclinic flip bifurcations.
J. Dyn. Differ. Equations 12 (2000) 807--850.
%
\bibitem{Kathass} A. Katok and B. Hasselblatt. Introduction to the Modern Theory of Dynamical Systems, with a supplementary chapter by A. Katok and L. Mendoza, Encyclopedia of Mathematics and its Applications, Vol. 54, Cambridge Univ. Press, Cambridge (1995). 
%
\bibitem{knobtang} J. Knobloch. Bifurcation of degenerate homoclinic orbits in reversible and conservative systems. J. Dyn. Differ. Equations 9, (1997) 427-444.
%
\bibitem{knobdiscr} J. Knobloch. Lin's method for discrete dynamical systems. 
J. Difference Equ. Appl. 6 (2000) 577-623.
%
\bibitem{kokubu} Hiroshi Kokubu. Homoclinic and heteroclinic bifurcations of vector fields. 
Japan J. Appl. Math. 5 (1988) 455--501. %
\bibitem{KrausOlde} B. Krauskopf and B.E. Oldeman, Bifurcations of global reinjection orbits near a saddle-node Hopf bifurcation, Nonlinearity 19 (2006) 2149-2167
%
\bibitem{flip} M. Kisaka, H. Kokubu, H. Oka. Bifurcations to N -homoclinic orbits and N -periodic 
orbits in vector Þelds, J. Dyn. Di?. Eqns. 5 (1993) 305Ð357. 
%
\bibitem{KrausRiess} B. Krauskopf, T. Riess. A Lin's method approach to finding and continuing heteroclinic connections involving periodic orbits. Nonlinearity 21(2008) 1655--1690.
%
\bibitem{bar} J. Krishnan,  I.G. Kevrekidis, M. Or-Guil,
M.G. Zimmermann, M. B\"ar. {\it Numerical bifurcation and stability
analysis of solitary pulses in an excitable reaction-diffusion
medium}. Comp. Methods Appl. Mech. Engrg. {170} (1999) 253--275.
%
\bibitem{Labarca} R. Labarca, Bifurcation of contracting
singular cycles, Ann. Sci. \'{E}cole Norm. Sup. (4), 28 (1995)
705--745.

\bibitem{LaSa97} R.~Labarca and B.~San Mart\'in, Prevalence of hyperbolicity for complex singular cycles, 
Bol. Soc. Brasil. Mat. (N.S.), 28 (1997), pp.~343--362.

%
\bibitem{langer} R. Langer. Existence of homoclinic travelling wave solutions to the FitzHugh-Nagumo equations. Ph.D. Thesis, Northeastern Univ., 1980.
%
\bibitem{Lin90} X.-B. Lin. Using Melnikov's method to
solve Silnikov's problems, Proc. R. Soc. Edinb., Sect. A
116 (1990) 295--325.
%
\bibitem{MorPac98} C.A.~Morales and M.J.~Pac\'ifico. Degenerated singular 
cycles of inclination-flip type,
Ann. Sci. \'Ecole Norm. Sup., 31 (1998), pp.~1--16. 

\bibitem{nii} S. Nii. N-homoclinic bifurcations for homoclinic orbit changing their twisting, J. Dyn. Diff. Eq. 8 (1996) 549Ð572.
%
\bibitem{PaRo93} M.J.~Pa\'ifico and A.~Rovella, Unfolding contracting singular cycles, Ann. Sci. \'Ecole 
Norm. Sup., 26 (1993), pp.~691--700.

\bibitem{palistakens} J. Palis, F. Takens. Hyperbolicity and sensitive chaotic dynamics at homoclinic bifurcations. Cambridge University Press 1993.
%
\bibitem{knobport} J. Porter, E. Knobloch, New type of complex dynamics in the 1:2 spatial resonance, Physica D 159 (2001) 125--154.
%
\bibitem{myhombif} J.D.M. Rademacher. Homoclinic orbits near heteroclinic cycles with one equilibrium and one periodic orbit. J. Diff. Eq. 218 (2005) 390--443.
%
\bibitem{Riess} T. Riess. A LinÕs method approach to Heteroclinic 
Connections involving Periodic Orbits - 
Analysis and Numerics. Doctoral thesis, University of Illmenau, 2008.
%
\bibitem{bjorndiss} B. Sandstede. Verzweigungstheorie
homokliner Verdopplungen, Doctoral thesis, University of Stuttgart, 1993.
%
\bibitem{sanTAMS} B Sandstede.
Stability of multiple-pulse solutions.
Transactions of the American Mathematical Society 350 (1998) 429-472.
%
\bibitem{ssglue} B. Sandstede, A. Scheel, Gluing unstable
fronts and backs together can produce stable pulses, Nonlinearity
13 (2000) 1465--1482.
%
\bibitem{Shilnikov}  Shil'nikov, L. P. Some cases of generation of periodic motions in an n-dimensional space. Soviet Math. Dokl. 3 (1962) 394--397.
 
Shil'nikov, L. P. A case of the existence of a countable number of periodic motions. Soviet Math. Dokl. 6 (1965) 163--166.
 
Shil'nikov, L. P. On the generation of a periodic motion from a trajectory which leaves and re-enters a saddle state of equilibrium. Soviet Math. Dokl. 7 (1966) 1155--1158.
 
Shil'nikov, L. P. The existence of a denumerable set of periodic motions in four-dimensional space in an extended neighborhood of a saddle-focus. Soviet Math. Dokl. 8  (1967) 54--57.
 
Shil'nikov, L. P. On the generation of a periodic motion from trajectories doubly asymptotic to an equilibrium state of saddle type. Math. USSR Sbornik 6 (1968) 427--437.
 
\bibitem{sieber} J. Sieber, Numerical Bifurcation Analysis
for Multisection Semiconductor Lasers, SIAM J. Appl. Dyn. Sys.
1 (2002) 248--270.

\bibitem{simpson} D. Simpson, V. Kirk, J. Sneyd. Complex
oscillations and waves of calcium in pancreatic acinar cells,
Physica D 200 (2005) 303--324.
%
\bibitem{vanfie} A. Vanderbauwhede, B. Fiedler. Homoclinic period blow-up in reversible and conservative systems. Z. Angew. Math. Phys. 43 (1992)  292--318. 
%
\bibitem{zhu1} Deming Zhu, Zhihong Xia. Bifurcations of heteroclinic loops. Science in China (Series A) 41 (1998) 837--848.
%
\bibitem{zhu} Deming Zhu. Exponential trichotomy and heteroclinic bifurcations. 
Nonlinear Anal. 28 (1997) 547--557.
%
\end{thebibliography}
\end{document}